\pdfoutput=1
\documentclass[a4paper,10pt,reqno]{amsart} 

\usepackage{dsfont,amsmath,amsfonts,amscd,amssymb,graphicx,mathrsfs}
\usepackage{setspace} 
\usepackage{tikz}
\usepackage{tikz-cd}
\usetikzlibrary{calc}
\usepackage[T1]{fontenc}
    \usepackage{ae}
    \usepackage{aecompl}

\usepackage{epstopdf}
 \usepackage{hyperref}

\usepackage{amsthm}
\usepackage{enumerate}
\usepackage[all,cmtip]{xy}
\usepackage[T1]{fontenc}
\usepackage{mathtools}
\usepackage{stmaryrd}
\usepackage{wasysym}
\usepackage{youngtab}
\usepackage{ytableau}
\usepackage{tcolorbox}
\usepackage{enumitem}
\usepackage{indentfirst}

\newtheoremstyle{plainx}
  {3pt} % Space above
  {3pt} % Space below
  {\itshape} % Body font
  {} % Indent amount
  {\bfseries} % Theorem head font
  {.} % Punctuation after theorem head
  {.5em} % Space after theorem head
  {} % Theorem head spec (can be left empty, meaning `normal')

\newtheoremstyle{definitionx}
  {3pt} % Space above
  {3pt} % Space below
  {} % Body font
  {} % Indent amount
  {\bfseries} % Theorem head font
  {.} % Punctuation after theorem head
  {.5em} % Space after theorem head
  {} % Theorem head spec (can be left empty, meaning `normal')

\theoremstyle{plainx}
\newtheorem{theorem}{Theorem}[section]
\newtheorem{lemma}[theorem]{Lemma}
\newtheorem{corollary}[theorem]{Corollary}
\newtheorem{proposition}[theorem]{Proposition}

\theoremstyle{definitionx}
\newtheorem{definition}[theorem]{Definition}
\newtheorem{example}[theorem]{Example}

\newtheorem{remark}[theorem]{Remark}
\newtheorem*{notat}{Notation}

\DeclareSymbolFont{TXlettersA}{U}{txmia}{m}{it}
\SetSymbolFont{TXlettersA}{bold}{U}{txmia}{bx}{it}
\DeclareFontSubstitution{U}{txmia}{m}{it}
\DeclareMathSymbol{\varpitx}{\mathord}{TXlettersA}{36}

\def \C {\mathbb{C}}
\def \Q {\mathbb{Q}}
\def \Z {\mathbb{Z}}
\def \N {\mathbb{N}}

\def \g {\mathfrak{g}}

\def \h {\mathfrak{h}}

\def \O {\mathcal{O}}

\DeclareMathOperator{\Ima}{Im}

\DeclareMathOperator{\End}{End}

\DeclareMathOperator{\Spec}{Spec}

\DeclareMathOperator{\pr}{pr}

\DeclareMathOperator{\Aut}{Aut}

\DeclareMathOperator{\Irrep}{Irrep}

\DeclareMathOperator{\Res}{Res}

\DeclareMathOperator{\ch}{ch}
\DeclareMathOperator{\mmod}{mod}
\DeclareMathOperator{\Proj}{Proj}
\DeclareMathOperator{\Hilb}{Hil}
\DeclareMathOperator{\Lie}{Lie}
\DeclareMathOperator{\MaxSpec}{MaxSpec}
\DeclareMathOperator{\Hom}{Hom}

\DeclareMathOperator{\id}{id}
\DeclareMathOperator{\Id}{Id}
\DeclareMathOperator{\GL}{GL}
\DeclareMathOperator{\Rep}{Rep}
\DeclareMathOperator{\diag}{diag}
\DeclareMathOperator{\Eig}{Eig}

\newcommand{\Eeig}{\underline{\mathsf{Eig}}}
\newcommand{\Eeigg}{\mathsf{Eig}}

\newcommand{\triv}{\mathsf{triv}}

\newcommand{\Bas}{\mathsf{Bas}}
\newcommand{\Cell}{\mathsf{Cell}}
\newcommand{\EG}{\mathsf{EG}}
\newcommand{\quot}{\underline{\mathsf{Quot}}}
\newcommand{\core}{\mathsf{Core}}
\newcommand{\arm}{\mathsf{arm}}
\newcommand{\leg}{\mathsf{leg}}
\newcommand{\type}{\mathsf{type}}
\newcommand\scalemath[2]{\scalebox{#1}{\mbox{\ensuremath{\displaystyle #2}}}}
\newcommand{\arxiv}[1]{\href{http://arxiv.org/abs/#1}{\tt arXiv:\nolinkurl{#1}}}

\begin{document}
\title[$\C^*$-fixed points in Calogero-Moser spaces and Hilbert schemes]{The combinatorics of $\C^*$-fixed points in generalized Calogero-Moser spaces and Hilbert schemes}
\author{Tomasz Przezdziecki}
\date{} 

\begin{abstract} 
In this paper we study the combinatorial consequences of the relationship between rational Cherednik algebras of type $G(l,1,n)$, cyclic quiver varieties and Hilbert schemes. We classify and explicitly construct $\C^*$-fixed points in cyclic quiver varieties and calculate the corresponding characters of tautological bundles. Furthermore, we give a combinatorial description of the bijections between $\C^*$-fixed points induced by the Etingof-Ginzburg isomorphism and Nakajima reflection functors. 
We apply our results to obtain a new proof as well as a generalization of the $q$-hook formula. 
\end{abstract}

\maketitle

\tableofcontents

\section{Introduction}

Rational Cherednik algebras were introduced by Etingof and Ginzburg \cite{EG}. They depend on a complex reflection group as well as two parameters $\mathbf{h}$ and $t$. In the case of Weyl groups they coincide with certain degenerations of double affine Hecke algebras introduced earlier by Cherednik \cite{Che}. When $t=0$, rational Cherednik algebras have large centres and the corresponding affine varieties are known as generalized Calogero-Moser spaces. 
In this paper we consider cyclotomic rational Cherednik algebras  at $t=0$, i.e., the rational Cherednik algebras $\mathbb{H}_{\mathbf{h}}:=\mathbb{H}_{t=0,\mathbf{h}}(\Gamma_n)$ associated to complex reflection groups $\Gamma_n:=(\Z/l\Z)~\wr~S_n$ of type $G(l,1,n)$. 
The associated generalized Calogero-Moser spaces are especially interesting because they admit a realization as Nakajima quiver varieties. This fact was exploited by Gordon \cite{gor-qv}, who established a connection between generalized Calogero-Moser spaces and Hilbert schemes of points in the plane. 

Let us recall the framework of \cite{gor-qv} in more detail. 
Assume that the variety $\mathcal{Y}_{\mathbf{h}} := \Spec Z(\mathbb{H}_{\mathbf{h}})$ is smooth. 
Etingof and Ginzburg showed in \cite{EG} that $\mathcal{Y}_{\mathbf{h}}$ is isomorphic to a~cyclic quiver variety $\mathcal{X}_{\theta}(n\delta)$ (see \S \ref{Tautbundlesec}) generalizing Wilson's construction of the Calogero-Moser space in \cite{Wil}. Considering $\mathcal{X}_{\theta}(n\delta)$ as a hyper-K\"{a}hler manifold, one can use reflection functors, defined by Nakajima in \cite{N}, to construct a hyper-K\"{a}hler isometry $\mathcal{X}_{\theta}(n\delta) \to \mathcal{X}_{-\mathbf{\frac{1}{2}}}(\gamma)$ between quiver varieties associated to different parameters. Furthermore, rotation of complex structure yields a diffeomorphism between $\mathcal{X}_{-\mathbf{\frac{1}{2}}}(\gamma)$ and a certain GIT quotient $\mathcal{M}_{-\mathbf{1}}(\gamma)$ (see \S \ref{Tautbundlesec}). The latter is isomorphic to an irreducible component $\Hilb_K^\nu$ of $\Hilb_K^{\Z/l\Z}$, where  
$\Hilb_K$ denotes the Hilbert scheme of $K$ points in $\C^2$. 
The following diagram summarizes all the maps involved:

\begin{equation} \label{Beginning all maps} \mathcal{Y}_{\mathbf{h}} \xrightarrow{\EG} \mathcal{X}_{\theta}(n\delta) \xrightarrow{\mathsf{Refl. Fun.}} \mathcal{X}_{-\mathbf{\frac{1}{2}}}(\gamma) \xrightarrow{\mathsf{Rotation}} \mathcal{M}_{-\mathbf{1}}(\gamma) \hookrightarrow \Hilb_K^{\Z/l\Z}.\end{equation}

Let us explain the parameters. The affine symmetric group $\tilde{S}_l$ acts on dimension vectors and the parameter space associated to the cyclic quiver with $l$ vertices (see \S \ref{subsec: part cyclic quiver}, \S \ref{reflection functors definition section}). 
We apply this action to the dimension vector $n \delta$ and the parameter $\mathbf{-\frac{1}{2}} := -\frac{1}{2l}(1, \hdots,1)$. 
Fix $w \in \tilde{S}_l$ and set $\theta := w^{-1} \cdot (\mathbf{-\frac{1}{2}})$ and $\gamma := w * n\delta$. Then $\gamma = n\delta + \gamma_0$, where $\gamma_0$ is the $l$-residue of a uniquely determined $l$-core partition $\nu$. Set $K:=nl+|\nu|$. The relation between the parameters $\mathbf{h}$ and $\theta$ is explained in \S \ref{subsec: EG iso}. 

Both $\mathcal{Y}_{\mathbf{h}}$ and $\Hilb_K$ carry natural $\C^*$-actions with respect to which  \eqref{Beginning all maps} is equivariant. It follows from \cite{gor-bvm} that the closed $\C^*$-fixed points in $\mathcal{Y}_{\mathbf{h}}$ are labelled by $l$-multipartitions of~$n$. On the other hand, the $\C^*$-fixed points in $\Hilb_K$ correspond to monomial ideals in $\C[x,y]$ of colength $K$ and are therefore labelled by partitions of $K$. 
In particular, the $\C^*$-fixed points in $\Hilb_K^{\nu}$ are labelled by partitions of $K$ with $l$-core $\nu$. 
Since \eqref{Beginning all maps} is equivariant, it induces a bijection 
\begin{equation} \label{intro pln pvk bij} \mathcal{P}(l,n) \longleftrightarrow \mathcal{P}_\nu(K),\end{equation}   
where $\mathcal{P}_\nu(K)$ denotes the set of partitions of $K$ with $l$-core $\nu$ and $\mathcal{P}(l,n)$ the set  of $l$-multipartitions of $n$. 

Our main result is an explicit combinatorial description of \eqref{intro pln pvk bij}. We note that such a description already appeared in \cite{gor-qv}, but its proof in \emph{loc.\ cit.} is incorrect (see Remark \ref{remark counterexample}). Our result has a number of interesting applications. We use it to establish a higher-level version of the combinatorial identity called the $q$-hook formula - a $q$-analogue of the well-known relationship between the dimension of a Specht module and hook lengths in the corresponding Young diagram. Our main result has also recently been used by Bonnaf\'{e} and Maksimau \cite{BM} in their study of fixed-point subvarieties in  Calogero-Moser spaces. Moreover, the combinatorial description of \eqref{intro pln pvk bij} is a key ingredient in several older results, such as Bezrukavnikov and Finkelberg's \cite{FinBez}, as well as Losev's \cite{Los}, proofs of Haiman's wreath Macdonald positivity conjecture.

\subsection{Main results.} 

 Our first result gives a classification as well as an explicit description of $\C^*$-fixed points in quiver varieties associated to the cyclic quiver. We also consider tautological bundles on these varieties and calculate the characters of their fibres at the $\C^*$-fixed points. 

\begin{theorem} \label{intro: class of fp}
Let $u \in \tilde{S}_l$, $\xi:= u * n\delta =  n\delta + \xi_0$ and let $\omega$ be the transpose of the $l$-core corresponding to $\xi_0$. Set $L := nl+|\omega|$. 
Let $\alpha \in \Q^l$ be any parameter such that $\mathcal{X}_\alpha(\xi)$ is smooth. Let $\mathcal{V}_\alpha(\xi)$ denote the tautological bundle on $\mathcal{X}_\alpha(\xi)$ (see \S \ref{Tautbundlesec}). Then: 
\begin{enumerate}[label=\alph*), font=\textnormal,noitemsep,topsep=3pt,leftmargin=0.55cm]
\item The $\C^*$-fixed points in $\mathcal{X}_\alpha(\xi)$ are naturally labelled by $\mathcal{P}_{\omega}(L)$. We construct them explicitly as equivalence classes of quiver representations. 
\item Let $\mu \in \mathcal{P}_{\omega}(L)$. Then the $\C^*$-character of the fibre of $\mathcal{V}_\alpha(\xi)$ at $\mu$ is given by
\[\ch_t \mathcal{V}_\alpha(\xi)_{\mu} = \Res_{\mu}(t):= \sum_{\square \in \mu} t^{c(\square)}.\]
\end{enumerate} 
\end{theorem}

This theorem combines the results of Theorem \ref{C fp classification theorem}, Proposition \ref{residue-character} and Corollary \ref{v-core fp bij}  below. Let us briefly explain our description of the $\C^*$-fixed points. To each partition $\mu \in \mathcal{P}_{\omega}(L)$ we associate a quadruple of matrices depending on $\alpha, \xi$ and the Frobenius form of $\mu$. Our construction can be regarded as a generalization of Wilson's description of the $\C^*$-fixed points in the (classical) Calogero-Moser space \cite[Proposition 6.11]{Wil}. One can also interpret our construction in terms of unravelling the cyclic quiver into the infinite linear quiver $A_\infty$. This connects our results with earlier work on quiver varieties of type $A_\infty$  \cite{FS,Sav}.

Our second result describes the bijection between the $\C^*$-fixed points induced by the Etingof-Ginzburg isomorphism. 
\begin{theorem}[Theorem \ref{Cor final final}] \label{thm intro 1}
The map $ \mathcal{Y}_{\mathbf{h}} \xrightarrow{\EG} \mathcal{X}_{\theta}(n\delta)$ induces a bijection \[ \mathcal{P}(l,n) \to \mathcal{P}_\varnothing(nl), \quad \quot(\mu)^\flat \mapsto \mu,\]
where $\quot(\mu)^\flat$ denotes the reverse of the $l$-quotient of $\mu$ (see 
\S \ref{partitions and multis section}, \S \ref{subsec: cores and quotients}) and $\varnothing$ is the empty partition. 
\end{theorem}
\noindent
The proof of Theorem \ref{thm intro 1} occupies sections \ref{dahas chapter} and \ref{lquo chapter}. We use the Dunkl-Opdam subalgebra of $\mathbb{H}_{\mathbf{h}}$ to construct a commutative diagram 
\begin{equation*} \label{intro comm}
\begin{tikzcd}
\mathcal{Y}_{\mathbf{h}} \arrow[rr,"\EG"] \arrow{dr}[swap]{\rho_1} &  & \mathcal{X}_{\theta}(n\delta) \arrow[ld, "\rho_2"] \\
 & \C^n/S_n & 
\end{tikzcd}
\end{equation*} 
\noindent
where $\rho_1$ sends a fixed point labelled by $\underline{\lambda}$ to its residue and $\rho_2$ sends a quiver representation to a certain subset of its eigenvalues. Given a partition $\mu$, we use our description of the corresponding fixed point from Theorem \ref{intro: class of fp} to obtain an explicit formula for $\rho_2(\mu)$ in terms of the Frobenius form of $\mu$. We then use a combinatorial argument to show that this formula describes the residue of $\quot(\mu)^\flat$.

We next consider reflection functors, which were introduced by Nakajima \cite{N} and also studied by Maffei \cite{Maf} and Crawley-Boevey and Holland \cite{CBH,CBH2}. 
Assuming smoothness, reflection functors are $U(1)$-equivariant hyper-K\"{a}hler isometries 
between quiver varieties associated to different parameters. They satisfy Weyl group relations and have been used by Nakajima to define Weyl group representations on homology groups of quiver varieties. 

Let us explain the role reflection functors play in our setting. 
To each simple reflection $\sigma_i \in \tilde{S}_l$, we associate a reflection functor $\mathfrak{R}_i : \mathcal{X}_{\alpha}(\xi) \to \mathcal{X}_{\sigma_i \cdot \alpha}(\sigma_i * \xi)$. One can show that $\sigma_i * \xi = n\delta + \sigma_i*\xi_0$, where $\sigma_i*\xi_0$ is the $l$-residue of a uniquely determined $l$-core $\omega'$.  Then the reflection functor $\mathfrak{R}_i$ induces a  bijection between the labelling sets of $\C^*$-fixed points 
\begin{equation} \label{reflection functor bij intro} \mathbf{R}_i : \mathcal{P}_\omega(L) \to \mathcal{P}_{(\omega')^t}(L'), \end{equation} 
where $L':= nl + |\omega'|$.

Our third result gives a combinatorial description of this bijection. We use the action of $\tilde{S}_l$ on the set of all partitions defined by Van Leeuwen in \cite{Leew}. This action involves combinatorial ideas reminiscent of those describing the $\hat{\mathfrak{sl}}_l$-action on the Fock space. More precisely, if $\mu$ is a partition then $\sigma_i * \mu$ is the partition obtained by simultaneously removing and adding all the removable (resp.\ addable) cells of content $i$ $\mmod$ $l$ from (to) the Young diagram of $\mu$. It is noteworthy that this action also plays a role in the combinatorics describing the Schubert calculus of the affine Grassmannian \cite{Lam1, Lam2}.

\begin{theorem}[Theorem \ref{RkTk}] \label{thm intro 2}
Let $\mu \in \mathcal{P}_{\omega}(L)$. Then \[ \mathbf{R}_i(\mu) = (\sigma_i*\mu^t)^t.\]
\end{theorem}

Combining Theorem \ref{thm intro 1} with (iterated applications of) Theorem \ref{thm intro 2} allows us to give an explicit combinatorial description of bijection \eqref{intro pln pvk bij}. 
\begin{theorem}[Theorem \ref{caracal1}] \label{intro fixed points bij all}
The map \eqref{Beginning all maps} induces the following bijections
\[
\begin{array}{c c c c c c c}
\mathcal{Y}_{\mathbf{h}}^{\C^*} &\longrightarrow& \mathcal{X}_{\theta}(n\delta)^{\C^*} &\longrightarrow&  \mathcal{X}_{-\mathbf{\frac{1}{2}}}(\gamma)^{\C^*} &\longrightarrow& (\Hilb_K^\nu)^{\C^*} \\
\mathcal{P}(l,n) &\longrightarrow& \mathcal{P}_{\varnothing}(nl) &\longrightarrow& \mathcal{P}_{\nu^t}(K) &\longrightarrow& \mathcal{P}_{\nu}(K) \\
\quot(\mu)^\flat &\longmapsto& \mu &\longmapsto& (w*\mu^t)^t &\longmapsto& w*\mu^t.
\end{array}
\]
Moreover, $\nu = w*\varnothing$. 
\end{theorem}

Let us rephrase our result slightly. Given $w \in \tilde{S}_l$, we define the $w$\emph{-twisted} $l$\emph{-quotient bijection} to be the map
\[ \tau_{w} \colon \mathcal{P}(l,n) \to \mathcal{P}_{\nu}(K), \quad \quot(\mu) \mapsto w*\mu.\]

\begin{corollary}[Corollary \ref{caracal2}] \label{cor intro tau}
Bijection \eqref{intro pln pvk bij} is given by 
\begin{equation*} \underline{\lambda} \mapsto \tau_{w}(\underline{\lambda}^t). \end{equation*} 
\end{corollary}

One of Gordon's motivations in \cite{gor-qv} was to give a geometric interpretation of highest weight structures on category $\O_{\mathbf{h}}$ for rational Cherednik algebras $\mathbb{H}_{t=1,\mathbf{h}}(\Gamma_n)$ at $t=1$. 
Consider the combinatorial ordering $\prec_{\mathbf{h}}^{\mathsf{com}}$ on $\mathcal{P}(l,n)$ defined by  
\[ \underline{\mu} \preceq_{\mathbf{h}}^{\mathsf{com}} \underline{\lambda} \iff \tau_{w}(\underline{\lambda}^t)  \trianglelefteq \tau_{w}(\underline{\mu}^t),\]
where $\trianglelefteq$ denotes the dominance ordering on partitions. 
It was shown by Dunkl and Griffeth \cite[Theorem 1.2]{Gri} that $\O_{\mathbf{h}}$ is a highest weight category with respect to this ordering. There is also a geometric ordering $\prec_{\mathbf{h}}^{\mathsf{geo}}$ on $\mathcal{P}(l,n)$, defined by the closure relations between the attracting sets of $\C^*$-fixed points in $\mathcal{M}_{2\theta}(n\delta)$. Using Corollary \ref{cor intro tau} and the results of Nakajima from \cite{Nak-Jack} we deduce the following partial geometric interpretation of the combinatorial ordering.  

\begin{corollary}[Corollary \ref{cor macdonald}] \label{intro orderings}
Let $\underline{\mu}, \underline{\lambda} \in \mathcal{P}(l,n)$. Then
$ \underline{\mu} \preceq_{\mathbf{h}}^{\mathsf{geo}} \underline{\lambda} \ \Rightarrow \ \underline{\mu} \preceq_{\mathbf{h}}^{\mathsf{com}} \underline{\lambda}.$ 
\end{corollary} 

We remark that the statements of Corollaries \ref{cor intro tau} and \ref{intro orderings} first appeared in \cite{gor-qv} (see Proposition 7.10 and its proof). However, the proof of Proposition 7.10 in \cite{gor-qv} is incorrect - see Remark \ref{remark counterexample} for an explanation. 

\subsection{The higher-level $q$-hook formula.}

Our results have several interesting applications. One of them is a new proof as well as a generalization of the $q$-hook formula:
\begin{equation} \label{intro hf} \sum_{\square \in \mu} t^{c(\square)} = [n]_t \sum_{\lambda \uparrow \mu}\frac{ f_{\lambda}(t)}{f_{\mu}(t)},\end{equation}
where $\mu$ is a partition of $n$, $c(\square)$ is the content of $\square$ and $f_\mu(t)$ is the fake degree polynomial associated to $\mu$. The sum on the RHS ranges over subpartitions of $\mu$ obtained by deleting precisely one cell in the Young diagram of $\mu$. Note that the RHS of \eqref{intro hf} can also be reformulated in terms of Schur functions and hook length polynomials. 
The $q$-hook formula has been proven by Kerov \cite{ker}, Garsia and Haiman \cite{GH} and Chen and Stanley \cite{CS} using probabilistic, combinatorial and algebraic methods. We prove the following generalization. 

\begin{theorem}[Theorem \ref{cyclotomic q-hook formula}] \label{theorem D}
Let $\mu \in \mathcal{P}_\varnothing(nl)$. Then:
\begin{equation} \label{cyclo} \sum_{\square \in \mu} t^{c(\square)} = [nl]_t \sum_{\underline{\lambda} \uparrow \quot(\mu)^\flat}\frac{f_{\underline{\lambda}}(t)}{f_{\quot(\mu)^\flat}(t)}.\end{equation}
\end{theorem}
We call \eqref{cyclo} the \emph{higher-level} $q$-hook formula. Setting $l=1$ we recover the classical $q$-hook formula. 
Our proof of Theorem \ref{theorem D} is geometric in nature. Let us briefly explain the main idea behind it. 
Let $e_n$ denote the symmetrizing idempotent in $\Gamma_n$. The right $e_n\mathbb{H}_{\mathbf{h}}e_n$-module $\mathbb{H}_{\mathbf{h}}e_n$ defines a coherent sheaf on $\mathcal{Y}_{\mathbf{h}}$. Since we are assuming that the variety $\mathcal{Y}_{\mathbf{h}}$ is smooth, this sheaf is also locally free. Let $\mathcal{R}_{\mathbf{h}}$ denote the corresponding vector bundle. 
It was shown in \cite{EG} that there exists an isomorphism of vector bundles $\mathcal{R}_{\mathbf{h}}^{\Gamma_{n-1}} \xrightarrow{\sim} \mathcal{V}_\theta(n\delta)$ 
lifting the Etingof-Ginzburg isomorphism $\mathcal{Y}_{\mathbf{h}} \xrightarrow{\sim} \mathcal{X}_{\theta}(n\delta)$. 
Let $\mu \in \mathcal{P}_\varnothing(nl)$. By Theorem \ref{thm intro 1}, the Etingof-Ginzburg map sends the fixed point labelled by $\quot(\mu)^\flat$ to the fixed point labelled by $\mu$. 
We obtain the higher level $q$-hook formula \eqref{cyclo} by comparing the $\C^*$-characters of the corresponding fibres $(\mathcal{R}_{\mathbf{h}}^{\Gamma_{n-1}})_{\quot(\mu)^\flat}$ and ${\mathcal{V}_\theta(n\delta)}_{\mu}$. 

\subsection{Other applications.} 
Let us mention a few other applications of our results. The first is related to Finkelberg and Bezrukavnikov's \cite{FinBez} as well as Losev's \cite{Los} proofs of Haiman's wreath Macdonald positivity conjecture. 
The original positivity conjecture, proven by Haiman in \cite{Hai2}, asserts that Kostka-Macdonald polynomials, which express the change of basis between transformed Macdonald functions and Schur functions, have non-negative coefficients. Haiman \cite{Hai} later proposed a generalized conjecture, known as the wreath Macdonald positivity conjecture, in which the ring of symmetric functions is replaced by the space of virtual characters of the group $\Gamma_n$. 

We will now explain the role the description of the bijection \eqref{intro pln pvk bij} from Corollary \ref{cor intro tau} plays in the above-mentioned proofs of the wreath Macdonald positivity conjecture. 
The key step in Bezrukavnikov and Finkelberg's proof is a characterization of the support of certain Verma modules in positive characteristic \cite[Proposition 2.6]{FinBez}. Losev's proof also relies on a calculation of the supports of certain quotients of Procesi bundles \cite[Proposition 5.3]{Los}. The proofs of these two statements invoke \cite[Lemma 3.8]{FinBez}. But the latter implicitly uses Corollary \ref{cor intro tau} (see also \cite[\S 2.3]{FinBez}). 

We mention two other applications of ours results. 
Gordon and Martino \cite{GM} gave a combinatorial description of the blocks of the restricted rational Cherednik algebra of type $G(l,1,n)$  (also for parameters $\mathbf{h}$ for which the corresponding Calogero-Moser space $\mathcal{Y}_{\mathbf{h}}$ is singular) in terms of $J$-classes of partitions. Corollary \ref{cor intro tau} is an important ingredient in their proof. 

More recently, Bonnaf\'{e} and Maksimau \cite{BM} studied the irreducible components of the fixed point subvariety under the action of a finite cyclic group on a smooth Calogero-Moser space. They use Theorem \ref{thm intro 1} to give an explicit description of these components for Calogero-Moser spaces of type $G(l,1,n)$. 

\subsection*{Acknowledgements.} 
The research for this paper was carried out with the financial support of the College of Science \& Engineering at the University of Glasgow. The material will form part of my PhD thesis. 
I would like to thank G. Bellamy for suggesting the problem to me as well as his support and encouragement. I am also grateful to C. Stroppel for comments on a draft version of this paper and to I. Gordon for discussing \cite{gor-qv} with~me. 

\section{Rational Cherednik algebras}
We begin by recalling the definition and basic properties of cyclotomic rational Cherednik algebras at $t=0$. 
We work over the field of complex numbers throughout the paper. 

\subsection{Partitions and multipartitions.} \label{partitions and multis section}
Let $k$ be a non-negative integer. A partition $\lambda$ of $k$ is an infinite non-increasing sequence $( \lambda_1, \lambda_2, \lambda_3, \hdots \ )$ of non-negative integers such that $\sum_{i=1}^{\infty} \lambda_i = k$. We write $|\lambda| = k$ and denote the set of all partitions of $k$ by $\mathcal{P}(k)$. Let $\ell(\lambda)$ be the positive integer $i$ such that $\lambda_i \neq 0$ but $\lambda_{i+1}=0$. We say that $\mu = (\mu_1, \mu_2, \mu_3, \hdots )$ is a subpartition of $\lambda$ if $\mu$ is a partition of some positive integer $m \leq k$ and $\mu_i \leq \lambda_i$ for all $i = 1, 2, \hdots \ .$ A subpartition $\mu$ of $\lambda$ is called a restriction of $\lambda$, denoted $\mu \uparrow \lambda$, if $| \mu | = k-1$. Let $\varnothing = (0,0,\hdots)$ denote the empty partition. 

An $l$-composition $\alpha$ of $k$ is an $l$-tuple $(\alpha_0, \hdots, \alpha_{l-1})$ of non-negative integers such that $\sum_{i=0}^{l-1} \alpha_i = k$. An $l$-multipartition $\underline{\lambda}$ of $k$ is an $l$-tuple $(\lambda^0, \hdots, \lambda^{l-1})$ such that each $\lambda^i$ is a partition and $\sum_{i=0}^{l-1} |\lambda^i| = k$. We consider the upper indices modulo $l$. 
Let $\mathcal{P}(l,k)$ denote the set of $l$-multipartitions of $k$. We say that $\underline{\mu} = (\mu^0, \hdots, \mu^{l-1})$ is a submultipartition of $\underline{\lambda}$ if each $\mu^i$ is a subpartition of $\lambda^i$. We call a submultipartition $\underline{\mu}$ of $\underline{\lambda}$ a restriction of $\underline{\lambda}$, denoted $\underline{\mu} \uparrow \underline{\lambda}$, if $\sum_{i=0}^{l-1} |\mu^i| = k-1$. 

If $\lambda$ is a partition we denote its transpose by $\lambda^t$. If $\underline{\lambda} = (\lambda^0, \hdots, \lambda^{l-1}) \in \mathcal{P}(l,k)$, we call $\underline{\lambda}^t = ((\lambda^0)^t, \hdots, (\lambda^{l-1})^t)$ the \emph{transpose} multipartition and $\underline{\lambda}^{\flat} := (\lambda^{l-1}, \lambda^{l-2}, \hdots, \lambda^0)$ the \emph{reverse} multipartition. Finally, we set
$\mathcal{P} := \bigsqcup_{k \in \Z_{\geq 0}} \mathcal{P}(k)$ and $\underline{\mathcal{P}} := \bigsqcup_{k \in \Z_{\geq 0}} \mathcal{P}(l,k).$

\subsection{Wreath products.} \label{subsec: wreath products} Let us fix once and for all two positive integers $n,l$. 
We regard the symmetric group $S_n$ as the group of permutations of the set $\{1, \hdots, n\}$. For $1 \leq i < j \leq n$ let $s_{i,j}$ denote the transposition swapping numbers $i$ and $j$. We abbreviate $s_i = s_{i,i+1}$ for $i=1,\hdots,n-1$. 
Let  $C_l := \Z / l \Z = \langle \epsilon \rangle$ and set $\Gamma_n := C_l \wr S_n = (C_l)^{n} \rtimes S_n$, the wreath product of $C_l$ and $S_n$. It is a complex reflection group of type $G(l,1,n)$. For $ 1 \leq i \leq n$ and $ 1 \leq j \leq l-1$ let $\epsilon^j_i$ denote the element $ (1, \hdots , 1 , \epsilon^j, 1 , \hdots , 1) \in (C_l)^n$ which is non-trivial only in the $i$-th coordinate. Let $e_{n} := (l^nn!)^{-1}\sum_{g \in \Gamma_{n}} g$ be the symmetrizing idempotent and let 
$\triv$ denote the trivial $\Gamma_n$-module. 

We regard $S_{n-1}$ as the subgroup of $S_n$ generated by the transpositions $s_{2}, \hdots, s_{n-1}$. 
We also regard $(C_l)^{n-1}$ as a subgroup of $(C_l)^n$ consisting of elements whose first coordinate is equal to one. This determines an embedding $\Gamma_{n-1} \hookrightarrow \Gamma_n$. 
Note that $(\C \Gamma_n)^{\Gamma_{n-1}} = e_{n-1}\C\Gamma_n$ and $| (\C \Gamma_n)^{\Gamma_{n-1}} | = nl$. 

Isomorphism classes of irreducible $\Gamma_n$-modules are naturally parametrized by $\mathcal{P}(l,n)$. We use the parametrization given in \cite[\S 6.1.1]{Rouq}. Let $S(\underline{\lambda})$ denote the irreducible $\Gamma_n$-module corresponding to the $l$-multipartition $\underline{\lambda}$. We will later need the following branching rule \cite[Theorem 10]{Pus}. 

\begin{proposition} \label{branching} 
Let $\underline{\lambda} \in \mathcal{P}(l,n)$. Then \ 
$ \displaystyle S(\underline{\lambda})|_{\Gamma_{n-1}} := \Res_{\Gamma_{n-1}}^{\Gamma_n} S(\underline{\lambda}) = \bigoplus_{\underline{\mu} \uparrow \underline{\lambda}} S(\underline{\mu}).$
\end{proposition}

\subsection{Rational Cherednik algebras.}
Let us recall the definition of the rational Cherednik algebra of type $G(l,1,n)$ at $t=0$.  
Set $\eta := e^{2 \pi i / l}$. 
Let $\h$ be the $n$-dimensional representation of $\Gamma_n$ with basis $y_1, \hdots, y_n$ such that $\epsilon_i\sigma.y_j = \eta^{-\delta_{i,\sigma(j)}} y_{\sigma(j)}$ for any $\sigma \in S_n$. Let $x_1, \hdots, x_n$ be the dual basis of $\h^*$.

\begin{definition} \label{RCA definition}
Let us choose a parameter $\mathbf{h} = (h, H_1, \hdots, H_{l-1}) \in \Q^l$ and set $H_0 = - (H_1 + \hdots + H_{l-1})$. The \emph{rational Cherednik algebra} $\mathbb{H}_{\mathbf{h}}$ (at $t=0$) associated to $\Gamma_n$ is the quotient of the cross-product $T(\h \oplus \h^*) \rtimes \C \Gamma_n$ 
by the relations
\begin{itemize}
\item $[x_i, x_j] = [y_i, y_j] = 0$ for all $1 \leq i,j \leq n$, 
\item $[x_i,y_j] = - h \sum_{k=0}^{l-1} \eta^k s_{i,j} \epsilon_i^k \epsilon^{-k}_j$ for all $1 \leq i \neq j \leq n$, 
\item $[x_i,y_i] =  h\sum_{j \neq i} \sum_{k=0}^{l-1} s_{i,j} \epsilon_i^k \epsilon^{-k}_j + \sum_{k=0}^{l-1} ( \sum_{m=0}^{l-1} \eta^{-mk} H_m) \epsilon_i^k$ for all $1 \leq i \leq n$.
\end{itemize}
\end{definition}

\noindent
Let $Z_{\mathbf{h}}$ denote the centre of $\mathbb{H}_{\mathbf{h}}$. 
By \cite[Theorem 3.1]{EG}, the map
\[ Z_{\mathbf{h}} \to e_n \mathbb{H}_{\mathbf{h}} e_n, \ z \mapsto z \cdot e_n\]
is an algebra isomorphism. It is called the Satake isomorphism. 
Consider the $\C^*$-action on $\mathbb{H}_{\mathbf{h}}$ defined by the rule $t.x_i = tx_i$, $t.y_i = t^{-1} y_i$ and $t.g = g$, where $1 \leq i \leq n$, $g \in \Gamma_n$ and $t \in \C^*$. This action restricts to actions on $e_n\mathbb{H}_{\mathbf{h}}e_n$ and $Z_{\mathbf{h}}$, with respect to which the Satake isomorphism is equivariant. 
Note that the $\C^*$-action on $\mathbb{H}_{\mathbf{h}}$ can also be interpreted as a $\Z$-grading such that $\deg x_i = 1$, $\deg y_i = -1$ and $\deg g = 0$. 

\begin{notat}
Given a $\Z$-graded vector space $V$ with finite-dimensional homogeneous components, let $\ch_t V \in \Z[[t,t^{-1}]]$ denote its Poincar\'{e} series. Since a $\Z$-grading is equivalent to a $\C^*$-module structure, we can regard $\ch_t V$ as the $\C^*$-character of $V$.
\end{notat}

\begin{definition}
Let  $ \C[\h]^{\Gamma_n}_+$ (resp.\ $\C[\h^*]^{\Gamma_n}_-$) denote the ideal of $\C[\h]^{\Gamma_n}$ (resp.\ $ \C[\h^*]^{\Gamma_n}$) generated by homogeneous elements of positive (resp.\ negative) degree, in the grading defined by the $\C^*$-action on $\mathbb{H}_{\mathbf{h}}$. The quotient 
\[ \overline{\mathbb{H}}_{\mathbf{h}} := \mathbb{H}_{\mathbf{h}}/ \mathbb{H}_{\mathbf{h}}. (\C[\h]^{\Gamma_n}_+ + \C[\h^*]^{\Gamma_n}_-)\]
is called the \emph{restricted rational Cherednik algebra}. It is a finite-dimensional algebra. 
\end{definition}

Let $\C[\h]^{co\Gamma_n} := \C[\h] / \C[\h].\C[\h]^{\Gamma_n}_+$ be the algebra of coinvariants with respect to the $\Gamma_n$-action.
It follows from the PBW theorem for rational Cherednik algebras \cite[Theorem 1.3]{EG} that there is an isomorphism of graded vector spaces $\overline{\mathbb{H}}_{\mathbf{h}} \cong \C[\h]^{co \Gamma_n} \otimes  \C[\h^*]^{co \Gamma_n} \otimes \C\Gamma_n$. Moreover, $\C[\h^*]^{co \Gamma_n} \rtimes \C \Gamma_n$ is a subalgebra of $\overline{\mathbb{H}}_{\mathbf{h}}$.
\begin{definition}
Let $\underline{\lambda} \in \mathcal{P}(l,n)$. The irreducible $\C \Gamma_n$-module $S(\underline{\lambda})$ becomes a module over $\C[\h^*]^{co \Gamma_n} \rtimes \C \Gamma_n$ by means of the projection $\C[\h^*]^{co \Gamma_n} \rtimes \C \Gamma_n \to \C\Gamma_n$. The \emph{baby Verma module} associated to $\lambda$ is the induced module
\[ \Delta(\underline{\lambda}) := \overline{\mathbb{H}}_{\mathbf{h}} \otimes_{\C[\h^*]^{co \Gamma_n} \rtimes \C \Gamma_n} S(\underline{\lambda}).\] 
We consider $\Delta(\underline{\lambda})$ as a graded $\overline{\mathbb{H}}_{\mathbf{h}}$-module  with $1 \otimes S(\underline{\lambda})$ in degree $0$. 
\end{definition}

\begin{proposition}[{\cite[Proposition 4.3]{gor-bvm}}] \label{bvms}
Let $\underline{\lambda} \in \mathcal{P}(l,n)$. The baby Verma module $\Delta(\underline{\lambda})$ is indecomposable with simple head $L(\underline{\lambda})$. Moreover, $\{ L(\underline{\lambda}) \mid \underline{\lambda} \in \mathcal{P}(l,n)\}$ form a complete and irredundant set of representatives of isomorphism classes of graded simple $\overline{\mathbb{H}}_{\mathbf{h}}$-modules, up to a grading shift. 
\end{proposition}

\subsection{The variety $\mathcal{Y}_{\mathbf{h}}$.} \label{RCA: taut v bundle}

Let $\mathcal{Y}_{\mathbf{h}} := \Spec Z_{\mathbf{h}}$. 
We will always assume that the parameter $\mathbf{h}$ is chosen so that the variety $\mathcal{Y}_{\mathbf{h}}$ is \emph{smooth}. 
A criterion for smoothness can be found in e.g.\ \cite[Lemma 4.3]{gor-qv}. 

Let $\Irrep (\mathbb{H}_{\mathbf{h}})$ denote the set of isomorphism classes of irreducible representations of $\mathbb{H}_{\mathbf{h}}$. If $[M] \in \Irrep (\mathbb{H}_{\mathbf{h}})$, let $\chi_M : Z_{\mathbf{h}} \to \C$ denote the character by which $Z_{\mathbf{h}}$ acts on $M$. 
By \cite[Theorem 1.7]{EG}, there is a bijection
\begin{equation}\label{bijection} \Irrep (\mathbb{H}_{\mathbf{h}})  \longleftrightarrow  \MaxSpec Z_{\mathbf{h}}, \quad  [M] \mapsto \ker \chi_M.
\end{equation}
We are now going to recall another description of $\mathcal{Y}_{\mathbf{h}}$. 
\begin{definition} \label{sec rep var}
Let $\Rep_{\C \Gamma_n} (\mathbb{H}_{\mathbf{h}})$ be the variety of all algebra homomorphisms $\mathbb{H}_{\mathbf{h}} \to \End_{\C}(\C \Gamma_n)$ whose restriction to $\C \Gamma_n \subset \mathbb{H}_{\mathbf{h}}$ is the $\C \Gamma_n$-action by left multiplication, i.e., the regular representation. This is an affine algebraic variety. 
\end{definition} 
Let $\phi \in \Rep_{\C \Gamma_n} (\mathbb{H}_{\mathbf{h}})$. The one-dimensional vector space $e_n \C \Gamma_n$ is stable under all the endomorphisms in $\phi(e_n\mathbb{H}_{\mathbf{h}}e_n)$. 
Therefore, $\phi|_{e_n\mathbb{H}_{\mathbf{h}}e_n}$ composed with the Satake isomorphism yields an algebra homomorphism $\chi_\phi : Z_{\mathbf{h}} \cong e_n\mathbb{H}_{\mathbf{h}}e_n \to \End_{\C}(e_n\C \Gamma_n) \cong \C$. 
We obtain a morphism of algebraic varieties
\begin{equation} \label{pi} \pi : \Rep_{\C \Gamma_n} (\mathbb{H}_{\mathbf{h}}) \to \mathcal{Y}_{\mathbf{h}}, \quad \phi \mapsto \ker \chi_\phi.\end{equation}
The $\C^*$-action on $\mathbb{H}_{\mathbf{h}}$ induces $\C^*$-actions on the varieties $\Rep_{\C \Gamma_n} (\mathbb{H}_{\mathbf{h}})$ and $\mathcal{Y}_{\mathbf{h}}$, with respect to which $\pi$ is equivariant. 

Let $\Aut_{\Gamma_n}(\C \Gamma_n)$ be the group of $\C$-linear $\Gamma_n$-equivariant automorphisms of $\C \Gamma_n$. 
The group $\Aut_{\Gamma_n}(\C \Gamma_n)$ acts naturally on $\Rep_{\C \Gamma_n} (\mathbb{H}_{\mathbf{h}})$: if $g \in \Aut_{\Gamma_n}(\C \Gamma_n)$ and $\phi \in \Rep_{\C \Gamma_n} (\mathbb{H}_{\mathbf{h}})$ then $(g.\phi)(z) = g\phi(z)g^{-1}$, for all $z \in \mathbb{H}_{\mathbf{h}}$. By \cite[Theorem 3.7]{EG}, there 
exists an irreducible component $\Rep^o_{\C \Gamma_n} (\mathbb{H}_{\mathbf{h}})$ of $\Rep_{\C \Gamma_n} (\mathbb{H}_{\mathbf{h}})$ such that 
\eqref{pi} induces a $\C^*$-equivariant isomorphism of algebraic varieties
\begin{equation} \label{pi2} \Rep^o_{\C \Gamma_n} (\mathbb{H}_{\mathbf{h}}) \sslash \Aut_{\Gamma_n}(\C \Gamma_n)   \xrightarrow{\sim} \mathcal{Y}_{\mathbf{h}}.\end{equation}

Next, consider the $(\mathbb{H}_{\mathbf{h}},e_n\mathbb{H}_{\mathbf{h}}e_n)$-bimodule $\mathbb{H}_{\mathbf{h}}e_n$ together with the $\C^*$-action inherited from $\mathbb{H}_{\mathbf{h}}$. The bimodule $\mathbb{H}_{\mathbf{h}}e_n$ defines a $\C^*$-equivariant coherent sheaf $\widetilde{\mathbb{H}_{\mathbf{h}}e_n}$ on $\Spec e_n\mathbb{H}_{\mathbf{h}}e_n  \cong  \mathcal{Y}_{\mathbf{h}}$. 
Since we are assuming that $\mathcal{Y}_{\mathbf{h}}$ is smooth, \cite[Theorem 1.7]{EG} implies that this sheaf is locally free. 

\begin{definition}
Let $\mathcal{R}_{\mathbf{h}}$ denote the $\C^*$-equivariant vector bundle whose sheaf of sections is $\widetilde{\mathbb{H}_{\mathbf{h}}e_n}$. 
\end{definition}

The group $\Gamma_n$ acts naturally on every fibre of $\mathcal{R}_{\mathbf{h}}$ from the left. Let $\mathcal{R}^{\Gamma_{n-1}}_{\mathbf{h}}=e_{n-1}\mathcal{R}_{\mathbf{h}}$ be the subbundle of $\mathcal{R}_{\mathbf{h}}$ consisting of $\Gamma_{n-1}$-invariants and let $(\mathcal{R}^{\Gamma_{n-1}}_{\mathbf{h}})_{\underline{\lambda}}$ denote its fibre at~$\chi_{\underline{\lambda}}$. 

\subsection{$\C^*$-fixed points.} \label{subsec:cfp}

Let us recall the classification of $\C^*$-fixed points in $\mathcal{Y}_{\mathbf{h}}$ in terms of $l$-multipartitions of $n$ from \cite{gor-bvm}. 
By \cite[Proposition 4.15]{EG}, the subalgebra $\C[\h]^{\Gamma_n} \otimes \C[\h^*]^{\Gamma_n}$ of $\mathbb{H}_{\mathbf{h}}$ is contained in $Z_{\mathbf{h}}$ and $Z_{\mathbf{h}}$ is a free $\C[\h]^{\Gamma_n} \otimes \C[\h^*]^{\Gamma_n}$-module of rank $|\Gamma_n|$.
The inclusion $\C[\h]^{\Gamma_n} \otimes \C[\h^*]^{\Gamma_n} \hookrightarrow Z_{\mathbf{h}}$ induces a $\C^*$-equivariant morphism of algebraic varieties 
\[ \Upsilon \colon \mathcal{Y}_{\mathbf{h}} \to \h/\Gamma_n\times \h^* / \Gamma_n.\] 
The only $\C^*$-fixed point in $\h/\Gamma_n\times \h^* / \Gamma_n$ is $0$. Since the group $\C^*$ is connected and the fibre $\Upsilon^{-1}(0)$ is finite, $\mathcal{Y}_{\mathbf{h}}^{\C^*} = \Upsilon^{-1}(0).$ 
By Theorem 5.6 in \cite{gor-bvm}, there is a bijection between the closed points of $\Upsilon^{-1}(0)$ and isomorphism classes of simple modules over the restricted rational Cherednik algebra $\overline{\mathbb{H}}_{\mathbf{h}}$. Hence there is a bijection
\[ \mathcal{P}(l,n) \longleftrightarrow (\MaxSpec Z_{\mathbf{h}})^{\C^*}, \quad \underline{\lambda} \mapsto \ker \chi_{L(\underline{\lambda})}.\]
We will also write $\chi_{\underline{\lambda}}$ for $\chi_{L(\underline{\lambda})}$. 

\section{Quiver varieties} \label{CM spaces chapter}

In this section we recall the connection between rational Cherednik algebras and cyclic quiver varieties via the Etingof-Ginzburg isomorphism. 

\subsection{The cyclic quiver.} \label{subsec: cyclic quiver}

Let $\mathbf{Q}$ be the cyclic quiver with $l$ vertices and a cyclic orientation. We label the vertices as $0, 1, \hdots, l-1$ (considered as elements of $\Z/l\Z$) in such a way that there is a (unique) arrow $i \to j$ if and only if $j = i+1$. Let $\overline{\mathbf{Q}}$ be the double of $\mathbf{Q}$, i.e., the quiver obtained from $\mathbf{Q}$ by adding, for each arrow $a$ in $\mathbf{Q}$, an arrow $a^*$ going in the opposite direction. Moreover, let $\mathbf{Q}_\infty$ be the quiver obtained from $\mathbf{Q}$ by adding an extra vertex, denoted $\infty$, and an extra arrow $a_\infty : \infty \to 0$. We write $\overline{\mathbf{Q}}_\infty$ for the double of $\mathbf{Q}_\infty$. 

Let $\mathbf{d} = (d_0, \hdots, d_{l-1}) \in (\Z_{\geq  0})^{l}$. We interpret $\mathbf{d}$ as the dimension vector for $\overline{\mathbf{Q}}$ so that the dimension associated to the vertex $i$ is $d_i$. For each $i = 0, \hdots, l-1$ let $\mathbf{V}_i$ be a complex vector space of dimension $d_i$. Set $\widehat{\mathbf{V}} := \bigoplus_{i=0}^{l-1}\mathbf{V}_i$. Moreover, let $\mathbf{V}_\infty$ be a one-dimensional complex vector space and set $\mathbf{V} := \mathbf{V}_\infty \oplus \widehat{\mathbf{V}}$. Define 
\[ \Rep(\overline{\mathbf{Q}}_\infty,\mathbf{d}) := \bigoplus_{i=0}^{l-1}\Hom(\mathbf{V}_i, \mathbf{V}_{i+1})  \oplus \bigoplus_{i=0}^{l-1}\Hom(\mathbf{V}_i, \mathbf{V}_{i-1}) \oplus \Hom(\mathbf{V}_0,\mathbf{V}_\infty) \oplus \Hom(\mathbf{V}_\infty,\mathbf{V}_0). \]
We denote an element of $\Rep(\overline{\mathbf{Q}}_\infty,\mathbf{d})$ as $(\mathbf{X},\mathbf{Y},I,J) = (X_0, \hdots, X_{l-1}, Y_0, \hdots, Y_{l-1},I,J)$ accordingly. 
There is a natural isomorphism of varieties $\Rep(\overline{\mathbf{Q}}_\infty,\mathbf{d}) \cong T^*\Rep(\mathbf{Q}_\infty,\mathbf{d})$, through which we can equip $\Rep(\overline{\mathbf{Q}}_\infty,\mathbf{d})$ with a symplectic structure. 
\vspace{1 em}

\hspace{6 em}
\scalebox{.8}
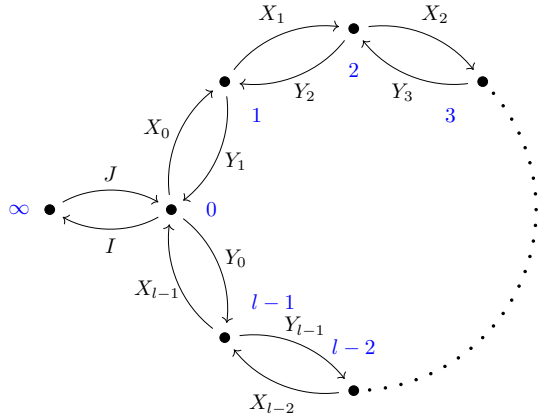
{
\begin{tikzpicture}
\foreach \ang\lab\anch in {90/2/north, 45/3/{north east}, 270/{l-2}/south, 180/0/west, 225/{l-1}/{south west}, 135/1/{north west}}{
  \draw[fill=black] ($(0,0)+(\ang:3)$) circle (.08);
  \node[anchor=\anch, color=blue!100] at ($(0,0)+(\ang:2.55)$) {$\lab$};
  %\draw[->,shorten <=7pt, shorten >=7pt] ($(0,0)+(\ang:3)$).. controls +(\ang+40:1.5) and +(\ang-40:1.5) .. ($(0,0)+(\ang:3)$);
}
\draw[fill=black] ($(-5,0)$) circle (.08); 
\draw[->,shorten <=7pt, shorten >=7pt] ($(-3,0)$) to [bend left] ($(-5,0)$); 
\draw[->,shorten <=7pt, shorten >=7pt] ($(-5,0)$) to [bend left] ($(-3,0)$); 
\node at ($(-4,-0.6)$) {$I$}; 
\node at ($(-4,0.6)$) {$J$}; 
\node[color=blue!100] at ($(-5.5,0)$) {$\infty$}; 
\foreach \ang\lab in {90/2,180/0,135/1,225/{l-1},270/{l-2}}{
  \draw[->,shorten <=7pt, shorten >=7pt] ($(0,0)+(\ang:3)$) to [bend left] ($(0,0)+(\ang-45:3)$); %(\ang:\ang-45:3);
  \node at ($(0,0)+(\ang-22.5:3.5)$) {$X_{\lab}$};
}
\foreach \ang\lab in {90/2,180/0,135/1,225/{l-1},45/3}{
  \draw[->,shorten <=7pt, shorten >=7pt] ($(0,0)+(\ang:3)$) to [bend left] ($(0,0)+(\ang+45:3)$); %(\ang:\ang-45:3);
  \node at ($(0,0)+(\ang+22.5:2.1)$) {$Y_{\lab}$};
}
\foreach \ang in {275,280,285,290,295,300,305,310,315,320,325,330,335,340,345,350,355,360,5,10,15,20,25,30,35,40}{
  \draw[fill=black] ($(0,0)+(\ang:3)$) circle (.02);
}
\end{tikzpicture}}
\vspace{1 em}

The algebraic group $G(\mathbf{d}):=\prod_{i=0}^{l-1}\GL(\mathbf{V}_i)$ acts on $\Rep(\overline{\mathbf{Q}}_\infty,\mathbf{d})$ by change of basis. If $\mathbf{g} = (g_0, \hdots, g_{l-1}) \in G(\mathbf{d})$ and $(\mathbf{X},\mathbf{Y}, I,J) \in \Rep(\overline{\mathbf{Q}}_\infty,\mathbf{d})$ then
\[ \mathbf{g}. (\mathbf{X},\mathbf{Y}, I,J) = (g_1X_0g_0^{-1}, \hdots, g_0X_{l-1}g_{l-1}^{-1}, g_{l-1}Y_0g_0^{-1}, \hdots, g_{l-2}Y_{l-1}g_{l-1}^{-1}, Ig_0^{-1}, g_0J).\]
The action of $G(\mathbf{d})$ on $\Rep(\overline{\mathbf{Q}}_\infty,\mathbf{d})$ is Hamiltonian. The moment map for this action is given by 
\[ \mu_{\mathbf{d}} : \Rep(\overline{\mathbf{Q}}_\infty,\mathbf{d}) \to \mathfrak{g}(\mathbf{d})^* \cong \mathfrak{g}(\mathbf{d}):=\Lie G(\mathbf{d}), \quad (\mathbf{X}, \mathbf{Y},I,J) \mapsto [\mathbf{X},\mathbf{Y}]+JI.\]

\subsection{Quiver varieties.}
\label{Tautbundlesec}

If $\mathbf{d} \in (\Z_{\geq 0})^{l}$ and $\theta = (\theta_0, \hdots,\theta_{l-1}) \in \Q^{l}$, we will also write $\theta = (\theta_0 \id_0, \theta_1 \id_1, \hdots , \theta_{l-1}\id_{l-1})\in \g(\mathbf{d})$, where $\id_i = \id_{\mathbf{V}_i}$ ($i = 0, \hdots, l-1$). 
Define 
\[ \mathcal{X}_{\theta}(\mathbf{d}) :=  \mu_{\mathbf{d}} ^{-1}(\theta) \sslash G(\mathbf{d}) := \Spec \C[\mu_{\mathbf{d}} ^{-1}(\theta)]^{G(\mathbf{d})}.\] 
We will always assume that the parameter $\theta$ is chosen in such a way that the variety $\mathcal{X}_{\theta}(\mathbf{d})$ is smooth. 
Moreover, define the GIT quotient
\[ \mathcal{M}_\theta(\mathbf{d}) := \mu_{\mathbf{d}} ^{-1}(0) \sslash_{\theta} G(\mathbf{d}) = \Proj \bigoplus_{i \geq 0} \C[\mu_{\mathbf{d}} ^{-1}(0)]^{\chi_\theta^i},\]
where $\chi_\theta : G(\mathbf{d}) \to \C^*$ is the character sending $\mathbf{g}$ to $\prod (\det g_i)^{\theta_i}$ and $ \C[\mu_{\mathbf{d}} ^{-1}(0)]^{\chi_\theta^i}$ denotes the space of semi-invariant functions on $\mu_{\mathbf{d}} ^{-1}(0)$, i.e., those functions $f$ satisfying $\mathbf{g}.f = \chi_\theta^i(\mathbf{g})f$. By definition, the space  $ \C[\mu_{\mathbf{d}} ^{-1}(0)]^{\chi_\theta^i}$ is zero unless $i\theta \in \Z^l$. 

The varieties $\mathcal{X}_\theta(\mathbf{d})$ and $\mathcal{M}_\theta(\mathbf{d})$ can be endowed with hyper-K\"{a}hler structures (see e.g.\ \cite[\S 3.6]{gor-qv}). Moreover, the group $\C^*$ acts on $\Rep(\overline{\mathbf{Q}}_\infty,\mathbf{d})$ by the rule $t.(\mathbf{X},\mathbf{Y},I,J)  =  (t^{-1}\mathbf{X},t\mathbf{Y}, I,J)$ for $t\in \C^*$. This action descends to actions on $\mathcal{X}_{\theta}(\mathbf{d})$ and $\mathcal{M}_{\theta}(\mathbf{d})$. By $\mathcal{X}_{\theta}(\mathbf{d})^{\C^*}$ and $\mathcal{M}_{\theta}(\mathbf{d})^{\C^*}$ we will always mean the sets of \emph{closed} $\C^*$-fixed points. 

Let us recall the definition of the tautological bundle on a quiver variety. 
Assume that the group $G(\mathbf{d})$ acts freely on the fibre $\mu_{\mathbf{d}}^{-1}(\theta)$ and 
consider the trivial vector bundle $\widehat{\mathcal{V}}_\theta(\mathbf{d}) :=\mu^{-1}_{\mathbf{d}}(\theta) \times \widehat{\mathbf{V}}$ on $\mu^{-1}_{\mathbf{d}}(\theta)$. 
We regard $\widehat{\mathcal{V}}_\theta(\mathbf{d})$ as a $\C^*$-equivariant vector bundle by letting $\C^*$ act trivially on $ \widehat{\mathbf{V}}$. 
Let $G(\mathbf{d})$ act diagonally on $\widehat{\mathcal{V}}_\theta(\mathbf{d})$. 
The vector bundle $\widehat{\mathcal{V}}_\theta(\mathbf{d})$ descends to a $\C^*$-equivariant vector bundle 
\[ \mathcal{V}_\theta(\mathbf{d}):= \mu^{-1}_{\mathbf{d}}(\theta) \times^{G(\mathbf{d})}  \widehat{\mathbf{V}} = (\mu^{-1}_{\mathbf{d}}(\theta) \times  \widehat{\mathbf{V}}) \sslash G(\mathbf{d})\]
on $\mathcal{X}_{\theta}(\mathbf{d})$, 
which is called the \emph{tautological bundle}. 

\begin{notat}
We will always consider the subscript $i$ in the expressions $d_i, \mathbf{V}_i, g_i, X_i, Y_i, \theta_i$ modulo $l$ (unless $i = \infty$). 
\end{notat}

\subsection{The Etingof-Ginzburg isomorphism.} \label{subsec: EG iso}
Throughout this subsection we assume $\mathbf{d}=n\delta$, where $\delta := (1,\hdots,1) \in \Z^{l}$. Given $\mathbf{h}$ as in Definition \ref{RCA definition}, set 
\begin{equation} \label{theta vs h} \theta_{\mathbf{h}} = (\theta_0, \hdots, \theta_{l-1}) = (-h+ H_0, H_1, \hdots, H_{l-1}).
\end{equation}  
Since we are assuming that the parameter $\mathbf{h}$ is generic, the group $G(n\delta)$ acts freely on the fibre $\mu_{n\delta}^{-1}(\theta_{\mathbf{h}})$ (see \cite[Proposition 11.11]{EG}). 
We abbreviate 
\[ \mathcal{C}_{\mathbf{h}}:=\mathcal{X}_{\theta_{\mathbf{h}}}(n\delta), \quad \mathcal{V}_{\mathbf{h}} := \mathcal{V}_{\theta_{\mathbf{h}}}(n\delta).\]

Consider $(\C\Gamma_n)^{\Gamma_{n-1}}$ as a module over $\langle \epsilon_1 \rangle = \Z/l\Z$. It  decomposes as a direct sum 
of $n$-dimensional isotypic components $(\C\Gamma_n)^{\Gamma_{n-1}}_{\chi_i}$, where $\chi_i$ is the character $\chi_i : \epsilon_1 \mapsto \eta^i$. Fix a linear isomorphism 
\begin{equation} \label{varpi} (\C\Gamma_n)^{\Gamma_{n-1}} \to \widehat{\mathbf{V}}\end{equation} mapping each  $(\C\Gamma_n)^{\Gamma_{n-1}}_{\chi_i}$ onto $\mathbf{V}_i$. It induces an isomorphism of endomorphism algebras
\begin{equation} \label{varpitx} \varpi : \End_{\C} \left(  (\C\Gamma_n)^{\Gamma_{n-1}} \right) \to \End_{\C} ( \widehat{\mathbf{V}} ).\end{equation}
\begin{definition}   
Recall $\Rep_{\C\Gamma_n}(\mathbb{H}_{\mathbf{h}})$ from Definition \ref{sec rep var}. Each $\phi \in \Rep_{\C\Gamma_n}(\mathbb{H}_{\mathbf{h}})$ defines endomorphisms $\phi(x_1), \phi(y_1) : \C\Gamma_n \to \C\Gamma_n$, where $x_1, y_1 \in \mathbb{H}_{\mathbf{h}}$ are as in Definition \ref{RCA definition}. Set 
\[ \mathbf{X}(\phi) := \varpi \left(\phi(x_1)|_{(\C\Gamma_n)^{\Gamma_{n-1}}}\right), \quad  \mathbf{Y}(\phi) := \varpi \left(\phi(y_1)|_{(\C\Gamma_n)^{\Gamma_{n-1}}}\right).  \]
Consider the maps 
\begin{align} \label{Psi defi} \Psi \colon \Rep^o_{\C\Gamma_n}(\mathbb{H}_{\mathbf{h}}) \to \Rep(\overline{\mathbf{Q}},n\delta), \quad& \phi \mapsto (\mathbf{X}(\phi), \mathbf{Y}(\phi)), \\
\label{p defi} p \colon \Rep(\overline{\mathbf{Q}}_\infty,n\delta) \to \Rep(\overline{\mathbf{Q}},n\delta), \quad& (\mathbf{X}, \mathbf{Y}, I,J) \mapsto (\mathbf{X}, \mathbf{Y})
\end{align}
\end{definition}

\begin{lemma}
The maps $\Psi$ and $p$ are $\C^*$-equivariant.
\end{lemma} 
\begin{proof}
The equivariance of $p$ is obvious. 
Let $t\in \C^*, \phi \in \Rep_{\C\Gamma_n}(\mathbb{H}_{\mathbf{h}})$ and $z \in \mathbb{H}_{\mathbf{h}}$. We have $(t.\phi)(z)= \phi(t^{-1}.z)$, so $(t.\phi)(x_1)=t^{-1}\phi(x_1)$ and $(t.\phi)(y_1)=t\phi(y_1)$. 
Hence \begin{align*} \Psi(t.\phi) =& \left(\varpi \left(t^{-1}\phi(x_1)|_{(\C\Gamma_n)^{\Gamma_{n-1}}}\right),\varpi \left(t\phi(y_1)|_{(\C\Gamma_n)^{\Gamma_{n-1}}}\right)\right)\\ =& (t^{-1}\mathbf{X}(\phi), t \mathbf{Y}(\phi)) = t.(\mathbf{X}(\phi),\mathbf{Y}(\phi)). \qedhere\end{align*} 
\end{proof} 

The proof of \cite[Proposition 11.24]{EG} carries over directly to yield the following generalization. 

\begin{theorem} \label{EG iso theoremm} 
Maps \eqref{Psi defi} and \eqref{p defi} 
induce a $\C^*$-equivariant isomorphism of varieties \begin{equation} \label{EG iso} \EG \colon \mathcal{Y}_{\mathbf{h}} \xrightarrow{\sim} \mathcal{C}_{\mathbf{h}}\end{equation} and vector bundles
\begin{equation} \mathcal{R}^{\Gamma_{n-1}}_{\mathbf{h}} \xrightarrow{\sim} \mathcal{V}_{\mathbf{h}}.\end{equation}
\end{theorem}
\noindent
We call \eqref{EG iso} the \emph{Etingof-Ginzburg isomorphism}.

\section{Combinatorics}

We recall several combinatorial notions which will be used throughout the paper. 

\subsection{Young diagrams.} \label{Felineorigin}

If $\mu$ is a partition of $k$, set $n(\mu) := \sum_{i \geq 1} i \cdot \mu_{i+1}$. If $\underline{\lambda}$ is an $l$-multipartition of $k$, define $r(\underline{\lambda}) = \sum_{i = 1}^{l-1} i \cdot |\lambda^{i}|$. 
Recall the notations $[n]_t = \frac{1-t^n}{1-t} = 1 + \hdots + t^{n-1}$ and $(t)_n = (1-t)(1-t^2)\hdots(1-t^n)$.

Let $\mu = (\mu_1, \hdots, \mu_m, 0, \hdots)$ be a partition of $k$, where $\mu_1, \hdots, \mu_m$ are non-zero. 
Let $\mathbb{Y}(\mu) := \{ (i,j) \mid 1 \leq i \leq m, \ 1 \leq j \leq \lambda_i \}$ denote the Young diagram of $\mu$. We will always display Young diagrams according to the English convention. 
We call each pair $(i,j) \in \mathbb{Y}(\mu)$ a cell. We will often use the symbol $\square$ to refer to cells. Sometimes we will also abuse notation and write $\mu$ instead of $\mathbb{Y}(\mu)$ where no confusion can arise, e.g., $\square \in \mu$ instead of $\square \in \mathbb{Y}(\mu)$. 
If $\square = (i,j) \in \mathbb{Y}(\mu)$ is a cell, let $c(\square) := j-i$ be the \emph{content} of $\square$. We call $\Res_{\mu}(t) := \sum_{\square \in \mu} t^{c(\square)}$ the \emph{residue} of $\mu$. We also call $c(\square) \mmod l$ the $l$\emph{-content} of $\square$ and $\sum_{\square \in \mu} t^{c(\square) \mmod l}$ the $l$\emph{-residue} of $\mu$. It is clear that a partition is determined uniquely by its residue. 

Now suppose that $\underline{\lambda}$ is an $l$-multipartition of $k$. By the Young diagram of $\underline{\lambda}$ we mean the $l$-tuple $\mathbb{Y}(\underline{\lambda}):=(\mathbb{Y}(\lambda^0), \hdots, \mathbb{Y}(\lambda^{l-1}))$. By a cell $\square \in \mathbb{Y}(\underline{\lambda})$ we mean a cell in any of the Young diagrams $\mathbb{Y}(\lambda^i)$. Let $\mathbf{e} = (e_0, \hdots, e_{l-1}) \in \Q^l$. We define the $\mathbf{e}$-residue of $\underline{\lambda}$ to be 
\[ \Res_{\underline{\lambda}}^{\mathbf{e}}(t) := \sum_{i=0}^{l-1} t^{e_i} \Res_{\lambda^i}(t).\] 
For sufficiently generic $\mathbf{e}$, 
an $l$-multipartition is determined uniquely by its $\mathbf{e}$-residue.

\subsection{Hook length polynomials.}   \label{Frob hooks sec}

Let $\mu$ be a partition and fix a cell $(i,j) \in \mathbb{Y}(\mu)$. By the \emph{hook} associated to the cell $(i,j)$ we mean the set $\{(i,j)\} \cup \{ (i',j) \in \mathbb{Y}(\mu) \mid i' > i \} \cup \{ (i,j') \in \mathbb{Y}(\mu) \mid j' > j \}$. We call $(i,j)$ the \emph{root} of the hook, $\{ (i',j) \in \mathbb{Y}(\mu) \mid i' > i \}$ the \emph{leg} of the hook and $\{ (i,j') \in \mathbb{Y}(\mu) \mid j' > j \}$ the \emph{arm} of the hook. The cell in the leg of the hook with the largest first coordinate is called the \emph{foot} of the hook, and the cell in the arm of the hook with the largest second coordinate is called the \emph{hand} of the hook. 

By a \emph{hook} in $\mathbb{Y}(\mu)$ we mean a hook associated to some cell $\square \in \mathbb{Y}(\mu)$. If $H$ is a hook, let $\arm(H)$ denote its arm and let $\leg(H)$ denote its leg. 
If $\square$ is the root of $H$, let $a_\mu(\square) := |\arm(H)|$ and $l_\mu(\square) := |\leg(H)|$. Set $h_\mu(\square) := 1 + a_\mu(\square) + l_\mu(\square)$. 
The \emph{hook length polynomial} of the partition $\mu$ is
\[ H_\mu (t) := \prod_{\square \in \mu} (1 - t^{h_\mu(\square)}).\]
Hook length polynomials are related to Schur functions by the following equality
\[ s_\mu(1,t,t^2, \hdots ) = \frac{t^{n(\mu)}}{H_\mu(t)}.\]

\subsection{Frobenius form of a partition.} \label{Frobenius form} 

By a \emph{Frobenius hook} in $\mathbb{Y}(\mu)$ we mean a hook whose root is a cell of content zero. Clearly $\mathbb{Y}(\mu)$ is the disjoint union of all its Frobenius hooks. Suppose that $(1,1), (2,2), \hdots, (k,k)$ are the cells of content zero in $\mathbb{Y}(\mu)$. Let $F_i$ denote the Frobenius hook with root $(i,i)$. We endow the set of Frobenius hooks with the natural ordering $F_1 < F_2 < \hdots < F_k$. We call $F_1$ the \emph{innermost} or \emph{first} Frobenius hook and $F_k$ the \emph{outermost} or \emph{last} Frobenius hook. Let $a_i = a_\mu(i,i)$ and $b_i = l_\mu(i,i)$. We call $(a_1, \hdots, a_k \mid b_1, \hdots, b_k)$ the \emph{Frobenius form} of $\mu$.

\subsection{Bead diagrams.} 
Let us recall the notion of a bead diagram (see e.g.\ \cite[\S 2.7]{JK}). 
We call an element $(i,j)$ of $\Z_{\leq -1} \times \{0, \hdots, l-1\}$ a \emph{point}. We say that the point $(i,j)$ lies to the left of $(i,j')$ if $j < j'$, and that $(i,j)$ lies above $(i',j)$ if $i'<i$. 

A \emph{bead diagram} is a function $f: \Z_{\leq -1} \times \{0, \hdots, l-1\} \to \{0,1\}$ which takes value $1$ for only finitely many points. If $f(i,j) = 1$ we say that the point $(i,j)$ is occupied by a \emph{bead}. If $f(i,j) = 0$ we say that the point $(i,j)$ is \emph{empty}. Suppose that a point $(i,j)$ is empty and that there exists an $i' < i$ such that the point $(i',j)$ is occupied by a bead. Then we call the point $(i,j)$ a \emph{gap}. 

We say that a point $(i,j) \in \Z_{\leq -1} \times \{0,\hdots,l-1\}$ is in the $(-i)$-th row and $j$-th column (or runner) of the bead diagram. We call the $i$-th row full (empty) if every point $(i,k)$ for $k=0,\hdots,l-1$ is occupied by a bead (is empty). A row is called redundant if it is a full row and if all the rows above it are full. 
A graphical interpretation of the notion of a bead diagram can be found in Example \ref{exa:bead diagram}. We only display the rows containing at least one bead or gap.

\begin{definition} \label{defi: beta numbers}
Let $\mu \in \mathcal{P}$ and $p \geq \ell(\mu)$. Set \[ \beta^p_i = \mu_i + p - i \quad (1 \leq i \leq p).\]
We call $\{ \beta_i^p \mid 1 \leq i \leq p \}$ a set of $\beta$-numbers for $\mu$. Note that $|\{ \beta_i^p \mid 1 \leq i \leq p \}| = p$. From each set of $\beta$-numbers one can uniquely recover the corresponding partition $\mu$.
\end{definition}

\begin{definition}
Given a set of $\beta$-numbers $\{ \beta_i^p \mid 1 \leq i \leq p \}$ we can naturally associate to it a bead diagram by the rule
\[ f(i,j) = 1 \iff -(i+1) \cdot l + j \in \{ \beta_i^p \mid 1 \leq i \leq p \}.\]
Let $\mu$ be as in Definition \ref{defi: beta numbers}. 
If $p$ is the smallest multiple of $l$ satisfying $p \geq \ell(\mu)$ we denote the resulting bead diagram by $\mathbb{B}(\mu)$. The diagram $\mathbb{B}(\mu)$ has no redundant rows and the number of beads in $\mathbb{B}(\mu)$ is a multiple of $l$. 
\end{definition}

\begin{remark}
Conversely, if we are given a bead diagram $f$, the set $\{ -(i+1) \cdot l + j \mid f(i,j) = 1 \}$ is a set of $\beta$-numbers for some partition. The relationship between bead diagrams, sets of $\beta$-numbers and partitions can therefore be illustrated as follows
\[ \left\{ \mbox{bead diagrams} \right\} \longleftrightarrow \left\{ \mbox{sets of } \beta\mbox{-numbers} \right\} \twoheadrightarrow  \{ \mbox{partitions} \}, \]
where the set of partitions contains partitions of an arbitrary integer. 
\end{remark}

\subsection{Cores and quotients.} \label{subsec: cores and quotients} 

Let $f: \Z_{\leq -1} \times \{0, \hdots, l-1\} \to \{0,1\}$ be a bead diagram. Suppose that the point $(i,j)$ with $i <-1$ is occupied by a bead, i.e., $f(i,j) = 1$, and that $f(i+1,j)=0$. To \emph{slide} or \emph{move} the bead in position $(i,j)$ \emph{upward} means to modify the function $f$ by setting $f'(i,j) = 0, f'(i+1,j) = 1$ and $f' = f$ otherwise. 

\begin{definition}
Let $\mu$ be a partition. Take any bead diagram $f$ corresponding to $\mu$. We obtain a new bead diagram $f'$ by sliding beads upward as long as it is possible. We call the partition corresponding to the bead diagram $f'$ the $l$-\emph{core} of $\mu$, denoted $\core(\mu)$. Let $\heartsuit(l)$ denote the set of all $l$-cores. 
We set
\[ \mathcal{P}_\nu(k) := \{ \mu \in \mathcal{P}(k) \mid \core(\mu) = \nu \}.\]
\end{definition}

\begin{definition}
Consider the bead diagram $\mathbb{B}(\mu)$. Each column of $\mathbb{B}(\mu)$ can itself be considered as a bead diagram for $l=1$. Let $Q^i(\mu)$ denote the partition corresponding to the $i$-th column. We call the $l$-multipartition $\quot(\mu) := (Q^0(\mu), Q^1(\mu), \hdots, Q^{l-1}(\mu))$ the $l$-\emph{quotient} of $\mu$. 
\end{definition}

A partition is determined uniquely by its $l$-core and $l$-quotient \cite[Theorem 2.7.30]{JK}. In particular, there is a bijection 
\[ \mathcal{P}_{\varnothing}(nl) \to  \mathcal{P}(l,n), \quad \mu \mapsto \quot(\mu).\]

\begin{example} \label{exa:bead diagram}
Consider the partition $\mu = (6,5,3,3,1,1)$ and take $l=3$. The first-column hook-lengths are $11,9,6,5,2,1$. They form a set of $\beta$-numbers. The bead diagram below on the left illustrates $\mathbb{B}(\mu)$ while the diagram on the right illustrates the effect of sliding all the beads upward. 
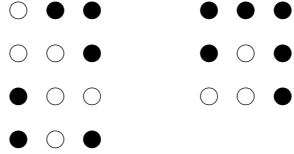
\[
\xy
\xymatrixrowsep{0.05in}
\xymatrixcolsep{0in}
\xymatrix "M"{
\Circle & \CIRCLE & \CIRCLE \\
\Circle & \Circle & \CIRCLE \\
\CIRCLE & \Circle & \Circle \\ 
\CIRCLE & \Circle & \CIRCLE 
}
\endxy
 \hspace{3 em}
\xy
\xymatrixrowsep{0.05in}
\xymatrixcolsep{0in}
\xymatrix "M"{
\CIRCLE & \CIRCLE & \CIRCLE \\
\CIRCLE & \Circle & \CIRCLE \\
\Circle & \Circle & \CIRCLE 
}
\endxy
\] 
Let us read off the $3$-core of $\mu$ from the diagram on the right. We can ignore all the beads before the first empty point, which we label as zero. We carry on counting. The remaining two beads get labels $1$ and $4$. These form a set of $\beta$-numbers, which we can interpret as the first-column hook-lengths corresponding to the partition $(3,1)$. It follows that $\core(\mu)=(3,1)$. 
To determine the $3$-quotient of $\mu$ we divide the diagram on the left into three columns and consider each separately. We read off the $\beta$-numbers as before - they are $2,3$ for the first column, $0$ for the second column and $1$ for the third column. It follows that $\quot(\mu) = ( (2,2), \varnothing, (1))$. 
\end{example}

\subsection{Rim-hooks.}  
The \emph{rim} of $\mathbb{Y}(\mu)$ is the subset  of $\mathbb{Y}(\mu)$ consisting of the cells $(i,j)$ such that $(i+1,j+1)$ does not lie in $\mathbb{Y}(\mu)$.
Fix a cell $(i,j) \in \mathbb{Y}(\mu)$. Recall that by the hook associated to $(i,j)$ we mean the subset of $\mathbb{Y}(\mu)$ consisting of all the cells $(i,k)$ with $k \geq j$ and all the cells $(k,j)$ with $k \geq i$. We define the \emph{rim-hook} associated to the cell $(i,j)$ to be the intersection of the set $\{ (i',j') \mid i' \geq i, j' \geq j \}$ with the rim of $\mathbb{Y}(\mu)$. We call a rim-hook an $l$-\emph{rim-hook} if it contains $l$ cells. 
The $l$-core of $\mu$ can also be characterised as the subpartition $\mu'$ of $\mu$ obtained from $\mu$ by a successive removal of $l$-rim-hooks, in whichever order (see \cite[Theorem 2.7.16]{JK}). 

\begin{lemma}[{\cite[Lemma 2.7.13]{JK}}] \label{rem hooks quot} 
Let $R$ be an $l$-rim-hook in $\mu$ and set $\mu' := \mu - R$. Then $\quot(\mu') = \quot(\mu) - \square$ for some $\square \in \quot(\mu)$. 
\end{lemma}

\subsection{The $\tilde{S}_l$-action on partitions.} \label{NorweskiLesny} 
Assume for the rest of this section that $l > 1$.  
All subscripts should be regarded modulo $l$. 
Let $\tilde{S}_l$ denote the affine symmetric group. It has a Coxeter presentation with generators $\sigma_0, \hdots, \sigma_{l-1}$ and relations \[ \sigma_i^2=1, \quad \sigma_i\sigma_{i+1}\sigma_i = \sigma_{i+1}\sigma_i\sigma_{i+1} \quad (0 \leq i \leq l-1).\] 
Let us recall the $\tilde{S}_l$-action on the set $\mathcal{P}$ of all partitions from \cite[\S 4]{Leew}. We will later use this action to describe the behaviour of the $\C^*$-fixed points under reflection functors. 
We need the following definition, reminiscent of the combinatorics of the Fock space. 

\begin{definition}
Let $k \in \{0, \hdots, l-1\}$. Consider the Young diagram $\mathbb{Y}(\mu)$ as a subset of the $\Z_{>0} \times \Z_{>0}$ space. We say that a cell $(i,j) \in \mathbb{Y}(\mu)$ is \emph{removable} if $\mathbb{Y}(\mu) - \{(i,j)\}$ is the Young diagram of a partition.
We say that it is
$k$-\emph{removable} if additionally $c(i,j) = j - i = k \ \mmod \ l$. We call a cell $(i,j) \notin \mathbb{Y}(\mu)$ \emph{addable} if $\mathbb{Y}(\mu) \cup \{(i,j)\}$ is the Young diagram of a partition. We call it 
$k$-\emph{addable} if additionally $c(i,j) = j - i = k \ \mmod \ l$. 
\end{definition}

We discuss the combinatorics of removability and addability in more detail in \S \ref{para-remadd}. We will, in particular, require Lemma \ref{Lemma-corners} proven there. 

\begin{definition} \label{defi:part-remadd} Suppose that $\mu \in \mathcal{P}$ and $k \in \{0, \hdots, l-1\}$.
Define $\mathbf{T}_k(\mu)$ to be the partition such that \begin{equation} \label{Ydiagram add rem} \mathbb{Y}(\mathbf{T}_k(\mu)) = \mathbb{Y}(\mu) \cup \{ \square \mbox{ is } k\mbox{-addable} \} - \{ \square \mbox{ is } k\mbox{-removable} \}.\end{equation}
The group $\tilde{S}_l$ acts on $\mathcal{P}$ by the rule \[ \sigma_i*\mu = \mathbf{T}_i(\mu) \quad (\mu \in \mathcal{P}, i \in \Z/l\Z).\] 
\end{definition}

This action also plays a role in the combinatorics of the Schubert calculus of the affine Grassmannian, see \cite[\S 8.2]{Lam1} and \cite[\S 11]{Lam2}. 
By \cite[Proposition 22]{Lam2}, we have $\tilde{S}_l * \varnothing = \heartsuit(l)$. 
Let us recall how the $\tilde{S}_l$-action behaves with respect to cores and quotients. 
Consider the finite symmetric group $S_l$ as  the group of permutations of the set $\{0,\hdots,l-1\}$. Let $s_i \in S_l$ ($i=1,\hdots,l-1$) be the transposition swapping $i-1$ and $i$. Let $s_0$ be the transposition swapping $0$ and $l-1$. Note that our conventions for $S_l$ differ from those for $S_n$ introduced in \S \ref{subsec: wreath products}. 
The finite symmetric group $S_l$ acts on the set $\underline{\mathcal{P}}$ of all $l$-multipartitions  by the~rule
\[ w\cdot\underline{\lambda} = (\lambda^{w^{-1}(0)}, \hdots, \lambda^{w^{-1}(l-1)}), \quad w \in S_l.\] 
Consider the group homomorphism
\[ \mathsf{pr} : \tilde{S}_l \twoheadrightarrow S_l, \quad \sigma_i \mapsto s_{i} \ (i = 0, \hdots, l-1).\]

\begin{proposition}[{\cite[Proposition 4.13]{Leew}}] \label{Affineactioncorequotient}
Let $\mu \in \mathcal{P}$ and $\sigma \in \tilde{S}_l$. Then 
\[ \core(\sigma*\mu) = \sigma*\core(\mu), \quad \quot(\sigma*\mu) = \mathsf{pr}(\sigma)\cdot\quot(\mu).\]
\end{proposition}

\subsection{Partitions and the cyclic quiver.} \label{subsec: part cyclic quiver}

Let $N_i(\lambda)$ be the number of cells of $l$-content $i$ in $\mathbb{Y}(\lambda)$. Using this notation, the $l$-residue of $\lambda$ equals $\sum_{i=0}^{l-1} N_i(\lambda) t^i$. Consider the map
\begin{equation} \label{dmapp} \mathfrak{d} : \mathcal{P} \to \Z^l, \quad \lambda \mapsto \mathbf{d}_\lambda:=(N_0(\lambda), \hdots, N_{l-1}(\lambda)).\end{equation}
We interpret this map as assigning to every partition a dimension vector for the cyclic quiver with $l$ vertices.
Let \[\Z_{\heartsuit} = \{ \mathbf{d} \in (\Z_{\geq 0})^l \mid \mathbf{d}=\mathbf{d}_\nu \mbox{ for some } \nu \in \heartsuit(l)\}\] be the set of all dimension vectors corresponding to $l$-cores.  By \cite[Theorem 2.7.41]{JK} an $l$-core is determined uniquely by its $l$-residue. Hence \eqref{dmapp} restricts to a bijection
\[ \mathfrak{d} : \heartsuit(l) \longleftrightarrow \Z_{\heartsuit}, \quad \nu \mapsto \mathbf{d}_\nu.\]

There is an $\tilde{S}_l$-action on $\Z^l$ defined as follows. Let $\mathbf{d} \in \Z^l$. Then $\sigma_i * \mathbf{d} = \mathbf{d}'$ with $d'_j = d_j$ ($j \neq i$) and 
\[ d'_i = d_{i+1} + d_{i-1} -d_i \quad (i \neq 0), \quad d'_0 = d_{1} + d_{l-1} -d_0+1 \quad (i = 0).\]
The following proposition follows by an elementary calculation from Lemma \ref{Lemma-corners}. 

\begin{proposition}
The following diagram is $\tilde{S}_l$-equivariant 
\[
\xymatrixcolsep{5pc}
\xymatrix{
\mathcal{P} \ar[r]^{\mathfrak{d}} & \Z^l \\
\heartsuit(l) \ar[r]^{\sim}_{\mathfrak{d}} \ar@{^{(}->}[u] & \Z_{\heartsuit} \ar@{^{(}->}[u].
}
\]
\end{proposition}

Let $\sigma_i \in \tilde{S}_l$ and $\nu \in \heartsuit(l)$. Then $\sigma_i * (n\delta + \mathbf{d}_\nu) = n\delta + \sigma_i * \mathbf{d}_\nu$ and $\sigma_i * \mathbf{d}_\nu = \mathfrak{d}(\sigma_i *\nu)$. 
By \cite[Theorem 2.7.41]{JK} any partition $\lambda$ of $nl + |\sigma_i * \nu|$ 
such that $\mathfrak{d}(\lambda) = n\delta + \sigma_i * \mathbf{d}_\nu$ 
has $l$-core $\sigma_i * \nu$. Hence
\[ \mathcal{P}_{\sigma_i*\nu}(nl+|\sigma_i*\nu|) = \mathfrak{d}^{-1}(n\delta+\sigma_i*\mathbf{d}_\nu).\]
 
\subsection{Reflection functors.} \label{reflection functors definition section} 
The group $\tilde{S}_l$ also acts on the parameter space $\Q^l$ for the quiver $\overline{\mathbf{Q}}_\infty$ by the rule $\sigma_i \cdot \theta = \theta'$ with \[\theta'_i = -\theta_i, \quad \theta'_{i-1} = \theta_{i-1}+\theta_i, \quad \theta_{i+1}' = \theta_{i+1} + \theta_i, \quad \theta'_j = \theta_j \quad (j \notin \{i-1,i,i+1\})\]
Fix $i \in \{0, \hdots, l-1\}$. Let $\theta \in \Q^l$ be such that $\theta_i \neq 0$. Choose $\nu \in \heartsuit(l)$. 
Let 
\begin{equation} \label{NRF definition} \mathfrak{R}_i : \mathcal{X}_\theta (n\delta + \mathbf{d}_\nu) \to \mathcal{X}_{\sigma_i\cdot\theta} (n\delta + \sigma_i * \mathbf{d}_\nu)\end{equation}
be the reflection functor associated to the simple reflection $\sigma_i \in \tilde{S}_l$. These functors were defined 
by Nakajima \cite[\S 3]{N} and Crawley-Boevey and Holland \cite[\S 2]{CBH}, \cite[\S 5]{CBH2}. 
One can endow the varieties $\mathcal{X}_\theta (n\delta + \mathbf{d}_\nu)$, $\mathcal{X}_{\sigma_i\cdot\theta} (n\delta + \sigma_i * \mathbf{d}_\nu)$ with hyper-K\"{a}hler structures with respect to which the reflection functor $\mathfrak{R}_i$ is a $U(1)$-equivariant hyper-K\"{a}hler isometry. 

\section{$\C^*$-fixed points in quiver varieties} \label{C-fp chapter}

In this section we explicitly construct the $\C^*$-fixed points in the quiver varieties $\mathcal{X}_\theta(n\delta+\mathbf{d}_\nu)$, assuming smoothness, as conjugacy classes of quadruples of certain matrices. Our description generalizes the work of Wilson, who classified the $\C^*$-fixed points in the special case $l=1$ in \cite[Proposition 6.11]{Wil}. Our construction depends on the Frobenius form of a partition. In \S \ref{eigdes} we define the matrices representing the fixed points in the special case when a partition consists of a single Frobenius hook. In \S \ref{Amu-sect} we define more general matrices for arbitrary partitions. In \S \ref{section-fp-proofs} we interpret our matrices as quiver representations and show that the corresponding orbits are in fact fixed under the $\C^*$-action. We finish by computing the character of the fibre of the tautological bundle at each fixed point. 

\subsection{The matrix $A(m,r)$.} \label{eigdes} 

Fix $\theta \in \Q^l$. 
The subscript in $\theta_i$ should always be considered modulo $l$. Suppose that $M$ is a matrix. Let $M_{ij}$ denote the entry of $M$ in the $i$-th row and $j$-th column.

\begin{definition} \label{alpha-defi}
Let $m \geq 1$ and $1 \leq r \leq m$. We let $\Lambda(m)$ denote the $m \times m$ matrix  with $1$'s on the first diagonal and all other entries equal to $0$.  Let $A(m,r)$ denote the $m \times m$ matrix whose only nonzero entries lie on the $(-1)$-st diagonal and satisfy
\[ A(m,r)_{j+1,j} = \left\{ {\begin{array}{c c c}
\displaystyle \sum_{i=1}^j \theta_{r-i} & \mbox{ if } & 1 \leq j < r \vspace{0.35 em} \\
\displaystyle -\sum_{i=0}^{m-j-1} \theta_{-m+r+i} & \mbox{ if } & r \leq j \leq m-1. 
\end{array} } \right.
\]
\end{definition}

\begin{lemma} \label{Alameig}
The matrix $[\Lambda(m), A(m,r)]$ is diagonal with eigenvalues 
\[ [\Lambda(m), A(m,r)]_{j,j} = \left\{ {\begin{array}{c c c}
\theta_{r-j} & \mbox{ if } & 1 \leq j \neq r \leq m  \vspace{0.35 em} \\
\displaystyle -\sum_{i=1}^{r-1}\theta_{r-i} - \sum_{i=0}^{m-r-1} \theta_{-m+r + i} & \mbox{ if } & j=r. \\
\end{array} } \right.
\]
\end{lemma}

\begin{proof}
Let $\alpha_j :=A(m,r)_{j+1,j}$. 
We have $\Lambda(m) A(m,r) = \diag(\alpha_1, \alpha_2, \hdots, \alpha_{m-1}, 0)$ and $A(m,r)\Lambda(m) = \diag(0, \alpha_1, \alpha_2, \hdots, \alpha_{m-1})$. Hence $[\Lambda(m), A(m,r)] = \diag(\alpha_1, \alpha_2 - \alpha_1, \hdots, \linebreak \alpha_{m-1} - \alpha_{m-2}, -\alpha_{m-1})$.
\end{proof}

\begin{example}
Let $l=3, m = 8, r = 5$. Then $A(m,r)$ is the following matrix
\[
\left( 
\scalemath{0.85}{
\begin{array}{cccccccc}
0 & 0 & 0 & 0 & 0 & 0 & 0 & 0 \\
\theta_1 & 0 & 0 & 0 & 0 & 0 & 0 & 0 \\
0 & \theta_1 + \theta_0 & 0 & 0 & 0 & 0 & 0 & 0 \\
0 & 0 & \theta_2 + \theta_1 + \theta_0 & 0 & 0 & 0 & 0 & 0 \\
0 & 0 & 0 & \theta_2 + 2\theta_1 + \theta_0 & 0 & 0 & 0 & 0 \\
0 & 0 & 0 & 0 & -\theta_2 - \theta_1 - \theta_0 & 0 & 0 & 0 \\
0 & 0 & 0 & 0 & 0 & -\theta_1 - \theta_0  & 0 & 0 \\
0 & 0 & 0 & 0 & 0 & 0 & - \theta_0 & 0 \\
\end{array} } \right)\]
\end{example}

\subsection{The matrix $A(\mu)$.} \label{Amu-sect}
Let $\nu \in \heartsuit(l)$ and $\mu \in \mathcal{P}_{\nu}(nl+ |\nu|)$. Let us write it in Frobenius form $\mu = (a_1, \hdots, a_k \mid b_1, \hdots, b_k)$ (see \S \ref{Frobenius form}). 
For each $1 \leq i \leq k$, let $r_i = b_i +1$,  $m_i = a_i + b_i +1$ and $\beta_i = \theta_0 + \sum_{i=1}^{r_i-1} \theta_{r_i-i} + \sum_{i=0}^{m-r_i-1} \theta_{-m+r_i + i}$. 

\begin{definition} \label{DefinitionofA-matrix}
We define $A(\mu)$ to be the matrix with diagonal blocks $A(\mu)^{ii} = A(m_i,r_i)$ and off-diagonal blocks $A(\mu)^{ij}$, where $A(\mu)^{ij}$ is the unique $m_i \times m_j$ matrix with nonzero entries only on the $(r_j - r_i -1)$-th diagonal satisfying
\begin{equation} \label{lambdaA} \Lambda({m_i}) A(\mu)^{ij} - A(\mu)^{ij}\Lambda({m_j}) = - \beta_i E(r_i,r_j),\end{equation}
where $E(r_i,r_j)$ is the $m_i \times m_j$ matrix with $E(r_i,r_j)_{s,t} = 0$ unless $s =r_i, t=r_j$ and $E(r_i,r_j)_{r_i,r_j}=1$.
\end{definition}
Explicitly, if $i>j$ then the non-zero diagonal of $A(\mu)^{ij}$ has $r_i$ entries equal to $\beta_i$ followed by $m_i - r_i$ entries equal to zero. If $i < j$ then the non-zero diagonal of $A(\mu)^{ij}$ has $r_j - 1$ entries equal to $0$ followed by $n_j - r_j +1$ entries equal to $- \beta_i$.
\begin{example}
Let $l=3$ and $\mu = (3,1 \mid 2,1)$. Then $m_1 = 6, m_2 = 3$ and $r_1 = 3, r_2 = 2$. Set $h = \theta_2 + \theta_1 + \theta_0$. Then $A(\mu)$ is the matrix
\[
\left( 
\scalemath{0.85}{
\begin{array}{cccccc|ccc}
0 & 0 & 0 & 0 & 0 & 0 & 0 & 0  & 0 \\
\theta_2 & 0 & 0 & 0 & 0 & 0 & 0 & 0  & 0\\
0 & \theta_2 + \theta_1 & 0 & 0 & 0 & 0 & 0 & 0 & 0\\
0 & 0 & -\theta_2 - \theta_1 - \theta_0 & 0 & 0 & 0 & 0 & -2h & 0\\
0 & 0 & 0 & -\theta_1 - \theta_0 & 0 & 0 & 0 & 0  & -2h\\
0 & 0 & 0 & 0 & -\theta_0 & 0 & 0 & 0  & 0\\ \hline
h & 0 & 0 & 0 & 0 & 0  & 0 & 0  & 0\\
0 & h & 0 & 0 & 0 & 0 & \theta_1 & 0  & 0\\
0 & 0 & 0 & 0 & 0 & 0 & 0 & - \theta_2  & 0\\
\end{array} } \right) \]
\end{example}
\begin{definition} \label{A-mu quadruple definition}
Let $\Lambda(\mu) = \bigoplus_{i=1}^k \Lambda({m_i})$. Setting $q_i = \sum_{s=1}^{i-1} m_s + r_i$, let $J(\mu)$ be the $nl \times 1$ matrix with entry $\beta_i$ in the $q_i$-th row (for $1 \leq i \leq k$) and all other entries zero. Furthermore, let $I(\mu)$ be the $1 \times nl$ matrix with entry $1$ in the $q_i$-th column (for $1 \leq i \leq k$) and all other entries zero. Finally, we set 
\[ \mathbf{A}(\mu) :=( \Lambda(\mu), A(\mu), I(\mu), J(\mu)).\]
\end{definition}

\subsection{The fixed points.} \label{section-fp-proofs}
Let us fix an $l$-core $\nu$ and a parameter $\theta \in \Q^l$ such that the variety $\mathcal{X}_{\theta}(n\delta + \mathbf{d}_{\tau})$ is smooth, where $\tau = \nu^t$. Let $\mathbf{d}_{\tau} := (d_0, \hdots, d_{l-1})$ and $\mathbf{d} = n\delta+\mathbf{d}_{\tau}$. Fix a complex vector space $\mathbf{V}_i^{\tau}$ of dimension $n+ d_i$ for each $i = 0, \hdots, l-1$. Additionally, let $\mathbf{V}_\infty$ be a complex vector space of dimension one. Set $\widehat{\mathbf{V}}^\tau = \bigoplus_{i=0}^{l-1} \mathbf{V}_i^\tau$ and $\mathbf{V}^\tau = \widehat{\mathbf{V}}^\tau \oplus \mathbf{V}_\infty$. 

We are now going to interpret $\mathbf{A}(\mu)$ as a quiver representation. With this goal in mind we choose a suitable ordered basis of the vector space $\mathbf{V}^\tau$. We show that the endomorphisms of $\mathbf{V}^\tau$ defined by $\mathbf{A}(\mu)$ with regard to this basis respect the quiver grading and thus constitute a quiver representation. We next show that this quiver representation lies in the fibre of the moment map at $\theta$. This allows us to conclude that the conjugacy class of $\mathbf{A}(\mu)$ is a point in the quiver variety $\mathcal{X}_\theta(\mathbf{d})$. We finish by showing that this point is fixed under the $\C^*$-action.

\begin{definition} 
Consider the sequence $\mathsf{Seq} := (1, \hdots, m_1, 1, \hdots, m_2, \hdots., 1, \hdots, m_k)$. We call each increasing subsequence of the form $(1, \hdots, m_i)$ the $p$-th \emph{block} of $\mathsf{Seq}$ and denote it by $\mathsf{Seq}_p$. Let $u_j$ be the $j$-th element in $\mathsf{Seq}$.  
Let $\zeta : \{1,\hdots,nl+|\nu|\} \to \{1,\hdots,k\}$ be the function given by the rule
 \[ \zeta(j) = p \iff u_j \in \mathsf{Seq}_p.\]
For each $1 \leq j \leq nl+|\nu|$ let 
\[\psi(j) = (r_{\zeta(j)} - u_j) \text{ mod } l.\] 
If $p,p'\in \N$, let $\delta(p,p') = 1$ if $p=p'$ and $\delta(p,p') = 0$ otherwise. For each $0 \leq i \leq l-1$ and $0 \leq j \leq nl+|\nu|$, let $\omega_i(j)$ be defined recursively by the formula 
\[ \omega_i(0) = 0, \quad \omega_i(j) = \omega_i(j-1) + \delta(\psi(j), i).\]
For each $0 \leq i \leq l-1$, fix a basis $\{v_i^1, \hdots, v_i^{n+d_i}\}$ of $\mathbf{V}_i^\tau$. We define a function
\[ \Bas \colon \{1, \hdots, nl+|\nu|\} \to \{v_i^{e_i} \mid 0 \leq i \leq l-1, \ 1 \leq e_i \leq n+d_i\}, \quad j \mapsto v_{\psi(j)}^{\omega_{\psi(j)}(j)}. \]
We also define a function $\Cell : \{1, \hdots, nl+|\nu|\} \to \mathbb{Y}(\mu^t)$ associating to a natural number $j$ a cell in the Young diagram of $\mu$. We define $\Cell(j)$ to be the $u_j$-th cell in the $\zeta(j)$-th Frobenius hook of $\mu^t$, counting from the hand of the hook, moving to the left towards the root of the hook and then down towards the foot. 
\end{definition}
\begin{lemma} \label{basbij} 
The functions $\Cell$ and $\Bas$ are bijections. 
\end{lemma}

\begin{proof}
The fact that $\Cell$ is a bijection follows directly from the definitions. 
Observe that $\psi(j)$ equals the $l$-content of $\Cell(j)$. We thus have a commutative diagram
\[
\xymatrixcolsep{1in}
\xymatrix{
\{1,\hdots,nl+|\nu|\} \ar[r]^-{\Bas} \ar[d]_{\Cell} & \{v_i^j \mid 0 \leq i \leq l-1, \ 1 \leq e_i \leq n+d_i\} \ar[d]^{ v_i^e \mapsto i} \\
\mathbb{Y}(\mu^t) \ar[r]_{l\textnormal{-content}} & \{0, \hdots, l-1\}.
}
\]
By \cite[Theorem 2.7.41]{JK}, the $l$-residue of $\mu^t$ equals $\sum_{i=0}^{l-1} (n+d_i)t^{i}$ because the $l$-core of $\mu^t$ is $\tau$. Hence for each $0 \leq i \leq l-1$ there are exactly $n+d_i$ elements $j \in \{1,\hdots,nl+|\nu|\}$ such that the $l$-content of $\Cell(j)$ equals $i$. By the commutativity of our diagram, we conclude that there are exactly $n+d_i$ elements $j \in \{1,\hdots,nl+|\nu|\}$ such that $\Bas(j) \in \mathbf{V}_i^\tau$. 

Now suppose that $j < j'$ and $\Bas(j), \Bas(j') \in \mathbf{V}_i^\tau$. Then $\psi(j) = \psi(j')$. Since $j < j'$ and the function $\omega_{\psi(j')}(-)$ is non-decreasing we have
$\omega_{\psi(j')}(j') = \omega_{\psi(j')}(j'-1) + 1 > \omega_{\psi(j')}(j'-1) \geq \omega_{\psi(j)}(j)$. Hence $\Bas(j) \neq \Bas(j')$. We conclude that the function $\Bas$ is injective. Since the domain and codomain have the same cardinality, $\Bas$ is also bijective. 
\end{proof} 

\begin{definition} 
Let $\mathbb{B} := (\Bas(1), \Bas(2), \hdots, \Bas(nl+|\tau|))$. By Lemma \ref{basbij}, $\mathbb{B}$ is an ordered basis of $\widehat{\mathbf{V}}^\tau$. From now on we consider the matrices $\Lambda(\mu)$ and $A(\mu)$ as linear endomorphisms of $\widehat{\mathbf{V}}^\tau$ relative to the ordered basis $\mathbb{B}$. Let us choose a nonzero vector $v_\infty \in \mathbf{V}_\infty$. We consider the matrix $I(\mu)$ as a linear transformation $\widehat{\mathbf{V}}^\tau \to \mathbf{V}_\infty$ relative to the ordered bases $\{v_\infty\}$ and $\mathbb{B}$. We also consider the matrix $J(\mu)$ as a linear transformation $\mathbf{V}_\infty \to \widehat{\mathbf{V}}^\tau$ relative to the ordered bases $\mathbb{B}$ and $\{v_\infty\}$.
\end{definition}

Let $\mu \in \mathcal{P}_{\nu}(nl+|\nu|)$. Suppose that $\mu = (a_1, \hdots, a_k \mid b_1, \hdots, b_k)$ is the Frobenius form of $\mu$. As before, set $r_i = b_i + 1$, $m_i = a_i + b_i +1$ and $q_i = \sum_{j<i} m_j + r_i$. 

\begin{lemma} \label{papryczka}
Suppose that $1 \leq i \leq k$. Then: 
\begin{itemize}
\item If $0\leq j < a_i$ then $A(\mu)(\Bas(q_i+j)) = \sum_{p=1}^{i} c_p \Bas(q_p+j+1)$, %c_1\Bas(q_1 + j+1)+c_2\Bas(q_2 + j +1) +\hdots+c_i\Bas(q_i + j + 1)$,
\item if $0\leq j = a_i$ then $A(\mu)(\Bas(q_i+j)) = \sum_{p=1}^{i-1} c_p \Bas(q_p+j+1)$, %c_1\Bas(q_1 + j+1)+c_2\Bas(q_2 + j +1)+\hdots+ c_{i-1}\Bas(q_{i-1} + j + 1)$,
\item if $0>j \geq -b_i$ then $A(\mu)(\Bas(q_i+j)) = \sum_{p=1}^{k} c_p \mathbf{1}_{-j \leq b_p+1} \Bas(q_p+j+1)$ %c_i\mathbf{1}_{-j \leq b_i+1}\Bas(q_i + j+1)+c_{i+1} \mathbf{1}_{-j \leq b_{i+1}+1}\Bas(q_{i+1} + j +1)+\hdots+c_k\mathbf{1}_{-j \leq b_k+1}\Bas(q_k+ j + 1)$
\end{itemize}
for some coefficients $c_m \in \C$, where
$\mathbf{1}_{-j \leq b_p+1}$ is the indicator function taking value one if $-j \leq b_p+1$ and zero otherwise. Moreover, for a generic parameter $\theta$ the coefficients $c_p$ are all non-zero.
\end{lemma}

\begin{proof}
This is immediate from Definition \ref{DefinitionofA-matrix}.
\end{proof}

Lemma \ref{papryczka} has a very intuitive diagrammatic interpretation. We explain it using the following example. 

\begin{example} \label{Examplekot}
Consider the partition $\mu = (5, 5, 4, 2)$. Its Frobenius form is $(4,3,1 \mid 3,2,0)$. We have $q_1 =4, q_2 = 11$ and $q_3 = 15$. The diagram below should be interpreted in the following way: $A(\mu)(\Bas(j))$ is a linear combination of those vectors $\Bas(j')$ for which there is an arrow $\Bas(j) \to \Bas(j')$. 
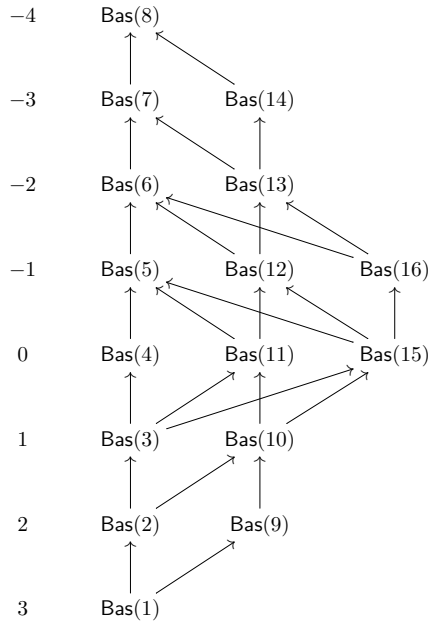
\[
\scalebox{.8}{
\xymatrix{
-4 & \Bas(8) \\
-3 & \Bas(7) \ar[u]& \Bas(14) \ar[ul]\\
-2 & \Bas(6) \ar[u]& \Bas(13) \ar[u] \ar[ul]\\
-1 & \Bas(5) \ar[u]& \Bas(12) \ar[u] \ar[ul]& \Bas(16) \ar[ul] \ar[ull]  \\
0 & \Bas(4) \ar[u] & \Bas(11) \ar[u] \ar[ul] & \Bas(15) \ar[u] \ar[ul] \ar[ull] \\
1 & \Bas(3) \ar[u] \ar[ur] \ar[urr]& \Bas(10) \ar[u] \ar[ur]\\
2 & \Bas(2) \ar[u] \ar[ur] & \Bas(9) \ar[u] \\
3 & \Bas(1) \ar[u] \ar[ur]
}}
\]
We have also introduced a numbering of the rows of the diagram. It is easy to see that $\Bas(j) \in \mathbf{V}_i^\tau$ if and only if $\Bas(j)$ lies in a row whose label is congruent to $i$ $\mmod$ $l$.
\end{example}

\begin{lemma}
Let $\mu \in \mathcal{P}_{\nu}(nl+|\nu|)$. Then $\mathbf{A}(\mu) \in \Rep(\overline{\mathbf{Q}}_\infty,\mathbf{d})$. 
\end{lemma}

\begin{proof}
We need to check that for each $0 \leq i \leq l-1$ the following holds:
\[   \Ima (A(\mu)|_{\mathbf{V}_i^\tau}) \subseteq \mathbf{V}_{i-1}^\tau,\quad \Ima \left( \Lambda(\mu)|_{\mathbf{V}_i^\tau}\right) \subseteq \mathbf{V}_{i+1}^\tau, \quad \Ima(J(\mu)) \subseteq \mathbf{V}_{0}^\tau, \quad \bigoplus_{i=1}^{l-1} \mathbf{V}_i^\tau \subseteq \ker(I(\mu)).\]
Let us first show the first statement. 
We can draw a diagram as in Example \ref{Examplekot}. 
The subspace $\mathbf{V}_{i}^\tau$ has a basis consisting of vectors $\Bas(j)$ in rows labelled by numbers congruent to $i$ $\mmod$ $l$. The diagram shows that $A(\mu)(\Bas(j))$ is a linear combination of basis vectors in the row above $\Bas(j)$. But that row is labelled by a number congruent to $i-1$ $\mmod$ $l$. Hence $A(\mu)(\Bas(j)) \in \mathbf{V}_{i-1}^\tau$. The argument for $\Lambda(\mu)$ is analogous.  

Let us prove the last claim. 
Let $j \in \{1, \hdots, nl+|\nu|\}$ and suppose that $\Bas(j) \notin \mathbf{V}_0^\tau$. Let $1 \leq p \leq k$ and set $q_p = \sum_{s=1}^{p-1} m_s + r_p$. Since $\psi(q_p) = r_p - r_p = 0$ we conclude that $p \notin \{q_1, \hdots, q_k\}$. But the only non-zero entries of $I(\mu)$ are those in columns numbered $q_p$, for $1 \leq p \leq k$. Hence $\Bas(j) \in \ker I(\mu)$. 
The calculation for $J(\mu)$ is similar. 
\end{proof}

\begin{proposition}
Let $\mu \in \mathcal{P}_{\nu}(nl+|\nu|)$. Then $\mathbf{A}(\mu) \in \mu^{-1}_{\mathbf{d}}(\theta)$. 
\end{proposition}

\begin{proof}
By the previous lemma, we know that $\mathbf{A}(\mu) \in \Rep(\overline{\mathbf{Q}}_\infty,\mathbf{d})$. Lemma \ref{Alameig} together with \eqref{lambdaA} implies that $[ \Lambda(\mu),A(\mu)] + J(\mu)I(\mu) = \theta$, so $\mathbf{A}(\mu) \in \mu^{-1}_{\mathbf{d}}(\theta)$. 
\end{proof}

\begin{theorem} \label{C fp classification theorem} 
Let $\mu \in \mathcal{P}_{\nu}(nl+|\nu|)$. Then $[\mathbf{A}(\mu)]:=G(\mathbf{d}).\mathbf{A}(\mu)$ 
is a $\C^*$-fixed point in the quiver variety $\mathcal{X}_\theta(\mathbf{d})$. 
\end{theorem}

\begin{proof}
Let $t \in \C^*$. We have $t.\mathbf{A}(\mu) = (t^{-1}\Lambda(\mu), tA(\mu), I(\mu), J(\mu))$. We need to find a matrix $N$ in $G(\mathbf{d})$ such that $Nt.\mathbf{A}(\mu)N^{-1} = \mathbf{A}(\mu)$.

For every $t \in \C^*$, let $Q(t) = \diag(1, t^{-1}, \hdots, t^{-nl-|\nu|+1})$. Conjugating an $(nl+|\nu|) \times (nl+|\nu|)$ matrix by $Q(t)$ multiplies the $j$-th diagonal by $t^{j}$. In particular, we have $Q(t)( \bigoplus_{i=1}^k tA(m_i,r_i))Q(t)^{-1} = \bigoplus_{i=1}^k A(m_i,r_i)$ and $Q(t)t^{-1}\Lambda(\mu)Q(t)^{-1} = \Lambda(\mu)$. 

Now consider the effect of conjugating $A(\mu)$ by $Q(t)$ on the off-diagonal block $A(\mu)^{ij}$ ($i \neq j$). This block contains only one nonzero diagonal. Counting within the block, it is the diagonal labelled $r_j - r_i - 1$. Counting inside the entire matrix $A(\mu)$, it is the diagonal labelled $q_j - q_i -1$. 
It follows that conjugation by $Q(t)$ multiplies the block $A(\mu)^{ij}$ by $t^{q_j - q_i - 1}$. Hence we have \[Q(t) \left( \bigoplus_{1 \leq i \neq j \leq k} tA(\mu)^{ij} \right) Q(t)^{-1} = \bigoplus_{1 \leq i \neq j \leq k} t^{q_j-q_i}A(\mu)^{ij}.\] 

Let $P(t) = \bigoplus_{i=1}^k t^{q_i}\Id_{m_i}$. Conjugating $A(\mu)$ by $P(t)$ doesn't change the diagonal blocks but multiplies each off-diagonal block $A(\mu)^{ij}$ by $t^{q_i-q_j}$. We conclude that \[P(t)Q(t) tA(\mu) Q(t)^{-1}P(t)^{-1} = A(\mu).\] Since the matrix $\Lambda(\mu)$ contains only diagonal blocks, conjugating by $P(t)$ doesn't have any impact. Hence \[P(t)Q(t) t^{-1} \Lambda(\mu) Q(t)^{-1}P(t)^{-1} = \Lambda(\mu).\]
The nonzero rows of $J(\mu)$ are precisely rows number $q_1, q_2, \hdots, q_k$. But the $q_i$-th entry of $P(t)$ is $t^{q_i}$ and the $q_i$-th entry of $Q(t)$ is $t^{1-q_i}$. Hence $P(t)Q(t)J(\mu) = tJ(\mu)$. Similarly, $I(\mu)q(t)^{-1}P(t)^{-1} = t^{-1} I(\mu)$. Let $D(t) = t^{-1} \Id_{nl+|\nu|}$. Since $D(t)$ is a scalar matrix, conjugating by $D(t)$ doesn't change $A(\mu)$ or $\Lambda(\mu)$. On the other hand, $D(t)P(t)Q(t)J(\mu) = J(\mu)$ and $I(\mu)q(t)^{-1}P(t)^{-1} D(t)^{-1}= I(\mu)$. 

The matrices $D(t),Q(t),P(t)$ are diagonal, so they represent linear automorphisms in $G(\mathbf{d})$. Hence $\mathbf{A}(\mu)$ and $t.\mathbf{A}(\mu)$ lie in the same $G(\mathbf{d})$-orbit, 
which is equivalent to saying that $\mathbf{A}(\mu)$ is a $\C^*$-fixed point in $\mathcal{X}_\theta(\mathbf{d})$.
\end{proof}

\subsection{Characters of the fibres of $\mathcal{V}$ at the fixed points.} 
Recall the tautological bundle $\mathcal{V}_\theta(\mathbf{d})$ on $\mathcal{X}_\theta(\mathbf{d})$ from \S \ref{Tautbundlesec}. 
Let us abbreviate $\mathcal{V}:=\mathcal{V}_\theta(\mathbf{d})$. 
Let $\mathcal{V}_\mu$ denote the fibre of $\mathcal{V}$ at the fixed point $[\mathbf{A}(\mu)]:=G(\mathbf{d}).\mathbf{A}(\mu)$.
\begin{proposition} \label{residue-character}
Let $\mu \in \mathcal{P}_{\nu}(nl+|\nu|)$. Then \[\ch_t \mathcal{V}_{\mu} = \Res_{\mu}(t) := \sum_{\square \in \mu} t^{c(\square)} .\] 
\end{proposition}

\begin{proof}
Consider $[(\mathbf{A}(\mu),v)] := G(\mathbf{d}).(\mathbf{A}(\mu),v) \in \mu_{\mathbf{d}}^{-1}(\theta) \times^{G(\mathbf{d})} \widehat{\mathbf{V}}^\tau = \mathcal{V}$. We have \begin{align*} t.(\mathbf{A}(\mu),v) = (t.\mathbf{A}(\mu),v) \sim& \ (D(t)P(t)Q(t)(t.\mathbf{A}(\mu))(D(t)P(t)Q(t))^{-1},(D(t)P(t)Q(t))^{-1}v) \\ =& \ (\mathbf{A}(\mu), Q(t)^{-1}P(t)^{-1}D(t)^{-1}v).\end{align*} The basis vectors $\{ \Bas(1), \Bas(2), \hdots, \Bas(nl+|\nu|)\}$ are eigenvectors of $(D(t)P(t)Q(t))^{-1}$ with corresponding eigenvalues \begin{equation*} \label{eigenvaluess} \{t^{1-r_1}, t^{2-r_1}, \hdots, t^{m_1 - r_1}; t^{1-r_2}, t^{2-r_2}, \hdots, t^{m_2 - r_2}; \hdots; t^{1-r_k}, t^{2-r_k}, \hdots, t^{m_k - r_k}\}.\end{equation*} 
Moreover, these eigenvalues are precisely the contents of the cells in the Young diagram of $\mu$, counting from the foot of the innermost Frobenius hook upward and later to the right, before passing to subsequent Frobenius hooks. 
Hence $\ch_t \mathcal{V}_{\mu} = \sum_{\square \in \mu} t^{c(\square)}$. 
\end{proof}
Recall that we have assumed that the parameter $\theta$ is chosen so that the variety $\mathcal{X}_\theta(\mathbf{d})$ is smooth. By \cite[Proposition 22]{Lam2} there exists $w \in \tilde{S}_l$ such that $w*\mathbf{d}_{\tau} = 0$. Let $w = \sigma_{i_1}\cdots \sigma_{i_m}$ be a reduced expression for $w$ in $\tilde{S}_l$. Furthermore, let $w\cdot \theta = (\vartheta_0,\hdots,\vartheta_{l-1})$ and $H_1 = \vartheta_1, \hdots, H_{l-1} = \vartheta_{l-1}$, $h = - \sum_{i=0}^{l-1} \vartheta_i$, $\mathbf{h} = (h, H_1, \hdots, H_{l-1})$. Composing the Etingof-Ginzburg map with reflection functors we obtain a $\C^*$-equivariant isomorphism
\[ \mathcal{Y}_{\mathbf{h}} \xrightarrow{\EG} \mathcal{C}_{\mathbf{h}} = \mathcal{X}_{w\cdot\theta}(n\delta) \xrightarrow{\mathfrak{R}_{i_m} \circ \cdots \circ \mathfrak{R}_{i_1}} \mathcal{X}_{\theta}(\mathbf{d}).\]
\begin{corollary} \label{v-core fp bij} 
The map \begin{equation} \label{partfp} \mathcal{P}_{\nu}(nl+|\nu|) \to \mathcal{X}_{\theta}(\mathbf{d})^{\C^*}, \quad \mu \mapsto [\mathbf{A}(\mu)] := G(\mathbf{d}).\mathbf{A}(\mu)\end{equation}
is a bijection. 
\end{corollary} 

\begin{proof}
The $\C^*$-fixed points in $\MaxSpec Z_{\mathbf{h}}$ are in bijection with $l$-multipartitions of $n$, which are themselves in bijection with partitions of $nl+|\nu|$ with $l$-core $\nu$. But $\mathcal{Y}_{\mathbf{h}}$ is $\C^*$-equivariantly isomorphic to $\mathcal{X}_{\theta}(\mathbf{d})$, so $|\mathcal{X}_{\theta}(\mathbf{d})^{\C^*}| = |(\MaxSpec Z_{\mathbf{h}})^{\C^*}|= |\mathcal{P}(l,n)|=|\mathcal{P}_{\nu}(nl+|\nu|)|$. 

Since a partition is uniquely determined by its residue, $\mu \neq \mu'$ implies $\ch_t \mathcal{V}_{\mu} \neq \ch_t \mathcal{V}_{\mu'}$,  which in turn implies that $[\mathbf{A}(\mu)] \neq [\mathbf{A}(\mu')]$. It follows that \eqref{partfp} is a bijection because it is an injective function between sets of the same cardinality.
\end{proof}

\section{Degenerate affine Hecke algebras} \label{dahas chapter}

In this section we use degenerate affine Hecke algebras and a version of the Chevalley restriction map to associate to each $\C^*$-fixed point in $\mathcal{Y}_{\mathbf{h}}$ and $\mathcal{C}_{\mathbf{h}}$ a distinct point in $\C^n/S_n$ in a manner which is compatible with the Etingof-Ginzburg isomorphism. 

\subsection{Degenerate affine Hecke algebras.}

Degenerate affine Hecke algebras associated to complex reflection groups of type $G(l,1,n)$ were defined in \cite{RS}. Let us recall their definition and basic properties. 

\begin{definition} \label{daha rels} Let $\kappa \in \C$. 
The degenerate affine Hecke algebra associated to $\Gamma_n$ is the $\C$-algebra $\mathcal{H}_{\kappa}$ generated by $\Gamma_n$ and pairwise commuting elements $z_1, \hdots , z_n$ satisfying the following relations:
\[ \epsilon_j z_i = z_i \epsilon_j \ (1 \leq i,j \leq n), \quad s_{i}z_j = z_j s_{i} \ (j \neq i, i+1), \]
\[ s_{i}z_{i+1} = z_i s_{i} + \kappa \sum_{k=0}^{l-1} \epsilon_i^{-k} \epsilon_{i+1}^k \ (1 \leq i \leq n-1).\]
Let $\mathcal{Z}_\kappa$ denote the centre of $\mathcal{H}_{\kappa}$. 
\end{definition}

\begin{proposition}
The algebra $\mathcal{H}_{\kappa}$ has the following properties.
\begin{enumerate}[label=\alph*), font=\textnormal,noitemsep,topsep=3pt,leftmargin=0.55cm]
\item As a vector space, $\mathcal{H}_{\kappa}$ is canonically isomorphic to $\C[z_1,\hdots,z_n] \otimes \C \Gamma_n$.  
\item There is an injective algebra homomorphism $\C[z_1,\hdots,z_n]^{S_n} \hookrightarrow \mathcal{Z}_\kappa.$
\item The algebra $\mathcal{H}_{\kappa}$ has a maximal commutative subalgebra $\mathfrak{C}_\kappa$ which is isomorphic to $\C[z_1, \hdots, z_n] \otimes \C (\Z/l\Z)^n$.
\item Suppose that $\mathbf{h} = (h, H_1, \hdots, H_{l-1})\in \Q^l$ satisfies $ h = \kappa$. Then there exists an injective algebra homomorphism $\mathcal{H}_{\kappa} \hookrightarrow \mathbb{H}_{\mathbf{h}}$ defined by
\begin{equation} \label{Dunkl-Opdam inclusion} g \mapsto g \ (g \in \Gamma_n), \quad z_i \mapsto y_ix_i + \kappa \sum_{1 \leq j < i} \sum_{k=0}^{l-1} s_{i,j} \epsilon_i^k  \epsilon_j^{-k} + \sum_{k=1}^{l-1} c_k \sum_{m=0}^{l-1} \eta^{-mk} \epsilon_i^m, \end{equation} 
where the $c_k$'s are the parameters obtained from $\mathbf{h}$ as in \cite[\S 2.7]{gor-qv}. 
This homomorphism restricts to a homomorphism 
\begin{equation} \label{invariants inclusion} \C[z_1,\hdots,z_n]^{S_n} \hookrightarrow Z_{\mathbf{h}}.\end{equation} 
\end{enumerate} 
\end{proposition}

\begin{proof}
See Propositions 1.1, 2.1, 2.3 and  \S3.1 in \cite{Dez} as well as Proposition 10.1 and Corollary 10.1 in \cite{Guay}.
\end{proof}

Let us recall the construction of some irreducible $\mathcal{H}_\kappa$-modules. 
\begin{definition}
Let $a = (a_1, \hdots, a_n) \in \C^n$ and $b = (b_1, \hdots, b_n) \in (\Z/l\Z)^n$. 
Let $\C_{a,b}$ be the one-dimensional representation of the commutative algebra $\mathfrak{C}_\kappa = \C[z_1, \hdots, z_n] \otimes \C (\Z/l\Z)^n$ defined by $z_i. v = a_i v$, $\epsilon_i. v = \eta^{b_i}v$ for each $1 \leq i \leq n$ and $v \in \C_{a,b}$. Define 
\[ M(a,b) := \mathcal{H}_\kappa \otimes_{\mathfrak{C}_\kappa} \C_{a,b}.\]
\end{definition}

\begin{proposition}[{\cite[Theorem 4.9]{Dez}}] \label{irred-daha}
Let $a \in \C^n$ and $b \in (\Z/l\Z)^n$.  If $a_i - a_j \neq 0, \pm l\kappa$ for all $1 \leq i \neq j \leq n$ then the $\mathcal{H}_\kappa$-module $M(a,b)$ is irreducible. 
\end{proposition}

\subsection{Restricting $\mathbb{H}_{\mathbf{h}}$-modules to $\mathcal{H}_h$-modules.}
Fix $\mathbf{h} \in \Q^l$ such that $\mathcal{Y}_\mathbf{h}$ is smooth and set $\kappa = h$. We are going to consider the generic behaviour of simple modules over $\mathbb{H}_{\mathbf{h}}$ under the restriction functor to $\mathcal{H}_h$-modules. 
\begin{definition} \label{defi: curly D}
Set
\[ \mathcal{D} := \{ a = (a_1, \hdots, a_n) \in \C^n \mid a_i - a_j \neq 0, \pm l \kappa \mbox{ for all } 1 \leq i \neq j \leq n\}.\]
Observe that $\mathcal{D}$ is a dense open subset of $\C^n$. Proposition \ref{irred-daha} implies that for all $a \in \mathcal{D}$ and $b \in (\Z/l\Z)^n$ the module $M(a, b)$ is irreducible. 
Consider the diagram
\[ \C^n \overset{\phi}{\longrightarrow} \C^n/S_n \overset{\rho_1}{\longleftarrow} \mathcal{Y}_{\mathbf{h}},\]
where $\phi$ is the canonical map and $\rho_1$ is the dominant morphism induced by \eqref{invariants inclusion}. 
Set 
\begin{equation} \label{calUset} \mathcal{U} := \rho^{-1}_1(\phi(\mathcal{D})). \end{equation} 
\end{definition}

\begin{lemma}
The subset $\mathcal{U}$ is open and dense in $\mathcal{Y}_{\mathbf{h}}$. 
\end{lemma}

\begin{proof}
The set $\mathcal{U}$ is open because $\phi$ is a quotient map and $\rho_1$ is continuous. 
Since the morphism $\rho_1$ is dominant, $\rho_1(\mathcal{Y}_{\mathbf{h}})$ is dense in $\C^n/S_n$. Therefore, since $\phi(\mathcal{D})$ is open in $\C^n/S_n$, we have $\phi(\mathcal{D}) \cap \rho_1(\mathcal{Y}_{\mathbf{h}}) \neq \varnothing$. Hence $\mathcal{U}$ is nonempty. The fact that the variety $\mathcal{Y}_{\mathbf{h}}$ is irreducible (see e.g.\ \cite[Corollary 3.9]{Eti}) now implies that $\mathcal{U}$ is dense. 
\end{proof}

Let $\hat{e} = \frac{1}{(n-1)!}\sum_{g \in S_{n-1} \subset \Gamma_n} g$ and $\mathbf{0} = (0,\hdots,0) \in (\Z/l\Z)^n$. For the rest of this subsection fix an irreducible $\mathbb{H}_{\mathbf{h}}$-module $L$ whose support is contained in $\mathcal{U}$ (i.e.\ $\chi_L \in \mathcal{U}$). Consider $L$ as an $\mathcal{H}_h$-module using the embedding \eqref{Dunkl-Opdam inclusion}. 

\begin{lemma} \label{LtoM}
There exists an injective homomorphism of $ \mathcal{H}_h$-modules $M(a,\mathbf{0}) \hookrightarrow L$ 
for some $a \in \mathcal{D}$. 
\end{lemma}

\begin{proof} 
We have a $(\Z/l\Z)^n$-module decomposition $L = \bigoplus_{b \in (\Z/l\Z)^n} L(b)$, where $L(b)$ is the subspace of $L$ such that $\epsilon_i.w = \eta^{b_i}w$ for all $w \in L(b)$. Since the $z_i$'s commute with the $\epsilon_j$'s, each subspace $L(b)$ is preserved under the action of the $z_i$'s. In particular, $z_1, \hdots, z_n$  
define commuting linear operators on $L(\mathbf{0})$, so they have some common eigenvector $v \in L(\mathbf{0})$. Let $a_1, \hdots, a_n$ be the respective eigenvalues of the $z_i$'s.
Since the support of $L$ is contained in $\mathcal{U}$, we have $a=(a_1,\hdots,a_n) \in \mathcal{D}$. 

Let $v_{a,\mathbf{0}} \in \C_{a,\mathbf{0}}$. Then the map $1 \otimes v_{a,\mathbf{0}} \mapsto v$ defines a non-zero $\mathcal{H}_h$-module homomorphism $M(a,\mathbf{0}) \to L$. Since $a = (a_1, \hdots, a_n) \in \mathcal{D}$, the module $M(a,\mathbf{0})$ is simple and so this homomorphism is injective. 
\end{proof}

Recall that $L^{\Gamma_{n-1}}$ is a module over $\langle \epsilon_1 \rangle \cong \Z/l\Z$. Let $L^{\Gamma_{n-1}}_{\chi_0}$ denote the isotypic component corresponding to the trivial character. 
 
\begin{lemma} \label{ehatM}
We have $\hat{e}M(a,\mathbf{0}) = L^{\Gamma_{n-1}}_{\chi_0}$. 
Moreover, $\hat{e}M(a,\mathbf{0})$ is stable under the action of $z_1$ and the eigenvalues of $z_1$ on $\hat{e}M(a,\mathbf{0})$ are $a_1, \hdots, a_n$. 
\end{lemma} 

\begin{proof}
The action of $\epsilon_1$ on $M(a,\mathbf{0})$ is trivial by definition. 
We have a vector space isomorphism $M(a,\mathbf{0}) \cong \C S_n \otimes \C_{a,\mathbf{0}}$. Therefore $\{ \hat{e}s_{1,j} \otimes v_{a,\mathbf{0}} \mid 1 \leq j \leq n \}$ form a basis of $\hat{e}M(a,\mathbf{0})$ for any nonzero $v_{a,\mathbf{0}} \in \C_{a,\mathbf{0}}$. In particular, $\dim \hat{e}M(a,\mathbf{0}) = n$. Let us show that each of the basis elements we defined is fixed under the action of $\Gamma_{n-1}$. We first note that since for each $g \in S_{n-1} \subset \Gamma_n$ we have $g\hat{e}=\hat{e}$, the subgroup $S_{n-1}$ fixes each $\hat{e}s_{1,j} \otimes v_{a,\mathbf{0}}$. Now consider $\epsilon_i .\hat{e}s_{1,j} \otimes v_{a,\mathbf{0}}$ with $2 \leq i \leq n$. We have $\epsilon_i .\hat{e}s_{1,j} \otimes v_{a,\mathbf{0}} = \sum_{g \in S_{n-1}} gs_{1,j}\epsilon_{i(g)} \otimes v_{a,\mathbf{0}}$, where $i(g)$ is an index depending on $g$. But each $\epsilon_{i(g)}$ acts on $v_{a,\mathbf{0}}$ by the identity, so we conclude that $\epsilon_i$ fixes $\hat{e}s_{1,j} \otimes v_{a,\mathbf{0}}$. The stability of $\hat{e}M(a,\mathbf{0})$ under the action of $z_1$ follows from the fact that $z_1$ commutes with $\hat{e}$. The calculation of the eigenvalues is similar to the calculation in the proof of \cite[Lemma 4.7]{Bel}. 
\end{proof}

\subsection{Connection to the Etingof-Ginzburg isomorphism.}
Suppose that $L$ is an irreducible $\mathbb{H}_{\mathbf{h}}$-module whose support is contained in $\mathcal{U}$. 
Let $M(a,\mathbf{0})$ be as in Lemma~\ref{LtoM}. 
Using Lemma \ref{ehatM} and \eqref{varpi} we can identify $\hat{e}M(a,\mathbf{0}) = L^{\Gamma_{n-1}}_{\chi_0} \cong (\C\Gamma_n)^{\Gamma_{n-1}} \cong \mathbf{V}_0$. 
Suppose that $\EG(\chi_L) = [(\mathbf{X},\mathbf{Y},I,J)]$. %(see \ref{subsec: cyclic quiver}. 
The embedding \eqref{Dunkl-Opdam inclusion} sends 
\[ z_1 \quad \mapsto \quad x_1y_1 + \sum_{k=1}^{l-1} c_k \sum_{m=0}^{l-1} \eta^{-mk} \epsilon_1^m.\] 
Since $\epsilon_1$ acts trivially on $M(a,\mathbf{0})$, the action of $z_1$ on $\hat{e}M(a,\mathbf{0})$ can be identified with the action of $y_1x_1$ on $L^{\Gamma_{n-1}}_{\chi_0}$. Using the Etingof-Ginzburg isomorphism, the latter can be identified, up to conjugation,  with the matrix $Y_1X_0$. 

\begin{definition}
Let $\rho_2 : \mathcal{C}_{\mathbf{h}} \to \C^n/S_n$ be the morphism sending $[(\mathbf{X},\mathbf{Y},I,J)]$ to the multiset of the generalized eigenvalues of the matrix $Y_1X_{0}$. 
\end{definition} 

\begin{proposition} \label{comm diagram lemma}
The following diagram commutes.
\[
\begin{tikzcd}
\mathcal{Y}_{\mathbf{h}} \arrow[rr,"\EG"] \arrow{dr}[swap]{\rho_1} &  & \mathcal{X}_{\theta}(n\delta) \arrow[ld, "\rho_2"] \\
 & \C^n/S_n & 
\end{tikzcd}
\]
\end{proposition}

\begin{proof}
Since $\EG$ is an isomorphism, it suffices to show there exists a dense open subset of $\mathcal{Y}_{\mathbf{h}}$ for which the diagram commutes. Consider the dense open subset $\mathcal{U}$ from \eqref{calUset}. Since $\mathcal{Y}_{\mathbf{h}}$ is smooth, for each $\chi \in \mathcal{U}$, there exists a unique simple $\mathbb{H}_{\mathbf{h}}$-module $L$ such that $\chi = \chi_{L}$. Moreover, there is an injective $\mathcal{H}_h$-module homomorphism $M(a,\mathbf{0})\hookrightarrow L$ for some $a \in \mathcal{D}$, by Lemma \ref{LtoM}. Set $[(\mathbf{X},\mathbf{Y},I,J)]:=\EG(\chi_L)$. 
The remarks at the beginning of this subsection imply that the matrix $Y_1X_0$ describes the action of $z_1$ on $\hat{e}M(a,\mathbf{0})$. Hence the eigenvalues of $Y_1X_0$ are the same as the eigenvalues of the operator $z_1|_{\hat{e}M(a,\mathbf{0})}$. By Lemma \ref{ehatM}, these eigenvalues are $a_1, \hdots, a_n$. Hence $\rho_2 \circ \EG (\chi_L) = \phi(a) \in \C^n/S_n$. 

On the other hand, consider the composition \begin{equation} \label{compemb} \C[z_1, \hdots, z_n]^{S_n} \hookrightarrow Z_{\mathbf{h}} \xrightarrow{\chi_L} \C.\end{equation} By the definition of $M(a,\mathbf{0})$, a symmetric polynomial $f(z_1, \hdots, z_n)$ acts on $1 \otimes \C_{a,\mathbf{0}}$ by the scalar $f(a_1, \hdots, a_n)$. Since $f(z_1, \hdots, z_n)$ is central in $\mathbb{H}_{\mathbf{h}}$, it acts by this scalar on all of $L$. Therefore, the kernel of \eqref{compemb} equals the maximal ideal in $\C[z_1, \hdots, z_n]^{S_n}$ consisting of those symmetric polynomials $f$ which satisfy $f(a_1, \hdots, a_n) = 0$, which is the vanishing ideal of $\phi(a)$. 
\end{proof}

\subsection{The images of the $\C^*$-fixed points in $\C^n/S_n$.} \label{subsec: images of fp} 
We are now going to identify the images of the $\C^*$-fixed points under $\rho_1$ and $\rho_2$. Set $\mathbf{e} = (e_0, \hdots, e_{l-1}) \in \Q^l$, where $e_0=0$ and $e_i = \sum_{j=1}^i H_j$ for $i = 1, \hdots, l-1$. Set $\theta:=\theta_{\mathbf{h}}$ as in \eqref{theta vs h}. For the rest of this section fix $\mu \in \mathcal{P}_{\varnothing}(nl)$. Let us identify a point $(a_1, \hdots, a_n) \in \C^n/S_n$ with the ``polynomial"~$\sum_{i=1}^n t^{a_i}$. 

\begin{lemma}[{\cite[\S 5.4]{Mar2}}] \label{Martino lemma}
Let $\underline{\lambda} \in \mathcal{P}(l,n)$. Then $\rho_1(\chi_{\underline{\lambda}}) = \Res^{\mathbf{e}}_ {\underline{\lambda}}(t^{h})$. 
\end{lemma}

\begin{definition} \label{Defi: alphas}
Let $\mu = (a_1, \hdots, a_k \mid b_1, \hdots, b_k)$ be the Frobenius form of $\mu$. For each $1 \leq i \leq k$, let $r_i = b_i +1$ and $m_i = a_i + b_i +1$. Recall the matrices $\Lambda(m_i)$ and $A(m_i,r_i)$ from Definition \ref{alpha-defi}. 
If $\Lambda(m_i) A(m_i,r_i) = \diag(\alpha_1, \alpha_2, \hdots, \alpha_{m_i-1}, \alpha_{m_i})$, then we define \[ \Eig(\mu,i) = \sum_{\substack{1 \leq j \leq m_i, \\ j = r_i - 1 \mmod l}} t^{\alpha_j}, \quad \Eig(\mu) = \sum_{i=1}^k \Eig(\mu,i).\] 
\end{definition}
\begin{lemma} \label{rho2-eig-calc b} We have $\rho_2([\mathbf{A}(\mu)]) = \Eig(\mu)$. 
\end{lemma} 
\begin{proof} 
$\Eig(\mu)$ picks out exactly the eigenvalues of the restricted endomorphism \linebreak $\Lambda(\mu)A(\mu)|_{\mathbf{V}_1}$ from all the eigenvalues of $\Lambda(\mu)A(\mu)$. But these are the same as the eigenvalues of $A(\mu)\Lambda(\mu)|_{\mathbf{V}_0}$. The fact that $ \rho_2([\mathbf{A}(\mu)]) = \Eig(\mu)$ now follows immediately from the definition of $\rho_2$. 
\end{proof} 

By Proposition \ref{comm diagram lemma}, Lemma \ref{Martino lemma}, Lemma \ref{rho2-eig-calc b} and the fact that a multipartition is uniquely determined by its $\mathbf{e}$-residue for generic $\mathbf{e}$, we have 
\begin{equation} \label{Eigrhorhores} \Eig(\mu) = \rho_2([\mathbf{A}(\mu)]) = \rho_1(\chi_{\underline{\lambda}}) = \Res^{\mathbf{e}}_ {\underline{\lambda}}(t^{h}) \end{equation} for unique $\underline{\lambda} \in \mathcal{P}(l,n)$.

\begin{definition} 
Define $\Eeig(\mu):= \left(\Eeigg(\mu)^0, \Eeigg(\mu)^1, \hdots, \Eeigg(\mu)^{l-1}\right) \in \mathcal{P}(l,n)$ by the equation
\[ \Res^{\mathbf{e}}_ {\Eeig(\mu)}(t^{h}) = \Eig(\mu).\]
\end{definition}
We thus have a bijection
\[ \Eeig : \mathcal{P}_{\varnothing}(nl) \to \mathcal{P}(l,n), \quad \mu \mapsto \Eeig(\mu).\] 
\begin{proposition} \label{Cor final2} Let $\mu \in \mathcal{P}_{\varnothing}(nl)$. 
The inverse of the Etingof-Ginzburg isomorphism sends the $\C^*$-fixed point $[\mathbf{A}(\mu)]$ in $\mathcal{C}_{\mathbf{h}}$ to the $\C^*$-fixed point $\chi_{\Eeig(\mu)}$ in $\mathcal{Y}_{\mathbf{h}}$.
\end{proposition}

\begin{proof}
This follows directly from \eqref{Eigrhorhores}. 
\end{proof}

\subsection{Calculation of $\Eig(\mu)$.} \label{subsec: calc of Eig} 
Set $\mathbf{e'} = (e'_0, e'_1, \hdots, e'_{l-1})$ and $\mathbf{e''} = (e''_0, e''_1\hdots, e''_{l-1})$ with 
$e'_0 = -h$, $e''_{l-1}=0$ and 
\[ e_i' = e_i \quad (i = 1, \hdots, l-1), \quad \quad e_i'' = h + \sum_{j=1}^{l-i-1} H_j \quad (i = 0, \hdots, l-2).\]
In these notations all the lower indices are to be considered $\mmod$ $l$.

\begin{lemma} \label{rho2-eig-calc} Let $\mu = (a_1, \hdots, a_k \mid b_1, \hdots, b_k) \in \mathcal{P}_{\varnothing}(nl)$. Then 
\begin{equation} \label{Eig:mu} \Eig(\mu) = \sum_{i=1}^k \left( \left(t^{e'_{b_i}} \sum_{j=1}^{\lceil b_i /l \rceil} t^{-(j-1)h}\right) + \left( t^{e''_{a_i }} \sum_{j=1}^{\lfloor(a_i +1)/l\rfloor} t^{(j-1)h} \right) \right).\end{equation}
\end{lemma} 

\begin{proof} 
It suffices to show that for each $i=1,\hdots,k$ we have
\begin{equation} \label{Eig:mu,i} \Eig(\mu,i) = \left(t^{e'_{b_i}} \sum_{j=1}^{\lceil b_i /l \rceil} t^{-(j-1)h}\right) + \left( t^{e''_{a_i}} \sum_{j=1}^{\lfloor(a_i +1)/l\rfloor} t^{(j-1)h} \right).\end{equation}
We can write 
\begin{equation} \label{Eig two sums} \Eig(\mu,i) := \sum_{\substack{1 \leq j \leq m_i, \\ j = r_i - 1 \mmod l}} t^{\alpha_j} = \sum_{\substack{1 \leq j \leq r_i-1, \\ j = r_i - 1 \mmod l}} t^{\alpha_j} + \sum_{\substack{r_i \leq j \leq m_i, \\ j = r_i - 1 \mmod l}} t^{\alpha_j}.\end{equation} 
Note that $r_i - 1 = b_i = l \cdot (\lceil b_i/l \rceil-1) + d_i$, where $d_i$ is an integer such that $1 \leq d_i \leq l$. The $j's$ satisfying $1 \leq j \leq r_i -1$ and $j=r_i-1 \mmod l$ are therefore precisely $b_i, b_i -l, b_i - 2l, \hdots, b_i - (\lceil b_i/l \rceil - 1) \cdot l=d_i$. Recall that $\theta_0 + \theta_1 + \hdots + \theta_{l-1} = -h$. Hence $\alpha_{d_i + pl} = \alpha_{d_i} - ph$ for $p = 0, \hdots, \lceil b_i/l \rceil - 1$, by Definition \ref{alpha-defi}. Therefore \[ \sum_{\substack{1 \leq j \leq r_i-1, \\ j = r_i - 1 \mmod l}} t^{\alpha_j} = t^{\alpha_{d_i}}\sum_{j=1}^{\lceil b_i /l \rceil} t^{-(j-1)h}.\]
Observe that $d_i = b_i \mmod l$ if $1 \leq d_i < l$. Hence $\alpha_{d_i}  = \sum_{j=1}^{d_i} \theta_{b_i +1 - j} = \sum_{j=1}^{d_i} \theta_{d_i +1 - j} =  \sum_{j=1}^{d_i} \theta_j = e'_{d_i} = e'_{b_i}$. If $d_i = l$ then $\alpha_{d_i} = \alpha_l = \sum_{j=1}^l \theta_{b_i +1 - j} = \sum_{j=1}^l \theta_j = -h = e'_0$. This shows that
\[ \sum_{\substack{1 \leq j \leq r_i-1, \\ j = r_i - 1 \mmod l}} t^{\alpha_j} = t^{e'_{b_i}}\sum_{j=1}^{\lceil b_i /l \rceil} t^{-(j-1)h}.\]

Let us now consider the second sum on the RHS os \eqref{Eig two sums}. 
We have $m_i - r_i + 1 = a_i + 1 = l \cdot \lfloor (a_i +1)/l\rfloor + c_i$ with $0 \leq c_i < l$. The $j's$ satisfying $r_i \leq j \leq m_i$ and $j=r_i-1 \mmod l$ are therefore precisely $b_i +l, b_i + 2l, \hdots, b_i + \lfloor(a_i+1)/l \rfloor \cdot l$. Note that $b_i +  \lfloor (a_i+1)/l \rfloor \cdot l = m_i - c_i$. Hence $\alpha_{m_i - c_i - pl} = \alpha_{m_i - c_i} + ph$ for $p=0, \hdots, \lfloor (a_i+1)/l \rfloor -1$. One computes, in a similar fashion as above, that $\alpha_{m_i - c_i} = e''_{a_i}$. This shows that 
\[ \sum_{\substack{r_i \leq j \leq m_i, \\ j = r_i - 1 \mmod l}} t^{\alpha_j} = t^{e''_{a_i}} \sum_{j=1}^{\lfloor(a_i +1)/l\rfloor} t^{(j-1)h}.\qedhere \]
\end{proof}

\section{$\C^*$-fixed points under the Etingof-Ginzburg isomorphism} \label{lquo chapter}

We will now identify the multipartition $\Eeig(\mu)$ and thereby establish the correspondence between the $\C^*$-fixed points under the Etingof-Ginzburg isomorphism. 

\subsection{The strategy.} 
Our next goal is to show that $\Eeig(\mu) = \quot(\mu)^\flat$. We will use the following strategy. Recall Lemma \ref{rem hooks quot}. We first prove that an analogous statement holds for the multipartition $\Eeig(\mu)$. This will allow us to argue by induction on $n$. We then prove that $\Eeig(\mu) = \quot(\mu)^\flat$ for partitions $\mu$ with the special property that only a unique $l$-rim-hook can be removed from $\mu$. We then deduce the result for arbitrary $\mu \in \mathcal{P}_{\varnothing}(nl)$.

\subsection{Types and contributions of Frobenius hooks.}
We need to introduce some notation to break down formula \eqref{Eig:mu} into simpler pieces. 
Throughout this section fix $\mu \in \mathcal{P}_{\varnothing}(nl)$. 

\begin{definition}
Let $\mu = (a_1, \hdots, a_k \mid b_1, \hdots, b_k)$ be the Frobenius form of $\mu$. Let $F_1, \hdots, F_k$ be the Frobenius hooks in $\mathbb{Y}(\mu)$ so that $(i,i)$ is the root of $F_i$. Let \[ \type_{\mu}(L,i) := b_i \mmod l, \quad \type_{\mu}(A,i) := -(a_i+1) \mmod l \quad (i = 1, \hdots, k).\]
We call the number $\type_{\mu}(L,i)$ the \emph{type} of $\leg(F_i)$ and the number $\type_{\mu}(A,i)$ the \emph{type} of $\arm(F_i)$. 
Let $\mathbf{e}$, $\mathbf{e}'$ and $\mathbf{e}''$ be as in \S \ref{subsec: images of fp}  and \S \ref{subsec: calc of Eig}. 
Define \[ \Xi_{\mu}(L,i) := t^{e'_{b_i}} \sum_{j=1}^{\lceil b_i /l \rceil} t^{-(j-1)h}, \quad \Xi_{\mu}(A,i) := t^{e''_{a_i}} \sum_{j=1}^{\lfloor(a_i +1)/l\rfloor} t^{(j-1)h}.\]
We call $\Xi_{\mu}(L,i)$ the \emph{contribution} of $\leg(F_i)$ and $\Xi_{\mu}(A,i)$ the \emph{contribution} of $\arm(F_i)$. 
\end{definition} 

By \eqref{Eig:mu,i} we have \begin{equation} \label {Eig-Xi-form} \Eig(\mu,i) = \Xi_{\mu}(L,i) + \Xi_{\mu}(A,i). \end{equation}
\begin{lemma} \label{lemma eig res form}

We have 
\begin{equation} \label{type-res-xi}
t^{e_j} \Res_{\Eeigg(\mu)^j}(t^h) = \sum_{\substack{1 \leq i \leq k, \\ \type_\mu(A,i) = j}} \Xi_{\mu}(A,i) \ + \sum_{\substack{1 \leq i \leq k, \\ \type_\mu(L,i)=j}} \Xi_{\mu}(L,i).\end{equation}
\end{lemma}

\begin{proof}
Each summand $t^d$ on the RHS of \eqref{Eig:mu} corresponds (non-canonically) to a cell in the multipartition $\Eeig(\mu)$ in the sense that it describes that cell's $\mathbf{e}$-shifted content.  
For generic $\mathbf{e}$ we can write $t^d = t^{e_i} t^c$ for unique $i = 0, \hdots, l-1$ and $c \in \Z$. The summand $t^d$ corresponds to a cell in the partition $\Eeigg(\mu)^j$ if and only if $i = j$, i.e., $t^d = t^{e_j} t^c$.
Since $\Eig(\mu) = \sum_{p=1}^k \Eig(\mu,p)$, formula \eqref{Eig-Xi-form} implies that there exists an $1 \leq p \leq k$ such that $t^d$ is a summand in $\Xi_{\mu}(L,p)$ or $\Xi_{\mu}(A,p)$. In the former case $t^d = t^{e_j} t^c$ if and only if $j = b_p \mmod l = \type_{\mu}(L,p)$. In the latter case $t^d = t^{e_j} t^c$ if and only if $e''_{a_p} = h + e_j$, which is the case if and only if $j = - (a_p +1) \mmod l = \type_{\mu}(A,p)$. 
\end{proof}

\subsection{Removal of rim-hooks.}
We will now investigate the effect of removing a rim-hook from $\mu$ on the multipartition $\Eeig(\mu)$.
Let $\mathbb{Y}_{+/-/0}(\mu)$ denote the subset of $\mathbb{Y}(\mu)$ consisting of cells of positive/negative/zero content. 

\begin{lemma} \label{lem to rimrem}
Let $R$ be an $l$-rim-hook in $\mathbb{Y}(\mu)$ and suppose that $R \subset (\mathbb{Y}_0(\mu) \cup \mathbb{Y}_+(\mu))$. Suppose that $R$ intersects $r$ Frobenius hooks, labelled $F_{p+1}, \hdots, F_{p+r}$ so that $(i,i)$ is the root of $F_i$. Let $\mu':= \mu - R$. Then 
\[ \type_{\mu'}(A,j) = \type_{\mu}(A,j+1), \quad \Xi_{\mu'}(A,j) = \Xi_{\mu}(A,j+1) \quad (j = p+1, \hdots, p+r-1),   \]
\[ \type_{\mu'}(A,p+r) = \type_{\mu}(A,p+1),\quad \Xi_{\mu'}(A,p+r) = \Xi_{\mu}(A,p+1) - M,  \]
where $M$ is the (monic) monomial in $\Xi_{\mu}(A,p)$ of highest degree. 
\end{lemma}

\begin{proof}
It is clear that $R$ must intersect adjacent Frobenius hooks. Recall that 
the residue of $R$ is of the form $\Res_R(t) = \sum_{i=i_0}^{i=i_0 + l -1} t^i$ with $i_0 \geq 0$. Moreover, we have 
\begin{equation} \label{Residue-rim-hook} \Res_{R\cap F_j}(t) = \sum_{i= i_{r+p-j}}^{i_{r+p-j+1}-1} t^i \end{equation}
for some integers $i_0 < i_1 < \hdots < i_r = i_0 +l$. One can easily see that these integers satisfy
\begin{equation} \label{ci} i_{r+p-j+1} - 1 = a_j,\end{equation}
where $a_j = |\arm(F_j)| = \max_{\square \in F_j} c(\square)$. Set $d_j := \max_{\square \in F_j - R} c(\square)$. If $F_j - R = \varnothing$ set $d_j = -1$. From \eqref{Residue-rim-hook} and \eqref{ci} we easily deduce that \begin{equation} \label{ad eq} d_j = a_{j+1} \quad (j = p+1, \hdots, p+r-1), \quad d_{p+r} = a_{p+1} - l.\end{equation} 

By definition, the type and contribution of $\arm(F_j)$ resp. $\arm(F_j-R)$ depend only on the numbers $a_j$ and $d_j$. The lemma now follows immediately from the definitions. 
\end{proof}

One can easily formulate a version of Lemma \ref{lem to rimrem} for $R \subset (\mathbb{Y}_0(\mu) \cup \mathbb{Y}_-(\mu))$. The proof is completely analogous. 
Lemma \ref{lem to rimrem} admits the following graphical interpretation. 

\begin{example}
The figure on the left shows the Young diagram of the partition $(5,5,4,3,3$). The cells of content zero are marked as green. The blue cells form a $4$-rim-hook. The figure on the right shows the same Young diagram rearranged so that cells of the same content occupy the same row. 
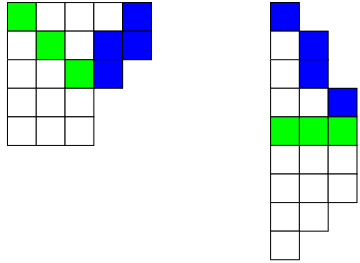
\[
\scalebox{.7}{
\ytableausetup{nosmalltableaux}
\begin{ytableau}
*(green) & *(white) & *(white) & *(white) & *(blue) \\
*(white) & *(green) & *(white) & *(blue) & *(blue) \\
*(white) & *(white) & *(green) & *(blue) \\
*(white) & *(white) & *(white) \\
*(white) & *(white) & *(white) \\
\end{ytableau} \quad \quad \quad \quad \quad \quad  
\begin{ytableau}
*(blue) \\
*(white) & *(blue) \\
*(white) & *(blue) \\
*(white) & *(white) & *(blue) \\
*(green) & *(green) & *(green) \\
*(white) & *(white) & *(white) \\
*(white) & *(white) & *(white) \\
*(white) & *(white) \\
*(white) \\
\end{ytableau}}
\]
Using this visual representation we can easily determine the impact of removing the blue rim-hook. The length of the first arm after the removal equals the length of the second arm before the removal. Similarly, the length of the second arm after the removal equals the length of the third arm before the removal. Finally, the length of the third arm after the removal equals the length of the first arm before the removal minus four. This is precisely the content of Lemma \ref{lem to rimrem}. 
\end{example}

\begin{proposition} \label{Eig rem hook pro}
Let $R$ be a rim-hook in $\mu$ and set $\mu' := \mu - R$. Then $\Eeig(\mu') = \Eeig(\mu) - \blacksquare$ for some $\blacksquare \in \Eeig(\mu)$. 
\end{proposition}

\begin{proof}
There are three possibilities: $R \subset \mathbb{Y}_0(\mu) \cup \mathbb{Y}_+(\mu), R \subset \mathbb{Y}_0(\mu) \cup \mathbb{Y}_-(\mu)$ or $R \cap \mathbb{Y}_+(\mu) \neq \varnothing, R \cap \mathbb{Y}_-(\mu) \neq \varnothing$. 
Consider the first case. Lemma \ref{lem to rimrem} and Lemma \ref{lemma eig res form} imply that there exists a $j \in \{0, \hdots,l-1\}$ such that $\Eeigg(\mu')^i = \Eeigg(\mu)^i$ if $i \neq j$ and $t^{e_j} \Res_{\Eeigg(\mu')^j}(t^h) = t^{e_j} \Res_{\Eeigg(\mu)^j}(t^h) - t^{s_j}M$ for some monic monomial $M = t^{qh}\in \Z[t^{h}]$. Hence $\Eeig(\mu') = \Eeig(\mu) - \blacksquare$ for some $\blacksquare \in \Eeig(\mu)$ with $c(\blacksquare) = q$.

The second case is analogous. Now consider the third case. We claim that $\Xi_\mu(A,i) = 0$ for every Frobenius hook $F_i$ whose arm intersects $R$ nontrivially. Indeed, by definition $\Xi_\mu(A,i) \neq 0$ only if $|\arm(F_i)| + 1\geq l$. We have $|\arm(F_i)| = \max_{\square \in \arm(F_i)} c(\square) = \max_{\square \in \arm(F_i)\cap R} c(\square)$. Hence $|\arm(F_i)| \leq \max_{\square \in R} c(\square)$. However, since $R \cap \mathbb{Y}_-(\mu) \neq \varnothing$, the rim-hook $R$ must contain a cell of content $-1$. The fact that $\Res_R(t) = t^q\sum_{p=0}^{l-1} t^p$ for some $q \in \Z$ implies that $\max_{\square \in R} c(\square) \leq l-2$. Hence $|\arm(F_i)| + 1\leq l-1$ and so $\Xi(A,i) = 0$. Therefore the removal of $R$ does not affect the contribution of the arm of any Frobenius hook. 

Now set $R' := R \cap (\mathbb{Y}_0(\mu) \cup \mathbb{Y}_-(\mu))$. We have reduced the third case back to the second case, with the modification that $R'$ is now a truncated rim-hook. We can still apply Lemma \ref{lem to rimrem} with minor adjustments. In particular, equations \eqref {ad eq} are still true with the exception that the final equation becomes $d_{p+r} = 0$. Let $j$ be the smallest integer such that $\leg(F_j) \cap R \neq \varnothing$. Using the same argument as before, we conclude that $t^{e_j} \Res_{\Eeigg(\mu')^j}(t^h) = t^{e_j} \Res_{\Eeigg(\mu)^j}(t^h) - t^{e_j-h}$ and $\Eeigg(\mu')^i = \Eeigg(\mu)^i$ if $i \neq j$. 
\end{proof}

\subsection{Partitions with a unique removable rim-hook.}
In this section we show that $\Eeig(\mu) = \quot(\mu)^\flat$ for a certain class of partitions which we call $l$-\emph{special}.
\begin{definition}
We say that a partition $\mu$ is $l$-\emph{special} if the rim of $\mathbb{Y}(\mu)$ contains a unique $l$-rim-hook $R$. We call $R$ the \emph{unique removable} $l$\emph{-rim-hook} in $\mathbb{Y}(\mu)$. Let $\mathcal{P}_{\varnothing}^{sp}(k)$ denote the set of partitions of $k$ which are $l$-special and have a trivial $l$-core. 
\end{definition}

Our goal now is to describe partitions of $nl$ which are $l$-special and have a trivial $l$-core. Throughout this subsection we assume that $\mu \in \mathcal{P}_{\varnothing}^{sp}(nl)$. 
We let $R$ denote the unique removable $l$-rim-hook in $\mathbb{Y}(\mu)$ and set $\mu':= \mu - R$. Sometimes, for the sake of brevity, we will just write "rim-hook"  instead of $l$-rim-hook. 

\begin{lemma}  \label{comb L1}
Let $\mu \in \mathcal{P}_{\varnothing}^{sp}(nl)$. Then: 
\begin{enumerate}[label=\alph*), font=\textnormal,noitemsep,topsep=3pt,leftmargin=0.55cm]
\item Every column of $\mathbb{B}(\mu)$ contains the same number of beads. 
\item Sliding distinct beads up results in the removal of distinct $l$-rim-hooks from $\mathbb{Y}(\mu)$.
\item The bead diagram $\mathbb{B}(\mu)$ contains $l-1$ columns with no gaps and one column with a unique string of adjacent gaps.
\end{enumerate}
\end{lemma}

\begin{proof}
(a) The empty bead diagram describes the trivial partition. But bead diagrams which describe the trivial partition and have the property that the number of beads in the diagram is divisible by $l$ are unique up to adding or deleting full rows at the top of the diagram. Hence any such diagram consists of consecutive full rows at the top. Since $\mu$ has a trivial $l$-core, the process of sliding beads upward in $\mathbb{B}(\mu)$ must result in a bead diagram of this shape. But this is only possible if every column of $\mathbb{B}(\mu)$ contains the same number of beads.\\
(b) We can see this by considering the quotients of partitions corresponding to the bead diagrams obtained by moving up distinct beads. If the beads moved are on distinct runners, then a box is removed from distinct partitions in $\quot(\mu)$, so distinct multipartitions arise. If the beads are on the same runner, sliding upward distinct beads implies changing the first-column hook lengths in different ways in the same partition, so different multipartitions arise as well. But a trivial-core partition is uniquely determined by its quotient, so these distinct multipartitions are quotients of distinct partitions of $l(n-1)$. \\
(c) Since only one rim-hook can be removed from $\mu$, only one bead in our bead diagram can be moved upward. This implies that $l-1$ runners contain no gaps (i.e. they contain a consecutive string of beads counting from the top). 
The remaining runner must contain a unique gap or a unique string of gaps.
\end{proof}

\begin{lemma} \label{comb L2}
The bead diagram $\mathbb{B}(\mu)$ can be decomposed into three blocks $A$, $B$ and $C$, counting from the top. Each block consists of identical rows. Rows in block $A$ are full except for one bead. Let's say that the gap due to the absent bead is on runner $k$. Rows in block $B$ are either all full or all empty. Rows in block $C$ are empty except for one bead on runner $k$. Moreover, the number of rows in block $A$ equals the number of rows in block~$C$. 
\end{lemma}
\begin{proof}
This is an immediate consequence of Lemma \ref{comb L1}.
\end{proof}
\begin{example} \label{two bead diagrams example} 
Let $l=3$. According to Lemma \ref{comb L2} the following bead diagrams correspond to $3$-special partitions:
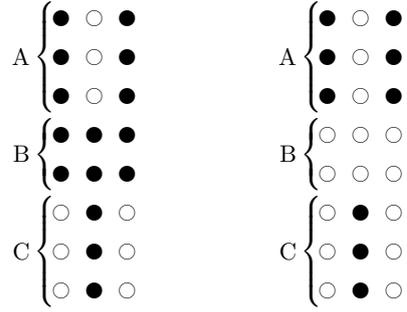
\[
\scalebox{.9}{
\xy
\xymatrixrowsep{0.05in}
\xymatrixcolsep{0in}
\xymatrix "M"{
\CIRCLE & \Circle & \CIRCLE \\
\CIRCLE & \Circle & \CIRCLE \\
\CIRCLE & \Circle & \CIRCLE \\
\CIRCLE & \CIRCLE & \CIRCLE \\
\CIRCLE & \CIRCLE & \CIRCLE \\
\Circle & \CIRCLE & \Circle \\
\Circle & \CIRCLE & \Circle \\
\Circle & \CIRCLE & \Circle \\
}
\POS"M1,1"."M3,1"!C*\frm{\{},+L*++!R\txt{A},
\POS"M4,1"."M5,1"!C*\frm{\{},+L*++!R\txt{B},
\POS"M6,1"."M8,1"!C*\frm{\{},+L*++!R\txt{C} 
\endxy \quad \quad \quad \quad \quad 
\xy
\xymatrixrowsep{0.05in}
\xymatrixcolsep{0in}
\xymatrix "N"{
\CIRCLE & \Circle & \CIRCLE \\
\CIRCLE & \Circle & \CIRCLE \\
\CIRCLE & \Circle & \CIRCLE \\
\Circle & \Circle & \Circle \\
\Circle & \Circle & \Circle \\
\Circle & \CIRCLE & \Circle \\
\Circle & \CIRCLE & \Circle \\
\Circle & \CIRCLE & \Circle \\
}
\POS"N1,1"."N3,1"!C*\frm{\{},+L*++!R\txt{A},
\POS"N4,1"."N5,1"!C*\frm{\{},+L*++!R\txt{B},
\POS"N6,1"."N8,1"!C*\frm{\{},+L*++!R\txt{C}
\endxy }
\]
The bead diagram on the left describes the partition $(8,6,4,3^7,2^2,1^2)$ while the bead diagram on the right describes the partition $(14,12,10,3,2^2,1^2)$. 
\end{example}
In the sequel we will only consider the case where all the rows in block $B$ are full.
All the following claims can easily be adapted to the case of empty rows. 

\begin{lemma}  \label{comb L3} 
Suppose that block $A$ of $\mathbb{B}(\mu)$ has $m$ rows and block $B$ of $\mathbb{B}(\mu)$ has $p$ rows. Then $\mathbb{Y}(\mu)$ can be decomposed into four blocks $\hat{A}, \hat{B}, \hat{C}, \hat{D}$:
\begin{itemize}[,leftmargin=0.75cm]
\item $\hat{A}$ is the Young diagram of the partition corresponding to the bead diagram $A$.  
\item $\hat{B}$ is a rectangle consisting of $m$ columns and $l \cdot p$ rows. 
\item If $k \neq 0$ block $\hat{C}$ is the Young diagram of the partition corresponding to the bead diagram $C$. If $k=0$ block $\hat{C}$ is the the Young diagram of the partition corresponding to the bead diagram obtained from $C$ by inserting an extra row at the top, which is full except for the empty point in column $l-1$. 
\item Block $\hat{D}$ is a square with $m$ rows and columns. \end{itemize} 
We recover $\mathbb{Y}(\mu)$ from these blocks by placing $\hat{A}$ at the bottom, stacking $\hat{B}$ on top, then stacking $\hat{D}$ on top and finally placing $\hat{C}$ on the right hand side of $\hat{D}$. 
\end{lemma}

\begin{proof}
This follows from Lemma \ref{comb L2} by a routine calculation - one merely has to recover the first column hook lengths from the positions of the beads. 
\end{proof}

\begin{example}
Consider the first bead diagram from Example \ref{two bead diagrams example}. Below we illustrate the block decomposition of the corresponding Young diagram as in Lemma \ref{comb L3}. 
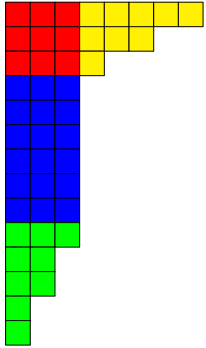
\[
\scalebox{.6}{
\ytableausetup{nosmalltableaux}
\begin{ytableau}
*(red) & *(red)  &*(red) & *(yellow) & *(yellow)  &*(yellow) & *(yellow)  &*(yellow) \\
*(red) & *(red) & *(red) & *(yellow) & *(yellow) & *(yellow)\\
*(red) & *(red) & *(red) & *(yellow) \\
*(blue) & *(blue) & *(blue) \\
*(blue) & *(blue) & *(blue) \\
*(blue) & *(blue) & *(blue) \\
*(blue) & *(blue) & *(blue) \\
*(blue) & *(blue) & *(blue) \\
*(blue) & *(blue) & *(blue) \\
*(green) & *(green) & *(green) \\
*(green) & *(green)  \\
*(green) & *(green)  \\
*(green) \\
*(green) 
\end{ytableau}}
\]
\fcolorbox{black}{green}{\rule{0pt}{2pt}\rule{2pt}{0pt}}\quad \small Block $\hat{A}$ \quad \fcolorbox{black}{blue}{\rule{0pt}{2pt}\rule{2pt}{0pt}}\quad Block $\hat{B}$ \quad
\fcolorbox{black}{yellow}{\rule{0pt}{2pt}\rule{2pt}{0pt}}\quad \small Block $\hat{C}$ \quad \fcolorbox{black}{red}{\rule{0pt}{2pt}\rule{2pt}{0pt}}\quad Block $\hat{D}$
\end{example} 

We are now ready to investigate the effect of removing the rim-hook $R$. 

\begin{lemma}
The partition $\mu$ can be decomposed into $m$ Frobenius hooks. The unique removable rim-hook lies on the outermost Frobenius hook.
\end{lemma}

\begin{proof}
By Lemma \ref{comb L3}, the $0$-th diagonal of the Young diagram of $\mu$ is contained in $\hat{D}$ and contains $m$ boxes. Hence there are $m$ Frobenius hooks. It follows easily from Lemma \ref{comb L2} and Lemma \ref{comb L3} that the outermost Frobenius hook $F_m$ in $\mu$ is the partition of the form $(k+1,1^{l\cdot p + l -k-1})$. In particular, it contains $l \cdot (p+1) \geq l$ cells. Let us first consider $F_m$ as a partition in its own right and check whether it contains an $l$-rim-hook. If $p=0$ then $F_m$ is itself a rim-hook. If $p>1$ then the subset $1^{l\cdot p+l-k-1}$ of $F_m$ contains a rim-hook. Hence, in either case, $F_m$ contains a rim-hook. But a rim-hook of the outermost Frobenius hook $F_m$ is also a rim-hook of $\mu$ (because the outermost Frobenius hook is part of the rim). 
\end{proof}

\begin{proposition} \label{comb P4} We have
$\Eeig(\mu') = \Eeig(\mu) - \blacksquare$, where $\blacksquare$ is the unique removable cell in $\Eeigg(\mu)^{l-k-1}$ with content $-p$. 
\end{proposition}
\begin{proof}
The outermost Frobenius hook $F_m$  in $\mu$ is the partition of the form $(k+1,1^{l\cdot p + l -k-1})$. 
By removing the rim-hook $R$ we obtain the partition $(k+1,1^{l\cdot p -k -1})$ if $p \geq 1$ 
or the trivial partition if $p=0$. Since the rim-hook $R$ is contained in the outermost Frobenius hook, its removal does not affect the type and contribution of the arms and legs of the other Frobenius hooks. 

There are several cases to be considered. Let $a_m = |\arm(F_m)|$ and $b_m = |\leg(F_m)|$. If $k\neq l-1$, then $\Xi_{\mu}(L,m) = t^{e'_{b_m \mmod l}} \sum_{i=0}^{p} t^{-ih}$ and $\type_{\mu}(L,m) = b_m = l-k-1$ (while $\Xi_{\mu}(A,m)=0$).
If $k=l-1$ and $p > 0$ then $\Xi_{\mu}(L,m) = \sum_{i=1}^{p} t^{-ih}$ and $\type_{\mu}(L,m) = 0$ (while $\Xi_{\mu}(A,m) = 1$).
If $k=l-1, p = 0$ then $\Xi_{\mu}(A,m) = 1$ and $\type_{\mu}(A,m) = 0$ (while $\Xi_{\mu}(L,m)=0$).

Upon removing the rim-hook the polynomials listed above change as follows.
We have $\Xi_{\mu'}(L,m)$ $= t^{e'_{b_m \mmod l}} \sum_{i=0}^{p-1} t^{-ih}$ in the first case, $\Xi_{\mu'}(L,m) = \sum_{i=1}^{p-1} t^{-ih}$ in the second case and $\Xi_{\mu'}(A,m) = 0$ in third case. The types do not change. We observe that in each case a monomial of degree $t^{-ph}$ (up to a shift) is subtracted, which corresponds to removing a cell of content $-p$ in $\Eeigg(\mu)^{l-k-1}$.
\end{proof}

We obtain an analogous result for the multipartition $\quot(\mu)$. 

\begin{lemma} \label{comb L5}
The following hold: 
\begin{enumerate}[label=\alph*), font=\textnormal,noitemsep,topsep=3pt,leftmargin=0.55cm]
\item $\quot(\mu) = (Q^0(\mu), \hdots, Q^{l-1}(\mu))$ is a multipartition consisting of $l-1$ trivial partitions and one non-trivial partition. 
\item Suppose that the $k$-th column in $\mathbb{B}(\mu)$ is the unique column which contains gaps. Then $Q^k(\mu)$ is the unique non-trivial partition in $\quot(\mu)$. If that column has a string of $m$ gaps followed by a string of $q=m+p$ beads then the Young diagram of $Q^k(\mu)$ is a rectangle consisting of $m$ columns and $q$ rows. 
\end{enumerate} 
\end{lemma}

\begin{proof}
The $l$-quotient of $\mu$ can be deduced directly from the bead diagram $\mathbb{B}(\mu)$. The description of the latter in Lemma \ref{comb L2} immediately implies the present lemma. 
\end{proof}

Recall that  $\mu' := \mu - R$, where $R$ is the unique rim-hook which can be removed from~$\mu$.

\begin{lemma} \label{comb L6}
We have $\quot(\mu') = \quot(\mu) - \square$, where $\square$ is the box in the bottom right corner of the rectangle described in Lemma \ref{comb L5}. That box has content $-p$. 
\end{lemma}

\begin{proof}
This is the only cell which can be removed from $\quot(\mu)$ by Lemma \ref{comb L5}, so the claim now follows from Lemma \ref{rem hooks quot}. 
\end{proof}

The lemma implies in particular that $\quot(\mu')^\flat = \quot(\mu)^\flat - \square$, where $\square$ is a box of content $-p$ in the $(l-k-1)$-th partition in $\quot(\mu)^\flat$. 

\subsection{Induction.} We will now use induction on $n$ to show that $\quot(\mu)^\flat = \Eeig(\mu)$. There are two cases to be considered: $\mu \in \mathcal{P}_{\varnothing}^{sp}(nl)$ and $\mu \notin \mathcal{P}_{\varnothing}^{sp}(nl)$. 

\begin{proposition} \label{comb P7} Suppose that $\quot(\lambda)^\flat = \Eeig(\lambda)$ for all partitions $\lambda \vdash l(n-1)$ with trivial $l$-core. Let $\mu \in \mathcal{P}_{\varnothing}^{sp}(nl)$.
Then $\quot(\mu)^\flat = \Eeig(\mu)$.
\end{proposition}

\begin{proof}
By induction, $\quot(\mu')^\flat = \Eeig(\mu')$. But by Proposition \ref{comb P4} and Lemma \ref{comb L6}, both $\quot(\mu)^\flat$ and $\Eeig(\mu)$ arise from $\quot(\mu')^\flat = \Eeig(\mu')$ by adding a box of content $-p$ to the $(l-k-1)$-th partition. Hence $\quot(\mu)^\flat = \Eeig(\mu)$. 
\end{proof}

\begin{proposition} \label{comb T8} Suppose that $\quot(\lambda)^\flat = \Eeig(\lambda)$ for all partitions  $\lambda \vdash l(n-1)$ with trivial $l$-core. Let $\mu \notin \mathcal{P}_{\varnothing}^{sp}(nl)$. 
 Then $\quot(\mu)^\flat = \Eeig(\mu)$. 
\end{proposition}

\begin{proof}
Since $\mu \notin \mathcal{P}_{\varnothing}^{sp}(nl)$ we can remove two distinct (but possibly overlapping) rim-hooks $R'$ and $R''$ from $\mu$. 
Let $\mu' = \mu - R'$ and $\mu'' = \mu - R''$. Then $\mu' \neq \mu''$ and so $\quot(\mu') \neq \quot(\mu'')$ (because the quotient of a partition with trivial core determines that partition uniquely). By Lemma \ref{rem hooks quot}, we have 
\[ \quot(\mu')^\flat = \quot(\mu)^\flat - \square, \quad \quot(\mu'')^\flat= \quot(\mu)^\flat - \hat{\square}\]
with $\square \neq \hat{\square} \in \quot(\mu)^\flat$. By Proposition \ref{Eig rem hook pro}, we have 
\[ \Eeig(\mu') = \Eeig(\mu) - \blacksquare, \quad \Eeig(\mu'') = \Eeig(\mu) - \hat{\blacksquare}\]
for some $\blacksquare, \hat{\blacksquare} \in \Eeig(\mu)$. We know that $\Eeig$ establishes a bijection between $l$-partitions of $n-1$ and partitions of $l(n-1)$ with a trivial $l$-core. Hence $\blacksquare \neq \hat{\blacksquare}$. 
By the inductive hypothesis in our lemma,
\[ \quot(\mu')^\flat = \Eeig(\mu'), \quad \quot(\mu'')^\flat = \Eeig(\mu'').\]
Hence
\[ \quot(\mu)^\flat - \square =  \Eeig(\mu) - \blacksquare, \quad \quot(\mu)^\flat - \hat{\square} =  \Eeig(\mu) - \hat{\blacksquare}\]
and so
\[ \Eeig(\mu) = \quot(\mu)^\flat - \square + \blacksquare = \quot(\mu)^\flat - \hat{\square} + \hat{\blacksquare}.\]
Since $\square \neq \hat{\square}$ and $\blacksquare \neq \hat{\blacksquare}$ we conclude that $\square = \blacksquare$ and $\hat{\square} = \hat{\blacksquare}$. Therefore $\Eeig(\mu) = \quot(\mu)^\flat$. 
\end{proof}

\begin{theorem} \label{Cor final} \label{Cor final final}
Let $\mu \in \mathcal{P}_{\varnothing}(nl)$. Then $\quot(\mu)^\flat = \Eeig(\mu)$. The bijection between the labelling sets of $\C^*$-fixed points induced by the Etingof-Ginzburg isomorphism is given by
\[ \mathcal{P}(l,n) \to \mathcal{P}_\varnothing(nl), \quad \quot(\mu)^\flat \mapsto \mu.\]
\end{theorem}

\begin{proof}
The first claim follows directly from Propositions \ref{comb P7} and \ref{comb T8}. The second claim follows from Proposition \ref{Cor final2}. 
\end{proof}

\begin{remark}
As a corollary, we also obtain the following explicit formula for the residue of the $l$-quotient of a partition $\mu = (a_1, \hdots, a_k \mid b_1, \hdots, b_k) \in \mathcal{P}_\varnothing(nl)$. Let us write $\quot(\mu) = (Q^0, \hdots, Q^{l-1})$. Then 
\[ \Res_{Q^{l-j-1}}(t) =  \sum_{\substack{1 \leq i \leq k, \\ -(a_i+1) = j\mmod l}} t^{p_{a_i}} \sum_{m=1}^{\lfloor(a_i +1)/l\rfloor} t^{(m-1)} \ + \sum_{\substack{1 \leq i \leq k, \\ b_i =j \mmod l}}  t^{p'_{b_i}} \sum_{m=1}^{\lceil b_i /l \rceil} t^{-(m-1)},\]
where $p_{i} = 1$ for $i= 0, \hdots, l-2$ and $p_{l-1}=0$ while $p'_{0} = -1$ and $p'_i = 0$ for $i=1,\hdots,l-1$. Indeed, the RHS of the formula above equals $\Res_{\Eeigg(\mu)^j}(t)$ by \eqref{type-res-xi}. But $\Eeigg(\mu)^j = Q^{l-j-1}$ by Theorem \ref{Cor final}. 
\end{remark}

\section{The higher level $q$-hook formula}

We will now use Theorem \ref{Cor final final} to obtain the following ``higher-level" generalization of the $q$-hook formula. 

\begin{theorem} \label{cyclotomic q-hook formula}
Let $\mu \in \mathcal{P}_{\varnothing}(nl)$. 
Then 
\begin{equation} \label{HL cyclo q hook} \sum_{\square \in \mu} t^{c(\square)} = [nl]_t \sum_{\underline{\lambda} \uparrow \quot(\mu)^\flat}\frac{ f_{\underline{\lambda}}(t)}{f_{\quot(\mu)^\flat}(t)}.\end{equation}
\end{theorem}

Our proof is based on comparing the $\C^*$-characters of the vector bundles $\mathcal{V}_{\mathbf{h}}$ and $\mathcal{R}^{\Gamma_{n-1}}_{\mathbf{h}}$, and uses Proposition \ref{residue-character}, Theorem \ref{Cor final final} as well as the Etingof-Ginzburg isomorphism. The remaining ingredient is a calculation of the $\C^*$-characters of the fibres $(\mathcal{R}^{\Gamma_{n-1}}_{\mathbf{h}})_{\underline{\lambda}}$ (see \S \ref{RCA: taut v bundle}). 
We carry out this calculation below. 
We first identify $(\mathcal{R}^{\Gamma_{n-1}}_{\mathbf{h}})_{\underline{\lambda}}$ with a graded shift of $e_{n-1}L(\underline{\lambda})$.  We next recall the graded multiplicity with which $L(\underline{\lambda})$ occurs in $\Delta(\underline{\lambda})$ and calculate the character of $e_{n-1}\Delta(\underline{\lambda})$. Finally, we use the equation 
\[ \ch_t e_{n-1}L(\underline{\lambda}) = \frac{\ch_t e_{n-1}\Delta(\underline{\lambda})}{[\Delta(\underline{\lambda}):L(\underline{\lambda})]_{\mathsf{gr}}}.\]

Note that, setting $l=1$ in \eqref{HL cyclo q hook}, we recover the usual $q$-hook formula, so our result also gives a new geometric proof of this well-known combinatorial identity. 
\subsection{Coinvariant algebras and fake degree polynomials.} \label{FelineOrigin22}

The algebra $\C[\h]^{co\Gamma_n}$ is a graded $\Gamma_n$-module. It is well known that $\C[\h]^{co\Gamma_n}$ is isomorphic to the regular representation $\C \Gamma_n$ as an ungraded $\Gamma_n$-module. 

Let $\C[\h]^{co\Gamma_n}|_{\Gamma_{n-1}}$ denote the restriction of $\C[\h]^{co\Gamma_n}$ to a $\Gamma_{n-1}$-module. 
Let $\h' \subset \h$ denote the subspace spanned by $y_2, \hdots, y_n$. We choose a splitting $\h \twoheadrightarrow \h'$ with kernel spanned by $y_1$. This splitting induces an inclusion $\C[\h'] \subset \C[\h]$.

\begin{lemma} \label{coinv iso}
We have an isomorphism of graded $\Gamma_{n-1}$-modules
\[ \C[\h]^{co\Gamma_n}|_{\Gamma_{n-1}} \cong \C[\h']^{co\Gamma_{n-1}} \otimes U,\]
where $U$ is a graded vector space with Poincar\'{e} polynomial $\ch_t  U = [nl]_t$. 
\end{lemma}

\begin{proof}
We have a sequence of inclusions of graded $\Gamma_{n-1}$-modules
\[ \C[\h]^{\Gamma_n} \hookrightarrow \C[\h]^{\Gamma_{n-1}} \hookrightarrow \C[\h]\]
such that each ring is a free graded module over the previous ring. Hence there is an isomorphism of graded $\Gamma_{n-1}$-modules
\[ \C[\h]/\langle\C[\h]^{\Gamma_n}_+\rangle \cong \C[\h]/\langle\C[\h]^{\Gamma_{n-1}}_+\rangle \otimes \C[\h]^{\Gamma_{n-1}}/\langle\C[\h]^{\Gamma_n}_+\rangle.\]
Observe that there is also an isomorphism of graded $\Gamma_{n-1}$-modules \[ \C[\h]/\langle\C[\h]^{\Gamma_{n-1}}_+\rangle \cong \C[\h']/\langle\C[\h']^{\Gamma_{n-1}}_+\rangle = \C[\h']^{co\Gamma_{n-1}}.\] To prove the lemma it now suffices to find the Poincar\'{e} polynomial of the graded vector space $\C[\h]^{\Gamma_{n-1}}/\langle\C[\h]^{\Gamma_n}_+\rangle$. We know that $\C[\h]^{\Gamma_{n-1}}$ is a polynomial algebra with generators in degrees $l,2l, \hdots, (n-1)l$ and an additional generator in degree $1$. The ring $\C[\h]^{\Gamma_{n}}$ is a polynomial algebra with generators in degrees $l, 2l, \hdots, nl$. Hence \[\ch_t \C[\h]^{\Gamma_{n-1}} = \frac{1}{1-t}\prod_{i=1}^{n-1} \frac{1}{1-t^{il}}, \quad \C[\h]^{\Gamma_{n}} = \prod_{i=1}^{n} \frac{1}{1-t^{il}}.\]
It follows that $\ch_t \C[\h]^{\Gamma_{n-1}}/\langle\C[\h]^{\Gamma_n}_+\rangle = \frac{\ch_t \C[\h]^{\Gamma_{n-1}}}{\ch_t\C[\h]^{\Gamma_n}} = \frac{1-t^{nl}}{1-t}=[nl]_t.$
\end{proof}

\begin{definition}
Suppose that we are given an $l$-multipartition $\underline{\lambda} \in \mathcal{P}(l,n)$ and the corresponding irreducible representation $S(\underline{\lambda})$ of $\Gamma_{n}$. 
We regard $S(\underline{\lambda})$ as a graded $\Gamma_n$-module concentrated in degree zero.
The \emph{fake degree polynomial} associated to $\underline{\lambda}$ is defined as 
\[ f_{\underline{\lambda}}(t) := \sum_{k \in \Z} [\C[\h]^{co\Gamma_n} : S(\underline{\lambda})^*[k]] t^k.\]
\end{definition}

\begin{theorem}[{\cite[Theorem 5.3]{Ste}}] \label{Schurfake}
Let $\underline{\lambda} \in \mathcal{P}(l,n)$. We have 
\[ f_{\underline{\lambda}}(t)  = t^{r(\underline{\lambda})} (t^l)_n \prod_{i=0}^{l-1} \frac{t^{l \cdot n(\lambda^i)}}{H_{\lambda^i}(t^l)} = t^{r(\underline{\lambda})} (t^l)_n \prod_{i=0}^{l-1} s_{\lambda^i}(1,t^l,t^{2l}, \hdots ) .\]
In particular, if $\lambda$ is a partition of $n$ then $f_\lambda = (t)_n \frac{t^{n(\lambda)}}{H_\lambda(t)} =  (t)_n s_\lambda(1,t,t^2, \hdots )$. 
\end{theorem}

\subsection{Auxiliary calculations.}
Fix $\underline{\lambda} \in \mathcal{P}(l,n)$. 
Let $q(\underline{\lambda})$ denote the degree in which the trivial $\Gamma_n$-module $\triv$ occurs in $L(\underline{\lambda})$. 
\begin{lemma} \label{fibre char}
We have a graded $\mathbb{H}_{\mathbf{h}}$-module isomorphism \[\mathcal{R}_{\mathbf{h},\underline{\lambda}}=\mathbb{H}_{\mathbf{h}}e_n \otimes_{e_n\mathbb{H}_{\mathbf{h}}e_n} e_nL_{\underline{\lambda}} \cong L(\underline{\lambda})[-q(\underline{\lambda})]\] 
and hence  a graded vector space isomorphism
\[ (\mathcal{R}^{\Gamma_{n-1}}_{\mathbf{h}})_{\underline{\lambda}} = e_{n-1}\mathcal{R}_{\mathbf{h},\underline{\lambda}} \cong e_{n-1}  L(\underline{\lambda})[-q(\underline{\lambda})].\]
\end{lemma}
\begin{proof}
As ungraded $\mathbb{H}_{\mathbf{h}}$-modules, $\mathcal{R}_{\mathbf{h},\underline{\lambda}}$ and $L(\underline{\lambda})$ are clearly isomorphic. Since they are simple, one is a graded shift of the other. 
The trivial $\Gamma_n$-representation $\triv$ occurs in $L(\underline{\lambda})$ in degree $q(\underline{\lambda})$. On the other hand, we can identify $\triv$ with the subspace $e_n \otimes_{e_n\mathbb{H}_{\mathbf{h}}e_n}e_nL_{\underline{\lambda}}$ of $\mathcal{R}_{\mathbf{h},\underline{\lambda}}$, so $\triv$ occurs in $\mathcal{R}_{\mathbf{h},\underline{\lambda}}$ in degree zero. 
\end{proof}

Let us calculate the graded multiplicity of $L(\underline{\lambda})$ in $\Delta(\underline{\lambda})$. 

\begin{lemma} \label{delta-l} Let $\underline{\lambda} \in \mathcal{P}(l,n)$.
The simple $\overline{\mathbb{H}}_{\mathbf{h}}$-module $L(\underline{\lambda})$ occurs in $\Delta(\underline{\lambda})$ with graded multiplicity 
\[ \sum_{k \in \Z} [\Delta(\underline{\lambda}) : L(\underline{\lambda})[k]] t^k = t^{-q(\underline{\lambda})}f_{\underline{\lambda}}(t).\] 
\end{lemma}   

\begin{proof}
By \cite[Lemma 3.3]{Bel2}, we have 
\[ \ch_t L(\underline{\lambda}) = \frac{t^{q(\underline{\lambda})} \ch_t \Delta(\underline{\lambda})}{f_{\underline{\lambda}}(t)}. \] 
Hence
\[ \sum_{k \in \Z} [\Delta(\underline{\lambda}) : L(\underline{\lambda})[k]] t^k = \frac{\ch_t \Delta(\underline{\lambda})}{\ch_t L(\underline{\lambda})} = t^{-q(\underline{\lambda})}f_{\underline{\lambda}}(t). \qedhere\] 
\end{proof}

Let us calculate the character of $e_{n-1} \Delta(\underline{\lambda})$. 

\begin{lemma} \label{del-fake}
We have 
\begin{equation} \label{eq: ch eDelta} \ch_t e_{n-1} \Delta(\underline{\lambda}) = [ln]_t \sum_{\underline{\mu} \uparrow \underline{\lambda}} f_{\underline{\mu}}(t).
\end{equation} 
\end{lemma}

\begin{proof}
By Lemma \ref{coinv iso} and Proposition \ref{branching}, we have isomorphisms of graded $\Gamma_{n-1}$-modules
\begin{align*}
\Delta(\underline{\lambda})|_{\Gamma_{n-1}} \cong& \ \C[\h]^{co\Gamma_n}|_{\Gamma_{n-1}} \otimes S(\underline{\lambda})|_{\Gamma_{n-1}} \\
 \cong& \ \left(\C[\h']^{co\Gamma_{n-1}} \otimes U \right) \otimes \bigoplus_{\underline{\mu} \uparrow \underline{\lambda}} S(\underline{\mu})  
 \cong \  \bigoplus_{\underline{\mu} \uparrow \underline{\lambda}} \Delta(\underline{\mu}) \otimes U,
\end{align*}
where $U$ is a graded vector space with character $\ch_t U = [ln]_t$. Hence
\begin{equation} \label{chdelta1} e_{n-1} \Delta(\underline{\lambda}) \cong \bigoplus_{\underline{\mu} \uparrow \underline{\lambda}} e_{n-1}\Delta(\underline{\mu}) \otimes U\end{equation} as graded $\Gamma_{n-1}$-modules. 
For each $\underline{\mu}\uparrow \underline{\lambda}$, we have 
\begin{equation} \label{chdelta2} \ch_t e_{n-1}\Delta \left( \underline{\mu} \right) = \sum_{k \in \Z} [ \Delta(\underline{\mu}) : \triv[k] ] t^k = f_{\underline{\mu}}(t).\end{equation}
The first equality above is obvious, for the second see, e.g., the proof of \cite[Theorem 5.6]{gor-bvm}.   
Combining \eqref{chdelta1} with \eqref{chdelta2} we obtain \eqref{eq: ch eDelta}. 
\end{proof}

\subsection{The character of the fibre.}

We can now put our calculations together to obtain the character of $(\mathcal{R}^{\Gamma_{n-1}}_{\mathbf{h}})_{\underline{\lambda}}$. 

\begin{theorem} \label{RCA fibre calc} 
Let $\underline{\lambda} \in \mathcal{P}(l,n)$. Then 
\[  \ch_t  (\mathcal{R}^{\Gamma_{n-1}}_{\mathbf{h}})_{\underline{\lambda}} = [ln]_t \sum_{\underline{\mu} \uparrow \underline{\lambda}}\frac{ f_{\underline{\mu}}(t)}{f_{\underline{\lambda}}(t)}.\]
\end{theorem}

\begin{proof}
By Lemmas \ref{fibre char}, \ref{delta-l} and \ref{del-fake}, we have 
\begin{align*} 
 \ch_t  (\mathcal{R}^{\Gamma_{n-1}}_{\mathbf{h}})_{\underline{\lambda}} =& \ t^{-q(\underline{\lambda})} \cdot \ch_t e_{n-1}L(\underline{\lambda}) = (t^{-q(\underline{\lambda})} \cdot \ch_t e_{n-1}\Delta(\underline{\lambda}) ) / (  t^{-q(\underline{\lambda})}f_{\underline{\lambda}}(t)) \\ =& \ \ch_t  e_{n-1}\Delta(\underline{\lambda}) / f_{\underline{\lambda}}(t) = [ln]_t \sum_{\underline{\mu} \uparrow \underline{\lambda}}\frac{ f_{\underline{\mu}}(t)}{f_{\underline{\lambda}}(t)}.\qedhere \end{align*}
\end{proof}

\begin{corollary}
We have 
\begin{align*}  \ch_t  (\mathcal{R}^{\Gamma_{n-1}}_{\mathbf{h}})_{\underline{\lambda}}  = \frac{1}{1-t} \sum_{i=0}^{l-1}t^{-i} \sum_{\substack{\underline{\mu} \uparrow \underline{\lambda}, \\ \mu^i \neq \lambda^i}} \frac{s_{\mu^{i}}(1, t^l, t^{2l}, \hdots)}{s_{\lambda^i}(1, t^l, t^{2l}, \hdots)}
= \frac{1}{1-t} \sum_{i=0}^{l-1}t^{-i} \sum_{\substack{\underline{\mu} \uparrow \underline{\lambda}, \\ \mu^i \neq \lambda^i}} \frac{t^{l \cdot n(\mu^i)}H_{\lambda^i}( t^l)}{t^{l \cdot n(\lambda^i)}H_{\mu^i}(t^l)}.
\end{align*}
In particular, if $l=1$ then $\ch_t  (\mathcal{R}^{\Gamma_{n-1}}_{\mathbf{h}})_{\lambda}   = \frac{1}{1-t} \sum_{\mu \uparrow \lambda}  \frac{s_{\mu}(1, t, t^{2}, \hdots)}{s_{\lambda}(1, t, t^{2}, \hdots)} = \frac{1}{1-t} \sum_{\mu \uparrow \lambda} \frac{t^{n(\mu)}}{t^{n(\lambda)}} \frac{H_{\lambda}(t)}{H_{\mu}(t)}.$ 
\end{corollary}

\begin{proof}
This follows immediately from Theorems \ref{Schurfake} and \ref{RCA fibre calc}. 
\end{proof}

We can now prove Theorem \ref{cyclotomic q-hook formula}. 

\begin{proof}[Proof of Theorem \ref{cyclotomic q-hook formula}]
Choose any parameter $\mathbf{h} \in \Q^l$ such that $\Spec Z_{\mathbf{h}}$ is smooth. By Theorem \ref{EG iso theoremm} the Etingof-Ginzburg map induces an isomorphism of vector bundles $\mathcal{R}_{\mathbf{h}}^{\Gamma_{n-1}} \cong \mathcal{V}_{\mathbf{h}}$. 
Since the Etingof-Ginzburg map is $\C^*$-equivariant, we have
\begin{equation} \label{charactercomparison}\ch_t (\mathcal{V}_{\mathbf{h}})_{\mu} = \ch_t (\mathcal{R}^{\Gamma_{n-1}}_{\mathbf{h}})_{\quot(\mu)^\flat},\end{equation}
by Theorem \ref{Cor final final}. 
Proposition \ref{residue-character} yields the formula for the LHS of \eqref{charactercomparison} while Theorem \ref{RCA fibre calc} yields the formula for the RHS.
\end{proof}

%\section{Matching the $\C^*$-fixed Points}

\section{$\C^*$-fixed points under reflection functors} \label{chapter Nak Hil}

Assume from now on that $l > 1$. 
In this section we compute the bijections between $\C^*$-fixed points induced by reflection functors (see \S \ref{reflection functors definition section}). 

Fix $\nu \in \heartsuit(l)$ and let $\mathbf{d}_\nu=(d_{0},\hdots,d_{l-1})$ be the corresponding dimension vector (see \S \ref{NorweskiLesny}). Assume that $\theta \in \Q^l$ is chosen so that $\theta_i \neq 0$ and the quiver variety $\mathcal{X}_\theta (n\delta+\mathbf{d}_\nu)$ is smooth.
The reflection functor
\begin{equation} \label{Ri refl functor map} \mathfrak{R}_i : \mathcal{X}_\theta (n\delta+\mathbf{d}_\nu) \to \mathcal{X}_{\sigma_i\cdot\theta} (n\delta+\mathbf{d}_{\sigma_i*\nu}) \end{equation}
is a $\C^*$-equivariant isomorphism and hence induces a bijection $\mathcal{X}_\theta(n\delta + \mathbf{d}_\nu)^{\C^*} \longleftrightarrow \linebreak \mathcal{X}_{\sigma_i\cdot\theta} (n\delta+\mathbf{d}_{\sigma_i*\nu})^{\C^*}$ between the $\C^*$-fixed points. Composing it  with the bijections from Corollary \ref{v-core fp bij}, we obtain a bijection 
\begin{equation} \label{Ri-bijection-blabla} \mathbf{R}_i : \mathcal{P}_{\nu^t}(nl + |\nu^t|) \to \mathcal{P}_{(\sigma_i*\nu)^t}(nl + |(\sigma_i*\nu)^t|).\end{equation}
Fix $\mu \in \mathcal{P}_{\nu^t}(nl + |\nu^t|)$. We are going to show that
\[ \mathbf{R}_i(\mu) = (\mathbf{T}_i(\mu^t))^t,\]
where $\mathbf{T}_i(\mu^t)$ is the partition obtained from $\mu^t$ by adding all $i$-addable and removing all $i$-removable cells relative to $\mu^t$, as in Definition \ref{defi:part-remadd}. 

\subsection{The strategy.}
Our goal is to describe the action of reflection functors on the $\C^*$-fixed points combinatorially. 
In \S \ref{para-mu-grading} we endow the vector space $\widehat{\mathbf{V}}^\nu$ with a $\Z$-grading, which we call the ``$\mu$-grading''. A $\C^*$-fixed point is characterized uniquely by this grading. In \S \ref{para-dim-for} we compute the $\mathbf{R}_i(\mu)$-grading on the vector space $\widehat{\mathbf{V}}^{\sigma_i*\nu}$. In \S \ref{para-remadd} and \S \ref{para-ramadd2} we use this calculation to give a combinatorial description of the partition~$\mathbf{R}_i(\mu)$.

\subsection{The $\mu$-grading.} \label{para-mu-grading} 

Fix a $\Z/l\Z$-graded complex vector space $\widehat{\mathbf{V}}^{\nu} := \bigoplus_{i=0}^{l-1} \mathbf{V}_i^{\nu}$ with $\dim \mathbf{V}_i^{\nu} = n+ d_{i}$. 
Set $\mathbf{V}^\nu = \widehat{\mathbf{V}}^{\nu} \oplus \mathbf{V}_\infty$ with $\dim \mathbf{V}_\infty=1$. We are now going to introduce a $\Z$-grading on $\widehat{\mathbf{V}}^{\nu}$ which ``lifts'' the $\Z/l\Z$-grading. 

\begin{definition} \label{lemma-Z-grading} 

We call a $\Z$-grading $\widehat{\mathbf{V}}^\nu = \bigoplus_{i \in \Z} \mathbf{W}_i$ a $\mu$-\emph{grading} if it satisfies the following condition:
\begin{description}
\item [(C)] for each $i \in \Z$ we have
\[A(\mu)(\mathbf{W}_i) \subseteq \mathbf{W}_{i-1}, \quad \Lambda(\mu)(\mathbf{W}_i) \subseteq \mathbf{W}_{i+1}, \quad J(\mu)(\mathbf{V}_\infty) \subseteq \mathbf{W}_{0}, \quad I(\mu)(\mathbf{W}_0) = \mathbf{V}_{\infty}\]
\end{description}
where $A(\mu), \Lambda(\mu), J(\mu)$ and $I(\mu)$ are as in Definition \ref{A-mu quadruple definition}. 
\end{definition}

\begin{proposition} \label{pro:mugrading1} A $\mu$-grading on $\widehat{\mathbf{V}}^\nu$ exists and is unique.
\end{proposition}

\begin{proof}
We first prove existence. 
Let $\mu = (a_1, \hdots, a_k \mid b_1, \hdots, b_k)$ be the Frobenius form of $\mu$. Set $r_i = b_i +1$, $m_i = a_i+b_i+1$ and $q_i = \sum_{j < i} m_j + r_i$ for $1 \leq i \leq k$. Recall the ordered basis $\{ \Bas(i) \mid 1 \leq i \leq nl + |\nu|\}$ of $\widehat{\mathbf{V}}^{\nu}$ from \S \ref{section-fp-proofs}. We now define $\mathbf{W}_j$ by the rule that for each $1 \leq i \leq k$:
\begin{equation} \label{Kacapuri} \Bas(q_i - j) \in \mathbf{W}_j \quad (0 \leq j \leq b_i), \quad \quad \Bas(q_i + j) \in \mathbf{W}_{-j} \quad (1 \leq j \leq a_i).\end{equation}
It follows directly from the construction of the matrices $A(\mu),\Lambda(\mu),J(\mu)$ and $I(\mu)$ that this grading satisfies condition $\mathbf{(C)}$ in Definition \ref{lemma-Z-grading}. Hence it is a $\mu$-grading. Finally we observe that the definition of $\Lambda(\mu)$ and Lemma \ref{papryczka} imply that
\begin{equation} \label{Syjamskikot} \mathbf{W}_j = (A(\mu))^j(\mathbf{W}_0) \quad (j <0), \quad \quad \mathbf{W}_j = (\Lambda(\mu))^j(\mathbf{W}_0) \quad (j>0).\end{equation}

We now prove uniqueness. Let $\widehat{\mathbf{V}}^\nu = \bigoplus_{i \in \Z} \mathbf{W}'_i$ be a $\mu$-grading. Choose $0\neq v_\infty \in \mathbf{V}_\infty$. Then by the definition of $J(\mu)$ for a generic parameter $\theta$ there exist $c_1,\hdots,c_k \in \C^*$ such that $u_1:=c_1\Bas(q_1) + \hdots + c_k\Bas(q_k) = J(\mu)(v_\infty) \in \mathbf{W}'_0$. 
Lemma \ref{papryczka} together with the fact that the parameter $\theta$ is generic implies that $0 \neq (A(\mu))^{a_1}(\Bas(q_1)) = t (A(\mu))^{a_1}(u_1)$ for some scalar $t \in \C^*$. Since $u_1 \in \mathbf{W}'_0$ and the operator $A(\mu)$ lowers degree by one, we have $0 \neq (A(\mu))^{a_1}(\Bas(q_1)) \in \mathbf{W}'_{-a_1}$. It now follows from the definition of the matrices $A(\mu)$ and $\Lambda(\mu)$ and the genericity of $\theta$ that $(\Lambda(\mu))^{a_1}\circ(A(\mu))^{a_1}(u_1) = t' \Bas(q_1)$ for some scalar $t' \in \C^*$. Since the operator $\Lambda(\mu)$ raises degree by one, we have $\Bas(q_1), c_2\Bas(q_2) + \hdots + c_k\Bas(q_k) \in \mathbf{W}'_0$. 

We can now apply an analogous argument to the vectors $\Bas(q_2)$ and $u_2:=c_2\Bas(q_2) + \hdots + c_k\Bas(q_k)$. By Lemma \ref{papryczka} we have $0 \neq (A(\mu))^{a_2}(\Bas(q_2)) = t (A(\mu))^{a_2}(u_2)+t' (A(\mu))^{a_2}(\Bas(q_1))$ for some scalars $t,t' \in \C^*$. Since $u_2$ and $\Bas(q_1)$ are homogeneous elements of degree zero and $A(\mu)$ lowers degree by one, we get $0 \neq (A(\mu))^{a_2}(\Bas(q_2)) \in \mathbf{W}'_{-a_2}$. Moreover, Lemma \ref{papryczka} implies that $(A(\mu))^{a_2}(\Bas(q_2))$ is a linear combination of $\Bas(q_1 + a_2)$ and $\Bas(q_2 + a_2)$. But $\Bas(q_1 + a_2) = (A(\mu))^{a_2}(\Bas(q_1))$ up to multiplication by a non-zero scalar, so $\Bas(q_1 + a_2) \in \mathbf{W}'_{-a_2}$. Hence $\Bas(q_2 + a_2) \in \mathbf{W}'_{-a_2}$ and so $\Bas(q_2) = (\Lambda(\mu))^{a_2}(\Bas(q_2+a_2)) \in \mathbf{W}'_0$.
We conclude that $\Bas(q_2), u_3:=c_3\Bas(q_3) + \hdots + c_k\Bas(q_k) \in \mathbf{W}'_0$. Repeating this argument sufficiently many times shows that $\Bas(q_1), \hdots, \Bas(q_k) \in \mathbf{W}'_0$. It follows that $\mathbf{W}'_0 = \mathbf{W}_0$. Condition $\mathbf{(C)}$ and \eqref{Syjamskikot} now imply that $\mathbf{W}'_i = \mathbf{W}_i$ for all~$i \in \Z$.
\end{proof}

Thanks to Proposition \ref{pro:mugrading1}, we can talk about \emph{the} $\mu$-grading on $\widehat{\mathbf{V}}^\nu$. Let us denote it by $\widehat{\mathbf{V}}^\nu = \bigoplus_{i \in \Z} \mathbf{W}_i^\mu$. We write $\deg_{\mu} v = i$ if $v \in \mathbf{W}_i^\mu$. Moreover, let $P_\mu := \sum_{i \in \Z} \dim \mathbf{W}_i^\mu t^i$ denote the Poincar\'{e} polynomial of $\mathbf{\mathbf{V}}^\nu$ with respect to the $\mu$-grading. Consider the function
\begin{equation} \label{theXtoPmap} \mathcal{X}_{\theta}(n\delta+\mathbf{d}_\nu)^{\C^*} \to \Z[t,t^{-1}], \quad [\mathbf{A}(\mu)] \mapsto P_\mu.\end{equation}

\begin{proposition} \label{Prop Z-grading}
We have $P_\mu = \Res_{\mu^t}(t)$. Moreover, the function \eqref{theXtoPmap} is injective.
\end{proposition}

\begin{proof}
Fix $j \geq 0$. By \eqref{Kacapuri}, $\Bas(q_i - j) \in \mathbf{W}_j^\mu$ if and only if $j \leq b_i$, for $i = 1, \hdots,k$. Moreover, $\{ \Bas(q_i - j) \mid j \leq b_i \}$ form a basis of $\mathbf{W}_j^\mu$. Hence
$\dim \mathbf{W}_j^\mu = \sum_{i=1}^k \mathbf{1}_{j\leq b_i}$, where $\mathbf{1}_{j\leq b_i}$ is the indicator function taking value $1$ if $j \leq b_i$ and $0$ otherwise. But $\sum_{i=1}^k \mathbf{1}_{j \leq b_i}$ is precisely the number of cells of content $-j$ in $\mu$, which is the same as the number of cells of content $j$ in $\mu^t$. The argument for $j <0$ is analogous. This proves the first claim. The second claim  follows immediately from the fact that partitions are determined uniquely by their residues. 
\end{proof}

Suppose that we are given a $\C^*$-fixed point and want to find out which partition it is labelled by. Proposition \ref{Prop Z-grading} implies that to do so we only need to compute the $\Z$-grading on $\widehat{\mathbf{V}}^\nu$ associated to the fixed point and the corresponding Poincar\'{e} polynomial. 

\begin{lemma} \label{Injectivity remark} 
The following hold: 
\begin{enumerate}[label=\alph*), font=\textnormal,noitemsep,topsep=3pt,leftmargin=0.55cm]
\item The restricted maps \[A(\mu) : \mathbf{W}^\mu_i \to \mathbf{W}^{\mu}_{i-1} \ (i \leq 0), \quad \Lambda(\mu) : \mathbf{W}^\mu_i \to \mathbf{W}^{\mu}_{i+1}, (i \geq 0)\] are surjective. 
\item The restricted maps \[A(\mu) : \mathbf{W}^\mu_i \to \mathbf{W}^{\mu}_{i-1} \ (i > 0), \quad \Lambda(\mu) : \mathbf{W}^\mu_i \to \mathbf{W}^{\mu}_{i+1} \ (i < 0)\] are injective. 
\item We have $\mathbf{V}_j^\nu = \bigoplus_{i \in \Z} \mathbf{W}_{j+il}^\mu$ for each $j=0, \hdots, l-1$.
\end{enumerate} 
\end{lemma}

\begin{proof}
The first claim is just a restatement of \eqref{Syjamskikot}. The second claim follows directly from Lemma \ref{papryczka}. The third claim follows from \eqref{Syjamskikot} and the fact that $A(\mu)$ (resp.\ $\Lambda(\mu)$) is a homogeneous operator of degree $-1$ (resp.\ $1$) with respect to the $\Z/l\Z$-grading on~$\widehat{\mathbf{V}}^\nu$. 
\end{proof}

\begin{remark}
We may interpret Lemma \ref{Injectivity remark}(c) as saying that the $\mu$-grading on $\widehat{\mathbf{V}}^\nu$ is a \emph{lift} of the $\Z/l\Z$-grading. Furthermore, we may think of the $\mu$-grading as the grading on the representation of the $A_\infty$-quiver corresponding to the fixed point $[\mathbf{A}(\mu)]$. More details about the connection between quiver varieties of type $A_\infty$ and the $\C^*$-fixed points in cyclic quiver varieties can be found in \cite[\S 4]{Sav}. 
\end{remark} 

\subsection{Explicit definition of reflection functors.} \label{explicit defi of refl fun}
Let us recall the explicit definition of reflection functors from \cite[Proposition 3.19]{N} and \cite[\S 5]{CBH2}. 
Fix a $\Z/l\Z$-graded complex vector space $\widehat{\mathbf{V}}^{\sigma_i*\nu} := \bigoplus_{j=0}^{l-1} \mathbf{V}_j^{\sigma_i*\nu}$ with $\dim \mathbf{V}_j^{\sigma_i*\nu} = n+ d_{j}'$. Set $\mathbf{V}^{\sigma_i*\nu} := \widehat{\mathbf{V}}^{\sigma_i*\nu} \oplus \mathbf{V}_\infty$. To simplify notation, set $\mathbf{d}:= \mathbf{d}_\nu$ and $\mathbf{d}':= \sigma_i * \mathbf{d}_\nu$.

Let us first define a map $\mu_{n\delta + \mathbf{d}} ^{-1}(\theta) \to \mu_{n\delta + \mathbf{d}'} ^{-1}(\sigma\cdot\theta)$ lifting \eqref{Ri refl functor map}, which we also denote by $\mathfrak{R}_i$. 
Let $\rho = (\mathbf{X}, \mathbf{Y}, I, J) \in \mu_{n\delta + \mathbf{d}} ^{-1}(\theta)$.  
The reflected quiver representation 
\[\mathfrak{R}_i(\rho) =: (\mathbf{X}', \mathbf{Y}', I', J')\] 
is defined as follows. 
Suppose that $i \neq 0$. We have maps
\begin{equation} \label{reflection definition functors} \mathbf{V}_i^\nu \xrightarrow{Y_i - X_i} \mathbf{V}_{i-1}^\nu \oplus \mathbf{V}_{i+1}^\nu \xrightarrow{X_{i-1} + Y_{i+1}} \mathbf{V}_i^\nu.\end{equation}
Set $\psi_i := Y_i - X_i$ and $\phi_i := X_{i-1} + Y_{i+1}$ (the indices should be considered $\mmod$ $l$). The preprojective relations (i.e.\ the relations defining the fibre $\mu_{n\delta + \mathbf{d}} ^{-1}(\theta)$) ensure that we have a splitting $ \mathbf{V}_{i-1}^\nu \oplus \mathbf{V}_{i+1}^\nu = \Ima \psi_i \oplus \ker \phi_i$. The underlying vector space of the quiver representation $\mathfrak{R}_i(\rho)$ is obtained from $\mathbf{V}^\nu$ by replacing $\mathbf{V}_i^\nu$ with $\ker \phi_i$.
We have an isomorphism of vector spaces $\mathbf{V}^{\sigma_i*\nu} \cong \ker \phi_i \oplus \bigoplus_{j \neq i} \mathbf{V}_j^\nu \oplus \mathbf{V}_\infty$ preserving the quiver grading. Let us define  $\mathfrak{R}_i(\rho)$. We have $X_j' = X_j$ unless $j \in \{i-1,i\}$. We also have $Y_j' = Y_j$ unless $j \in \{i, i+1\}$. Set $I' = I$ and $J'=J$. The maps $X_i'$ and $Y_i'$ are defined as the composite maps
\begin{align*} X'_i:& \ \ker \phi_i \hookrightarrow \ker \phi_i \oplus \Ima \psi_i = \mathbf{V}_{i-1}^\nu \oplus \mathbf{V}_{i+1}^\nu \twoheadrightarrow \mathbf{V}_{i+1}^\nu,\\
Y'_i:& \ \ker \phi_i \hookrightarrow \ker \phi_i \oplus \Ima \psi_i = \mathbf{V}_{i-1}^\nu \oplus \mathbf{V}_{i+1}^\nu \twoheadrightarrow \mathbf{V}_{i-1}^\nu.\end{align*}
The maps $X_{i-1}'$ and $Y_{i+1}'$ are defined as the composite maps
\begin{align} X'_{i-1}:& \ \mathbf{V}_{i-1}^\nu \xrightarrow{\cdot (- \theta_i)} \mathbf{V}_{i-1}^\nu \hookrightarrow  \mathbf{V}_{i-1}^\nu \oplus \mathbf{V}_{i+1}^\nu = \ker \phi_i \oplus \Ima \psi_i \twoheadrightarrow \ker \phi_i,\\ \label{refl functor expl Y}
Y'_{i+1}:& \ \mathbf{V}_{i+1}^\nu \xrightarrow{\cdot (- \theta_i)} \mathbf{V}_{i+1}^\nu \hookrightarrow  \mathbf{V}_{i-1}^\nu \oplus \mathbf{V}_{i+1}^\nu = \ker \phi_i \oplus \Ima \psi_i \twoheadrightarrow \ker \phi_i \xrightarrow{\cdot (-1)} \ker \phi_i.\end{align}
The minus signs before $X_i$ in \eqref{reflection definition functors} as well as in the last arrow above come from the fact that our quiver does not have a sink at the vertex $i$ as in \cite[\S 5]{CBH2} - hence the need for these adjustments. 
If $i=0$ one defines $\mathfrak{R}_i(\rho)$ analogously using maps
\begin{equation} \label{NRF-ker2} \mathbf{V}_0^\nu \xrightarrow{\psi_0} \mathbf{V}_{l-1}^\nu \oplus \mathbf{V}_{1}^\nu \oplus \mathbf{V}_\infty \xrightarrow{\phi_0} \mathbf{V}_0^\nu \end{equation}
with $\psi_0 = Y_0 - X_0 + I$ and $\phi_0 = X_{l-1} + Y_1 + J$. 

\subsection{The reflected grading.} \label{para-dim-for}
Let us apply the definitions from \S \ref{explicit defi of refl fun} to $\rho = \mathbf{A}(\mu)$. More specifically, for $i = 0,\hdots,l-1$, we set $X_i = \Lambda(\mu)|_{\mathbf{V}_i^\nu}$, $Y_i = A(\mu)|_{\mathbf{V}_i^\nu}$, $I=I(\mu)$ and $J=J(\mu)$. 
The reflected quiver representation $\mathfrak{R}_i(\mathbf{A}(\mu)) \in \mu_{n\delta+\mathbf{d}'} ^{-1}(\sigma\cdot\theta)$ is conjugate under the $G(n\delta+\mathbf{d}')$-action to $\mathbf{A}(\mathbf{R}_i(\mu))$. 
We now compute the $\mathbf{R}_i(\mu)$-grading on $\widehat{\mathbf{V}}^{\sigma_i*\nu}$. 

We have direct sum decompositions \begin{equation} \label{Venice} \psi_i := \bigoplus_{j \in \Z} \psi_i^j,\quad \phi_i := \bigoplus_{j \in \Z} \phi_i^j,\end{equation} with $\psi_i^j = \psi_i|_{\mathbf{W}^\mu_{jl+i}}$ and $\phi_i^j = \phi_i|_{\mathbf{W}^\mu_{jl+i-1} \oplus \mathbf{W}^\mu_{jl+i+1}}$ $(i \neq 0$ or $j \neq 0)$ and $\psi_0^0 := \psi_0|_{\mathbf{W}^\mu_{0}}$. Hence $\mathbf{V}_i^{\sigma_i*\nu} = \ker \phi_i = \bigoplus_{j \in \Z} \ker \phi_i^j$. Moreover, \eqref{reflection definition functors} and \eqref{NRF-ker2} decompose as direct sums~of~maps
\begin{align*} &\mathbf{W}_{jl+i}^\mu \xrightarrow{\psi^j_i} \mathbf{W}_{jl+i-1}^\mu \oplus \mathbf{W}_{jl+i+1}^\mu \xrightarrow{\phi^j_i} \mathbf{W}_{jl+i}^\mu \quad (j \neq 0 \mbox{ or } i \neq 0).\\
&\mathbf{W}_0^\mu \xrightarrow{\psi_0^0} \mathbf{W}_{-1}^\mu \oplus \mathbf{W}_{1}^\mu \oplus \mathbf{V}_\infty \xrightarrow{\phi_0^0} \mathbf{W}_0^\mu. \end{align*} 
These direct sum decompositions together with the preprojective relations imply the following lemma. 
\begin{lemma} \label{Wydmuszka}
Let $i \in \{0, \hdots, l-1\}$ and $j \in \Z$. If $i =0$ and $j=0$ then $\Ima \psi_0^0 \oplus \ker \phi_0^0 = \mathbf{W}_{-1}^\mu \oplus \mathbf{W}_{1}^\mu \oplus \mathbf{V}_\infty$. Otherwise $\Ima \psi_i^j \oplus \ker \phi_i^j = \mathbf{W}_{jl+i-1}^\mu \oplus \mathbf{W}_{jl+i+1}^\mu$. 
\end{lemma}
\noindent
Define 
\[ \mathbf{U}_{j} := \mathbf{W}_{j}^\mu \quad (j \neq i \mmod l), \quad \quad \mathbf{U}_{lj+i} := \ker \phi_i^j \quad (j \in \Z).\]
\begin{proposition}\label{L-ker-W}
The $\Z$-grading $\widehat{\mathbf{V}}^{\sigma_i*\nu} = \bigoplus_{j \in \Z} \mathbf{U}_j$ is the $\mathbf{R}_i(\mu)$-grading on $\widehat{\mathbf{V}}^{\sigma_i*\nu}$.
\end{proposition}

\begin{proof}
It suffices to show that  $\widehat{\mathbf{V}}^{\sigma_i*\nu} = \bigoplus_{j \in \Z} \mathbf{U}_j$ satisfies condition $\mathbf{(C)}$ in Definition \ref{lemma-Z-grading}. Suppose $i \neq 0$. Then for each $j \in \Z$ we need to check that \[A(\mathbf{R}_i(\mu))(\mathbf{U}_{jl+i+1}) \subseteq \mathbf{U}_{jl+i}, \quad \Lambda(\mathbf{R}_i(\mu))(\mathbf{U}_{jl+i-1}) \subseteq \mathbf{U}_{jl+i},\] \[ A(\mathbf{R}_i(\mu))(\mathbf{U}_{jl+i}) \subseteq \mathbf{U}_{jl+i-1}, \quad \Lambda(\mathbf{R}_i(\mu))(\mathbf{U}_{jl+i}) \subseteq \mathbf{U}_{jl+i+1}.\]
If $i=0$ we additionally need to check that $I(\mathbf{R}_i(\mu))(\mathbf{U}_0) = \mathbf{V}_\infty$ and $J(\mathbf{R}_i(\mu))(\mathbf{V}_\infty) \subseteq \mathbf{U}_0$. 

All of the inclusions above follow directly from Lemma \ref{Wydmuszka} and the definition of reflection functors. For example, let us assume that $i \neq 0$ or $j \neq 0$ and consider the inclusion 
$A(\mathbf{R}_i(\mu))(\mathbf{U}_{jl+i+1}) \subseteq \mathbf{U}_{jl+i} := \ker \phi_i^j$. The map $A(\mathbf{R}_i(\mu)):\mathbf{V}_{i+1}^{\sigma_i*\nu} = \mathbf{V}_{i+1}^{\nu} \to \mathbf{V}_i^{\sigma_i*\nu} = \ker \phi_i$ is given by \eqref{refl functor expl Y}. 
Consider its restriction to the subspace $\mathbf{U}_{jl+i+1} = \mathbf{W}_{jl+i+1}^\mu \subseteq \mathbf{V}_{i+1}^\nu$. 
By Lemma \ref{Wydmuszka}, we have $\mathbf{W}_{jl+i+1}^\mu \subseteq  \mathbf{W}_{jl+i-1}^\mu \oplus \mathbf{W}_{jl+i+1}^\mu = \ker \phi_i^j \oplus \Ima \psi_i^j$. Hence $A(\mathbf{R}_i(\mu))$, restricted to $\mathbf{W}_{jl+i+1}^\mu$, is the composition
\[ \mathbf{W}_{jl+i+1}^\mu \xrightarrow{\cdot (- \theta_i)} \mathbf{W}_{jl+i+1}^\mu \hookrightarrow  \mathbf{W}_{jl+i-1}^\mu \oplus \mathbf{W}_{jl+i+1}^\mu = \ker \phi_i^j \oplus \Ima \psi_i^j \twoheadrightarrow \ker \phi_i^j \xrightarrow{\cdot (-1)} \ker \phi_i^j.\]
In particular, $A(\mathbf{R}_i(\mu))(\mathbf{W}_{jl+i+1}^\mu) \subseteq \ker \phi_i^j$, as desired. The other inclusions are proven analogously.
\end{proof}

We have described the $\mathbf{R}_i(\mu)$-grading on $\widehat{\mathbf{V}}^{\sigma_i*\nu}$. We are now going to compute the corresponding Poincar\'{e} polynomial $P_{\mathbf{R}_i(\mu)}$. 

\begin{lemma} \label{L-W-V-dim}
Let $\widehat{\mathbf{V}}^\nu = \bigoplus_{i \in \Z} \mathbf{W}_i^\mu$ be the $\mu$-grading on $\widehat{\mathbf{V}}^\nu$. Then
\[\dim \ker \phi_i^j = \dim \mathbf{W}_{lj+i+1}^\mu + \dim \mathbf{W}_{lj+i-1}^\mu - \dim \mathbf{W}_{lj+i}^\mu\] if $j \neq 0$ or $i \neq 0$. Otherwise \[\dim \ker \phi_0^0 = \dim \mathbf{W}_{1}^\mu + \dim \mathbf{W}_{-1}^\mu - \dim \mathbf{W}_{0}^\mu + 1.\] 
\end{lemma}

\begin{proof} 
Assume that $j \neq 0$ or $i \neq 0$.
Recall that $\psi_i^j = A(\mu)|_{\mathbf{W}_{jl+i}^\mu} - \Lambda(\mu)|_{\mathbf{W}_{jl+i}^\mu}$. 
Lemma \ref{Injectivity remark} implies that either $A(\mu)|_{\mathbf{W}_{jl+i}^\mu}$ or $\Lambda(\mu)|_{\mathbf{W}_{jl+i}^\mu}$ is injective. Hence $\psi_i^j$ is injective. 
Therefore we have $\dim \Ima \psi_i^j = \dim \mathbf{W}_{lj+i}^\mu$. The equality $\dim \ker \phi_i^j =  \dim \mathbf{W}_{lj+i+1}^\mu + \dim \mathbf{W}_{lj+i-1}^\mu - \dim \Ima \psi_i^j$ now implies the lemma. The case $i=j=0$ is similar. 
\end{proof}

Write $P_\mu = \sum_{j \in \Z} a^\mu_j t^j$ with $a_j^\mu = \dim \mathbf{W}_{j}^\mu$ and $P_{\mathbf{R}_i(\mu)} = \sum_{j \in \Z} a_j^{\mathbf{R}_i(\mu)} t^j$ with $a_j^{\mathbf{R}_i(\mu)} = \dim \mathbf{W}_{j}^{\mathbf{R}_i(\mu)}$. Proposition \ref{L-ker-W} and Lemma \ref{L-W-V-dim} directly imply the following.

\begin{corollary} \label{Cor-Poincare polynomials}
We have $a^\mu_j = a^{\mathbf{R}_i(\mu)}_j$ for $j \neq i \mmod l$. Moreover $a^{\mathbf{R}_i(\mu)}_{lj+i} = a^{\mu}_{lj+i+1} + a^{\mu}_{lj+i-1} - a^\mu_{lj+i}$ if $j \neq 0$ or $i \neq 0$ and  $a^{\mathbf{R}_0(\mu)}_0 = a^\mu_1 + a^\mu_{-1} - a^\mu_0 +1$ otherwise. 
\end{corollary}

\subsection{Removable and addable cells.} \label{para-remadd}

Throughout this subsection let $\mu$ be an arbitrary partition.
Recall the partition $\mathbf{T}_k(\mu)$ from Definition \ref{defi:part-remadd} obtained from $\mu$ by  adding all $k$-addable cells and removing all $k$-removable cells. We will now interpret addability and removability in terms of the residue of $\mu$. 

\begin{lemma} \label{Lemma-corners}
Let us write $\Res_{\mu}(t) = \sum_{j \in \Z} b_j t^j$. Let $k \in \Z$. 
\begin{enumerate}[label=\alph*), font=\textnormal,noitemsep,topsep=3pt,leftmargin=0.55cm]
\item The following are equivalent: (1) a cell of content $k$ is removable; (2) $b_{k-1} = b_k, b_{k+1} = b_k - 1$ $(k>0)$ or $b_{k+1} = b_k, b_{k-1} = b_k - 1$ $(k<0)$ or $b_{-1} = b_{1}, b_{0} = b_1 + 1$ $(k=0)$. 
\item The following are equivalent: (1) a cell of content $k$ is addable; (2) $b_{k+1} = b_k, b_{k-1} = b_k + 1$ $(k>0)$ or $b_{k-1} = b_k, b_{k+1} = b_k + 1$ $(k<0)$ or $b_{-1} = b_0 = b_{1}$ $(k=0)$. 
\item The following are equivalent: (1) no cell of content $k$ is addable or removable; (2) $[b_{k+1} = b_k = b_{k-1}$ or $b_{k+1} = b_k -1 = b_{k-1}-2$ $(k>0)]$, or $[b_{k+1} = b_k = b_{k-1}$ or $b_{k+1} = b_k +1 = b_{k-1}+2$ $(k<0)]$, or $[b_{-1} = b_0 = b_{1}+1$ or $b_{1} = b_0 = b_{-1}+1$ $(k=0)]$. 
\end{enumerate} 
\end{lemma}

\begin{proof}

Let $k>0$. Suppose that the cell $(i,j)$ is removable from $\mathbb{Y}(\mu)$ and has content $k$. Then $j-i=k$ and $(1,k+1), (2,k+2), \hdots, (i,j)$ are precisely the cells of content $k$ in $\mathbb{Y}(\mu)$. Since $\mathbb{Y}(\mu)$ has a corner at $k$, the cells of content $k-1$ in $\mathbb{Y}(\mu)$ are precisely $(1,k), (2,k+1), \hdots, (i,j-1)$ and the cells of content $k+1$ in $\mathbb{Y}(\mu)$ are precisely $(1,k+2), (2,k+3), \hdots, (i-1,j)$. Hence $b_k = i, b_{k-1} = i, b_{k+1} = i-1$, which yields the desired equalities. Conversely, suppose that  $b_{k-1} = b_k, b_{k+1} = b_k - 1$. Then the cell $(b_k,b_k+k)$ is removable. Indeed, $b_{k-1} = b_k$ implies that $(b_k+1,b_k+k) \notin \mathbb{Y}(\mu)$ and $b_{k+1} = b_k - 1$ implies that $(b_k,b_k+k+1) \notin \mathbb{Y}(\mu)$. 
The proofs of the remaining cases are analogous. 
\end{proof}

\subsection{Combinatorial interpretation of reflection functors.} \label{para-ramadd2}
We can now interpret the effect of applying reflection functors to the fixed points combinatorially. 
\begin{theorem} \label{RkTk}
Let $\mu \in \mathcal{P}_{\nu^t}(nl+|\nu^t|)$ and $0 \leq k \leq l-1$.
We have $\mathbf{R}_k(\mu) = (\mathbf{T}_k(\mu^t))^t$. 
\end{theorem}

\begin{proof}
It suffices to show that the residue of $\mathbf{T}_k(\mu^t)$ equals $P_{\mathbf{R}_k(\mu)}$. Let us write 
\[ \Res_{\mathbf{T}_k(\mu^t)}(t) =: \sum_{i \in \Z} a_i t^i, \quad \Res_{\mu^t}(t) =: \sum_{i \in \Z} b_i t^i, \quad P_{\mathbf{R}_k(\mu)} =: \sum_{i \in \Z} c_i t^i, \quad P_{\mu} =: \sum_{i \in \Z} d_i t^i.\] By Lemma \ref{Prop Z-grading}, we have $b_i = d_i$. 
Suppose that $i \neq k \mmod l$. Then $a_i = b_i = d_i = c_i$. The first equality follows from the fact that no cells of content $i$ are added to or removed from $\mu^t$ when we transform $\mu^t$ into $\mathbf{T}_k(\mu^t)$. The third equality follows from Corollary \ref{Cor-Poincare polynomials}. 

Now suppose that $i = k \mmod l$ and $i > 0$. Then $c_i = d_{i+1} + d_{i-1} - d_i$ by Corollary \ref{Cor-Poincare polynomials}. Hence $c_i = b_{i+1} + b_{i-1} - b_i$. We now argue that $b_{i+1} + b_{i-1} - b_i = a_i$. There are three possibilities: $a_i = b_i +1$ and one cell of content $i$ is addable to $\mu^t$, or $a_i = b_i - 1$ and one cell of content $i$ is removable from $\mu^t$, or $a_i = b_i$ and no cell of content $i$ is addable to or removable from $\mu^t$. In the first case we have $b_i = b_{i+1}$ and $b_{i-1} = b_i +1$. In the second case we have $b_i = b_{i-1}$ and $b_{i+1} = b_i - 1$. In the third case we have $b_{i-1} = b_i = b_{i+1}$ or $b_{i} = b_{i+1} + 1 = b_{i-1} - 1$. These equalities follow immediately from Lemma \ref{Lemma-corners}. In each of the three cases we see that the equality $b_{i+1} + b_{i-1} - b_i = a_i$ holds. Hence $a_i = c_i$. 
The proof for $i \leq 0$ is analogous.
\end{proof}
\noindent
Recall that if $\lambda \in \mathcal{P}$ then
\begin{equation} \label{Bridgetjones} \quot(\lambda^t) = (\quot(\lambda)^t)^{\flat}.\end{equation}

\begin{corollary} \label{bigcorAmbassadors}
Let $\mu \in \mathcal{P}_{\nu^t}(nl+|\nu^t|)$ and $i \in \{0,\hdots,l-1\}$. Then $\mathbf{R}_i(\mu) = (\sigma_i * \mu^t)^t$. Moreover,
\[ \core(\mathbf{R}_i(\mu)) = (\sigma_i*\nu)^t = (\mathbf{T}_i(\nu))^t, \quad \quot(\mathbf{R}_i(\mu)) = s_{l-i}\cdot\quot(\mu).\]
\end{corollary}

\begin{proof}
The first claim follows directly from Theorem \ref{RkTk} and the definition of the $\tilde{S}_l$-action on partitions in \S \ref{NorweskiLesny}. The formula for $\core(\mathbf{R}_i(\mu))$ follows directly from Proposition \ref{Affineactioncorequotient}. The formula for $\quot(\mathbf{R}_i(\mu))$ follows from Proposition \ref{Affineactioncorequotient} and \eqref{Bridgetjones}. Indeed,
\begin{align*}
 \quot(\mathbf{R}_i(\mu)) = \quot( (\sigma_i * \mu^t)^t) =& \ ((\quot(\sigma_i*\mu^t))^t)^\flat \\
=& \  ((s_{i}\cdot \quot(\mu^t))^t)^\flat \\
=& \ (s_{i}\cdot(\quot(\mu))^\flat)^\flat =  s_{l-i}\cdot\quot(\mu).\qedhere
\end{align*} 
\end{proof}

Note that, by Proposition \ref{Affineactioncorequotient}, we also have: 
\begin{equation} \label{Catcreature} \core((\mathbf{R}_i(\mu))^t) = \sigma_i*\nu = \mathbf{T}_i(\nu), \quad \quot((\mathbf{R}_i(\mu))^t) =  \pr(\sigma_{i})\cdot\quot(\mu) = s_{i}\cdot\quot(\mu).\end{equation}

\section{Connection to the Hilbert scheme} 

\subsection{The Hilbert scheme.} \label{Hilbertschemesection}
Let $K$ be a positive integer. We let $\Hilb_K$ denote the \emph{Hilbert scheme} of $K$ points in $\C^2$. The underlying set of the scheme $\Hilb_K$ consists of ideals of $\C[z_1,z_2]$ of colength $K$, i.e., ideals $I \subset \C[z_1,z_2]$ such that $\dim \C[z_1,z_2]/I = K$.

We let $\C^*$ act on $\C[z_1,z_2]$ by the rule $t.z_1 = tz_1, t.z_2 = t^{-1}z_2$. This action induces an action on $\Hilb_K$. 
The $\C^*$-fixed points in $\Hilb_K$ are precisely the monomial ideals in $\C[z_1,z_2]$. Let $\lambda \in \mathcal{P}(K)$. Let $I_\lambda$ be the $\C$-span of the monomials $\{z_1^iz_2^j \mid (i+1,j+1) \notin \mathbb{Y}(\lambda)\}$. We have a bijection
\[ \mathcal{P}(K) \longleftrightarrow \Hilb_K^{\C^*}, \quad \lambda \mapsto I_\lambda.\] 
Let $\mathcal{T}_K$ denote the tautological bundle on $\Hilb_K$. Its fibre $(\mathcal{T}_{K})_I$ at $I$ is isomorphic to $\C[z_1,z_2]/I$. 
The following lemma follows immediately from the definitions. 
\begin{lemma} \label{KocinkaFionka}
Let $\lambda \in \mathcal{P}(K)$. We have $\ch_t (\mathcal{T}_{K})_{I_\lambda} = \Res_{\lambda^t}(t)$. 
\end{lemma}

There is also a $\Z/l\Z$-action on $\Hilb_K$ induced by the $\Z/l\Z$-action on $\C[z_1,z_2]$ given by $\epsilon.z_1 = \eta^{-1}z_1, \epsilon.z_2 = \eta z_2$. 

\subsection{Connection to rational Cherednik algebras.} \label{Glaurung}

Set $\mathbf{-1} := -\frac{1}{l}(1, \hdots,1) \in \Q^l$ and $\mathbf{-\frac{1}{2}}:= -\frac{1}{2l} (1, \hdots, 1)\in\Q^l$. Let $w \in \tilde{S}_l$ and set $\theta := w^{-1} \cdot (\mathbf{-\frac{1}{2}}) \in \Q^l$ as well as $\gamma := w * n\delta \in \Z^l$. We have $\gamma = n\delta + \gamma_0$, where $\gamma_0 = w * \varnothing$. Let $\nu:=\mathfrak{d}^{-1}(\gamma_0)$ be the $l$-core corresponding to $\gamma_0$. 
By \cite[Lemmas 4.3, 7.2]{gor-qv}, the quiver variety $\mathcal{X}_{\theta} (n\delta)$ is smooth. 
Set $\mathbf{h} := (h, H_1, \hdots, H_{l-1})$ with $H_j = \theta_j$ ($1 \leq j \leq l-1$) and $h = -\theta_0 - \sum_{j=1}^{l-1} H_j$. 
Let us fix a reduced expression $w = \sigma_{i_1} \cdots \sigma_{i_m}$ for $w$ in $\tilde{S}_l$.  
Composing reflection functors yields a $U(1)$-equivariant hyper-K\"{a}hler isometry 
\begin{equation} \label{RiRi} \mathfrak{R}_{i_1} \circ \cdots \circ \mathfrak{R}_{i_m} : \mathcal{X}_\theta (n\delta) \xrightarrow{\sim} \mathcal{X}_{-\mathbf{\frac{1}{2}}} (\gamma).
\end{equation}
By \cite[\S 3.7]{gor-qv} there exists a $U(1)$-equivariant diffeomorphism 
\begin{equation} \label{XtoM} \mathcal{X}_\mathbf{-\frac{1}{2}}(\gamma) \xrightarrow{\sim} \mathcal{M}_{-\mathbf{1}}(\gamma).\end{equation} Set $K = nl + |\nu|.$
By \cite[Lemma 7.8]{gor-qv}, there is a $\C^*$-equivariant embedding 
\begin{equation} \label{embedding-M-Hilb} \mathcal{M}_{-\mathbf{1}}(\gamma) \hookrightarrow  \Hilb_K.\end{equation}
Its image is the component $\Hilb_K^\nu$ of $\Hilb_K^{\Z/l\Z}$ whose generic points have the form $V(I_\nu) \cup O$, where $O$ is a union of $n$ distinct free $\Z/l\Z$-orbits in $\C^2$. Moreover, $(\Hilb_K^\nu)^{\C^*} = \{ I_\lambda \mid \lambda \in \mathcal{P}_{\nu}(nl+|\nu|)\}$. 
Let
\begin{equation} \label{Phi-XtoHilb2}\Phi : \mathcal{X}_\mathbf{-\frac{1}{2}}(\gamma) \to \mathcal{M}_{-\mathbf{1}}(\gamma) \to \Hilb_K^\nu \end{equation}
be the composition of \eqref{XtoM} and \eqref{embedding-M-Hilb}.  It induces a bijection between the $\C^*$-fixed points and hence also a bijection between their labelling sets
\[ \Psi : \mathcal{P}_{\nu^t}(nl+|\nu^t|) \to \mathcal{P}_{\nu}(nl+|\nu|), \quad \mu \mapsto \lambda,\]
where the partition $\lambda$ is defined by the equation $I_{\lambda} = \Phi([\mathbf{A}(\mu)])$. 
\begin{lemma} \label{Pro:Hilb bijection}
Let $\mu \in \mathcal{P}_{\nu^t}(nl+|\nu^t|)$. We have $\Psi(\mu) = \mu^t$. 
\end{lemma}

\begin{proof}
Let $\mathcal{V}_\mathbf{-\frac{1}{2}}(\gamma) := \mu^{-1}_{\gamma}(\mathbf{-\frac{1}{2}} )\times^{G(\gamma)} \widehat{\mathbf{V}}^\nu$ denote the tautological bundle on $\mathcal{X}_\mathbf{-\frac{1}{2}}(\gamma)$. The diffeomorphism \eqref{Phi-XtoHilb2} lifts to a $U(1)$-equivariant isomorphism of tautological vector bundles 
\begin{equation} \label{VtoTiso} \mathcal{V}_\mathbf{-\frac{1}{2}}(\gamma) \xrightarrow{\cong} \mathcal{T}_K^\nu,
\end{equation}
where $\mathcal{T}_K^\nu$ denotes the restriction of $\mathcal{T}_K$ to the subscheme $\Hilb_K^\nu$.  Proposition \ref{residue-character} implies that $\ch_t \mathcal{V}_{\mathbf{-\frac{1}{2}}}(\gamma)_{[\mathbf{A}(\mu)]} = \Res_{\mu}(t)$.
Hence the $\C^*$-characters of the fibres of $\mathcal{V}_{\mathbf{-\frac{1}{2}}}(\gamma)$ at any two distinct $\C^*$-fixed points are distinct. By Lemma \ref{KocinkaFionka}, we have $\ch_t (\mathcal{T}_K^\nu)_{I_{\mu}} = \Res_{\mu^t}(t)$. The $U(1)$-equivariance of \eqref{VtoTiso} implies that $\Phi([\mathbf{A}(\mu)]) = I_{\mu^t}$ and so $\Psi(\mu) =~\mu^t$. 
\end{proof}

We are now ready to collect all our results about the $\C^*$-fixed points. 
We have a sequence of equivariant isomorphisms (of varieties or manifolds): 
\begin{equation} \label{EGRiRiPhi2} \mathcal{Y}_{\mathbf{h}} \xrightarrow{\EG} \mathcal{X}_\theta(n\delta) \xrightarrow{\mathfrak{R}_{i_1} \circ \cdots \circ \mathfrak{R}_{i_m}} \mathcal{X}_{-\mathbf{\frac{1}{2}}} (\gamma) \xrightarrow{\Phi} \Hilb_K^\nu.\end{equation}
They induce bijections between the labelling sets of $\C^*$-fixed points. 
\begin{theorem} \label{caracal1}
The map \eqref{EGRiRiPhi2} induces the following bijections
\[
\begin{array}{c c c c c c c}
\mathcal{P}(l,n) &\longrightarrow& \mathcal{P}_{\varnothing}(nl) &\longrightarrow& \mathcal{P}_{\nu^t}(nl+|\nu^t|) &\longrightarrow& \mathcal{P}_{\nu}(nl+|\nu|) \\
\quot(\mu)^\flat &\longmapsto& \mu &\longmapsto& (w*\mu^t)^t &\longmapsto& w*\mu^t.
\end{array}
\]
Moreover,
\[ \nu = w*\varnothing = \mathbf{T}_{i_1} \circ \hdots \circ \mathbf{T}_{i_m}(\varnothing), \quad \quot(w*\mu^t) = \mathsf{pr}(w)\cdot \quot(\mu^t).\]
\end{theorem}

\begin{proof}
The theorem just collects the results of Theorem \ref{Cor final final}, Corollary \ref{bigcorAmbassadors}, Lemma \ref{Pro:Hilb bijection} and \eqref{Catcreature}. 
\end{proof}

Let us rephrase our result slightly. Given $w \in \tilde{S}_l$, we define the $w$\emph{-twisted} $l$\emph{-quotient bijection} to be the map
\[ \tau_{w} \colon \mathcal{P}(l,n) \to \mathcal{P}_{\nu}(nl+|\nu|), \quad \quot(\mu) \mapsto w*\mu.\]

\begin{corollary} \label{caracal2} 
The bijection $\mathcal{P}(l,n) \to \mathcal{P}_{\nu}(nl+|\nu|)$ induced by \eqref{EGRiRiPhi2} is given by 
\begin{equation} \label{cor: twisted bijection} \underline{\lambda} \mapsto \tau_{w}(\underline{\lambda}^t). \end{equation} 
\end{corollary}
\begin{proof}
Suppose that $\underline{\lambda} = \quot(\mu)^\flat$. Then Theorem \ref{caracal1} implies that $\underline{\lambda}$ is sent to $w * \mu^t$. On the other hand, $\underline{\lambda}^t = (\quot(\mu)^\flat)^t = \quot(\mu^t)$ by \eqref{Bridgetjones}. Hence $\tau_{w}( \underline{\lambda}^t ) =  w * \mu^t$.
\end{proof}

\begin{remark} \label{remark counterexample}
The statement of Corollary \ref{caracal2} appears in the proof of \cite[Proposition 7.10]{gor-qv}. However, the proof of this statement in \cite{gor-qv} is incorrect. The problem lies in an incorrect assumption about the function $c_{\mathbf{h}}: \mathcal{P}(l,n) \times \Q^l \to \Q$,  defined by: 
\[ c_{\mathbf{h}}(\underline{\lambda}) = l \sum_{i=1}^{l-1} |\lambda^i|(H_1 + \hdots + H_i) - l\left(\frac{n(n-1)}{2} + \sum_{i=0}^{l-1}  n(\lambda^i) - n((\lambda^i)^t)\right)h.\]
Given $\mathbf{h} \in \Q^l$, the function $c_{\mathbf{h}}$ induces an ordering on $\mathcal{P}(l,n)$, called the $c$-order, given by the rule
\[ \underline{\mu} <_{\mathbf{h}} \underline{\lambda} \iff c_{\mathbf{h}}(\underline{\mu}) < c_{\mathbf{h}}(\underline{\lambda}).\]
Dependence of this order on $\mathbf{h}$ decomposes the parameter space $\Q^l$ into a finite number of so-called $c$-chambers. 
It is stated in \cite[\S 2.5]{gor-qv} that the $c$-order is a total order inside $c$-chambers. This, however, is not true.  
Let us consider counterexamples in which the $c$-order is not total for all values of $\mathbf{h}$. For example, take $l=1$. 
Then 
\begin{equation} \label{c-func} c_{\mathbf{h}}(\lambda) = -\left(\frac{n(n-1)}{2} + n(\lambda) - n(\lambda^t)\right)h.\end{equation}
It follows immediately from \eqref{c-func} that $c_{\mathbf{h}}(\lambda) = c_{\mathbf{h}}(\mu)$ for all values of $\mathbf{h}$ if $\lambda$ and $\mu$ are two symmetric partitions in $\mathcal{P}(n)$. There are other examples. Take $\mu = (6,3,2,2,2)$. Then $n(\mu) = 3 + 4 + 6 + 8 = 21$. Since $\mu^t = (5,5,2,1,1,1)$ we have $n(\mu^t) = 5 + 4 + 3 + 4 + 5 = 21$. It follows that $\mu$ and $\mu^t$ are incomparable in the $c$-order for all values of $\mathbf{h}$. 

Gordon's proof of \eqref{cor: twisted bijection} relies on comparing the values of some Morse functions on the quiver variety $\mathcal{M}_{2\theta}(n\delta)$ and the Hilbert scheme at the $\C^*$-fixed points. This approach would work if the Morse functions assigned distinct values to each fixed point. However, this isn't the case because the Morse function on $\mathcal{M}_{2\theta}(n\delta)$, evaluated at the fixed points, is given by $c_{\mathbf{h}}$. 
\end{remark} 

\subsection{The combinatorial and geometric orderings.} 

In \cite[\S 5.4]{gor-qv} Gordon defines a \emph{geometric ordering} $\preceq_{w}^{\mathsf{geo}}$ on $\mathcal{P}(l,n)$ using the closure relations between the attracting sets of $\C^*$-fixed points in $\mathcal{M}_{2\theta}(n\delta)$ (Gordon uses the notation $\preceq_{\mathbf{h}}$). We also have a \emph{combinatorial ordering} on $\mathcal{P}(l,n)$ given~by: 
\[ \underline{\mu} \preceq_{w}^{\mathsf{com}} \underline{\lambda} \iff \tau_{w}(\underline{\lambda}^t)  \trianglelefteq \tau_{w}(\underline{\mu}^t),\]
where $\trianglelefteq$ denotes the dominance ordering on partitions. 

\begin{corollary} \label{cor macdonald} 
Let $w \in \tilde{S}_l$ and $\underline{\mu}, \underline{\lambda} \in \mathcal{P}(l,n)$. Then
$ \underline{\mu} \preceq_{w}^{\mathsf{geo}} \underline{\lambda} \ \Rightarrow \ \underline{\mu} \preceq_{w}^{\mathsf{com}} \underline{\lambda}.$ 
\end{corollary} 

\begin{proof}
Using the closure relations between the attracting sets of $\C^*$-fixed points in $\Hilb_K^\nu$, one can also define a geometric ordering $\preceq^{\mathsf{geo}}$ on $\mathcal{P}_{\nu}(nl+|\nu|)$. By construction, the isomorphism $\mathcal{M}_{2\theta}(n\delta) \xrightarrow{\sim} \Hilb_K^\nu$ intertwines the two geometric orderings. Hence, by Corollary \ref{caracal2}, $\underline{\mu} \preceq_{w}^{\mathsf{geo}} \underline{\lambda} \iff \tau_{w}(\underline{\mu}^t)  \preceq^{\mathsf{geo}} \tau_{w}(\underline{\lambda}^t)$. But, by \cite[(4.13)]{Nak-Jack}, the ordering $\preceq^{\mathsf{geo}}$ is refined by the anti-dominance ordering on $\mathcal{P}_{\nu}(nl+|\nu|)$. 
\end{proof}

\end{document}